\documentclass[aop]{imsart}

\RequirePackage{amsthm,amsmath,amsfonts,amssymb,mathrsfs}

\RequirePackage[numbers]{natbib}
\RequirePackage[colorlinks,citecolor=blue,urlcolor=blue]{hyperref}
\RequirePackage{graphicx}
\usepackage[inline]{enumitem} 
\usepackage[usenames,dvipsnames]{xcolor}
\usepackage[percent]{overpic}
\usepackage{comment}
\usepackage{stmaryrd}
\usepackage{algorithm}
\usepackage{algorithmic}
\usepackage{pgfplots}
\usepackage{tikz}

\usepackage{hyperref}
\usepackage{graphicx,color,colortbl}
\usepackage{upgreek}
\usepackage[mathscr]{euscript}
\usepackage{hhline}
\usepackage{mathbbol,mathtools,calc}
\usepackage{bm}
\usepackage{framed,comment}
\usepackage{musicography}
\usepackage{tikz-cd}

\usepackage{blkarray}

\DeclareMathAlphabet{\mathpzc}{OT1}{pzc}{m}{it}
\DeclareMathAlphabet{\mathdutchcal}{U}{dutchcal}{m}{n}
\SetMathAlphabet{\mathdutchcal}{bold}{U}{dutchcal}{b}{n}

\newcommand{\matteo}[1]{{\color{brown} #1}}

\renewcommand{\Pr}{\mathbb{P}}
\newcommand{\Ex}{\mathbb{E}}
\newcommand{\Con}{C}
\newcommand{\mc}{\mathfrak{M}}
\renewcommand{\d}{\diff}
\newcommand{\logq}{\eta_q}

\allowdisplaybreaks
\numberwithin{equation}{section}
\usepackage{ytableau}
\usepackage{mathdots}
\usepackage{tikz}
\usetikzlibrary{
	shapes,
	arrows,
	positioning,
	decorations.markings,
	decorations.pathmorphing,
	circuits.logic.US,
	circuits.logic.IEC,
	fit,
	calc,
	plotmarks,
	matrix
}

\tikzset{
>=stealth',
help lines/.style={dashed, thick},
axis/.style={<->},
important line/.style={thick},
connection/.style={thick, dotted},
punkt/.style={
rectangle,
rounded corners,
draw=black, thick,
text width=4.5em,
minimum height=2em,
text centered,
},
pil/.style={
->,
thick,
gray,
shorten <=2pt,
shorten >=2pt,}
}

\usepackage{array}
\usepackage{adjustbox}
\usepackage{cleveref}
\usepackage{enumitem}

\usetikzlibrary{calc}
\usetikzlibrary{arrows.meta,backgrounds}
\usetikzlibrary{arrows}
\usetikzlibrary{decorations.pathmorphing}
\usepackage[multiple]{footmisc}
\DeclareMathAlphabet{\mathpzc}{OT1}{pzc}{m}{it}
\usepackage{bbm}

\startlocaldefs
\numberwithin{equation}{section}
\theoremstyle{plain}
\newtheorem{theorem}{Theorem}[section]
\newtheorem{corollary}[theorem]{Corollary}
\newtheorem{proposition}[theorem]{Proposition}
\newtheorem{lemma}[theorem]{Lemma}
\theoremstyle{remark}
\newtheorem{definition}[theorem]{Definition}

\newtheorem{remark}[theorem]{Remark}

\newcommand{\R}{\mathbb{R}}
\renewcommand{\hat}{\widehat}

\renewcommand{\bar}{\overline}

\renewcommand{\Pr}{\mathbb{P}}

\newcommand{\e}{\varepsilon}

\newcommand{\ind}{\mathbf{1}}
\newcommand{\N}{\mathbb{N}}
\newcommand{\h}{\mathfrak{h}}

\newcommand{\agamma}{\tau}

\renewcommand{\i}{\infty}

\newcommand{\arccosh}{\mathrm{arccosh}}

\newcommand{\Z}{\mathbb{Z}}
\newcommand*\diff{\mathop{}\!\mathrm{d}}

\newcommand{\norm}[1]{\Vert#1\Vert}
\newcommand{\tr}{\operatorname{tr}}

\graphicspath{{./figs/}} 

\endlocaldefs

\begin{document}

\begin{frontmatter}
\title{Large deviations for the $\lowercase{q}$-deformed polynuclear growth}
\runtitle{Large deviations for $\lowercase{q}$-PNG}

\begin{aug}
\author[A]{\fnms{Sayan}~\snm{Das}\ead[label=e1]{sayan.das@columbia.edu}\orcid{0000-0002-8098-9992}},
\author[B]{\fnms{Yuchen}~\snm{Liao}\ead[label=e2]{ ycliao@ustc.edu.cn}\orcid{0000-0003-4972-3376}}
\and
\author[C]{\fnms{Matteo}~\snm{Mucciconi}\ead[label=e3]{matteomucciconi@gmail.com}\orcid{0000-0003-2349-7614}}

\address[A]{Department of Mathematics,
Columbia University\printead[presep={,\ }]{e1}}

\address[B]{School of Mathematical Sciences, University of Science and Technology of China\printead[presep={,\ }]{e2}}

\address[C]{Department of Mathematics, National University of Singapore\printead[presep={,\ }]{e3}}
\end{aug}

\begin{abstract}
    In this paper, we study large time large deviations for the height function $\h(x,t)$ of the $q$-deformed polynuclear growth introduced in \cite{aggarwal_borodin_wheeler_tPNG}. 
    We show that the upper-tail deviations have speed $t$ and derive an explicit formula for the rate function $\Phi_+(\mu)$. On the other hand, we show that the lower-tail deviations have speed $t^2$ and express the corresponding rate function $\Phi_-(\mu)$ in terms of a variational problem. 

    Our analysis relies on distributional identities between the height function $\mathfrak{h}$ and two important measures on the set of integer partitions: the Poissonized Plancherel measure and the cylindric Plancherel measure. Following a scheme developed in \cite{dt21} we analyze a Fredholm determinant representation for the $q$-Laplace transform of $\h(x,t)$, from which we extract exact Lyapunov exponents and through inversion the upper-tail rate function $\Phi_+$. The proof of the lower-tail Large Deviation Principle is more subtle and requires several novel ideas which combine classical asymptotic results for the Plancherel measure and log-concavity properties of Schur polynomials. The techniques we develop to characterize the lower-tail are rather flexible and have the potential to generalize to other solvable growth models.
\end{abstract}



\end{frontmatter}
\setcounter{tocdepth}{1}
\tableofcontents


\section{Introduction}

\subsection{The Model and Main Results}

In this paper, we consider the $q$-Deformed Polynuclear Growth, or $q$-PNG in short, which is a growth process introduced rather recently \cite{aggarwal_borodin_wheeler_tPNG}. The $q$-PNG is a \emph{solvable} one-parameter deformation of the more well-studied Polynuclear Growth model (PNG), famously analyzed by Pr{\"a}hofer and Spohn \cite{PhahoferSpohn2002} to characterize universal processes of growing interfaces (i.e.~the Airy$_2$ process) and also more recently by Mateski-Quastel-Remenik \cite{matetski2022polynuclear} for its connection to integrable systems. The literature on the PNG is somewhat broad, given its well understood connections with Ulam's problem, determinantal point processes, free fermions, Bethe Ansatz, and last passage percolation to name a few 
(see \cite{hammersley1972few,aldous1995hammersley,seppalainen_burgers_1996,seppalainen_98_increasing,baik1999distribution,johansson2003discrete,ferrari2004polynuclear,prahofer2004exact,imamura_sasamoto_half_space_png,imamura2005polynuclear,dauvergne2021scaling,johansson_rahman_multitime,johansson_rahman_inhomogeneous}).
On the other hand, the solvability structures of the $q$-deformation of the PNG we consider here are only recently starting to emerge: this paper makes advancements in this direction, while in parallel establishing its probabilistic properties.

\subsubsection{The model} 
In the $q$-PNG the height function $\mathfrak{h}(x,t)$ is a piecewise constant function taking integer values and with jump discontinuities of unit size. Such height function can be seen as the profile created stacking on top of each other islands one unit size thick. As time moves forward, islands expand laterally at a constant speed, which we assume to be unitary. This lateral expansion might lead two islands to collide: in this situation, the two colliding islands merge and with probability $q\in(0,1)$, on top of the contact point a new island of infinitesimal width gets created. Moreover, new islands of infinitesimal width randomly appear on the surface and these additional ``nucleation" events follow a Poisson point process in space and time of intensity $\Lambda >0$. A snapshot of the evolution of the height function $\mathfrak{h}$ is given in \Cref{fig:qPNG}. A special choice of the initial condition, which is the one we focus on in this paper, is given setting 
\begin{equation*}
    \mathfrak{h}(x,t=0) = 
    \begin{cases}
        0 \qquad & \text{if } x=0,
        \\
        -\infty & \text{if } x \neq 0.
    \end{cases}
\end{equation*}
This is referred to as \emph{droplet initial condition}. 

\begin{figure}
    \centering
    \includegraphics{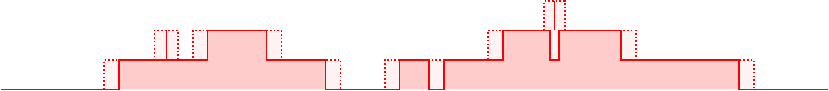}
    \caption{Illustration of the dynamics of the $q$-PNG. The solid line denotes the height profile at a time $t$, while the dotted broken line denotes the height functions at a slightly later time $t'$. Thin vertical lines denote nucleations that took place in the time interval $(t,t')$.}
    \label{fig:qPNG}
\end{figure}

In place of this rather informal description of the process, we take a different perspective and draw in the $(x,t)$ plane the trajectories of up and down unit jumps of the height $\mathfrak{h}$.

\begin{definition}\label{def:qpng}
Consider the half plane $\mathbb{R}\times\mathbb{R}_+$ and on the cone $\mathcal{C}=\{ (x,t) : |x|\le t\}$ sample a Poisson process $\mathcal{P} = \{ (x_i,t_i) \}_{i \ge 0}$ of intensity $\Lambda>0$. {From every point $p\in \mathcal{P}$ emanates} two rays, one directed north-eastward and the other directed north-westward respectively with angles $45^\circ$ and $135^\circ$. Whenever two rays emanating from different vertices intersect, with probability $1-q$ they terminate at the intersection point, and with probability $q$ they cross each other and continue along their trajectory. To any point $(x,t)\in \mathcal{C}$ we associate the height $\mathfrak{h}(x,t)$ as the number of rays that intersect with a straight segment from $(x,t)$ to $(0,0)$. By agreement we set $\mathfrak{h}(x,t)=-\infty$ if $(x,t) \notin \mathcal{{C}}$. This procedure is depicted in \Cref{fig:qPNG_lines}.
\end{definition}

\begin{figure}[h]
    \centering
    \includegraphics{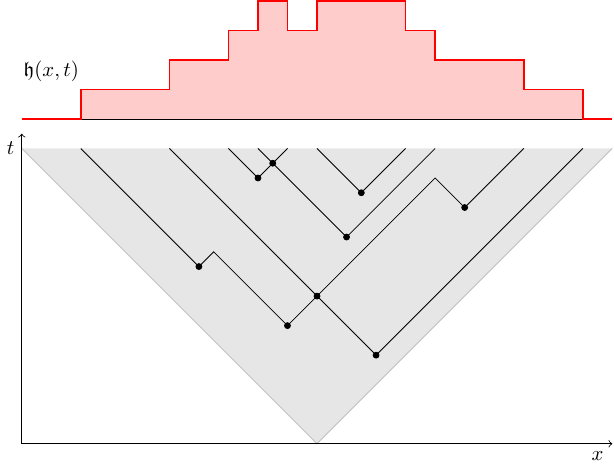}
    \caption{Above we see a snapshot of the height $\mathfrak{h}(x,t)$ of the $q$-PNG. Below we see a representation in space-time coordinates of a possible evolution leading to the profile $\mathfrak{h}$. }
    \label{fig:qPNG_lines}
\end{figure}

\subsubsection{Main results}
Recently several properties of the $q$-PNG with droplet initial condition have been established. In the original paper \cite{aggarwal_borodin_wheeler_tPNG}, Aggarwal, Borodin, and Wheeler characterized the large time fluctuations of the height around the expected value, which obey the Tracy-Widom GUE law \cite{tracy1993level} as
\begin{equation} \label{eq:Tracy_Widom}
    \lim_{t^2-x^2\to +\infty} \mathbb{P} \left( \frac{\mathfrak{h}(x,t) - v_{\Lambda,q} \, (t^2-x^2)^{1/2}}{ \sigma_{\Lambda,q}\, (t^2-x^2)^{1/6}} \le s \right) = F_{\operatorname{GUE}}(s),
\end{equation}
for $v_{\Lambda,q} = \Lambda/(1-q)$ and $\sigma_{\Lambda,q} = \sqrt[3]{\Lambda/2(1-q) }$. In \cite{aggarwal_borodin_wheeler_tPNG}, the authors also showed that, under a proper scaling in space and time, the height function $\mathfrak{h}$ converges, as $q\to 1$, to the solution of the KPZ equation \cite{KPZ1986} in the sense of 1-point distribution. Later, Drillick and Lin \cite{Drillick_Lin_t_PNG} proved a strong law of large numbers for $\h$ by constructing a colored version of the model\footnote{Both in \cite{aggarwal_borodin_wheeler_tPNG} and \cite{Drillick_Lin_t_PNG} authors denote the deformation parameter by $t$, while we use the letter $q$.}.

In this paper, we consider another important question, from the probabilistic standpoint, which is that of characterizing the probability of rare events for the growth of the height function. While restricting ourselves to the particular case of droplet initial condition, we focus on two particular, although very natural, questions: determining the decay of large deviation events where the height function assumes values, larger or smaller (of order $t$) respectively, than those prescribed by the law of large numbers. Intuitively, upper- and lower-tail events \emph{should} possess different decay rates. In fact, for the height $\mathfrak{h}(x,t)$ function to assume values larger than the expected value it is sufficient to require that nucleation events at location $x$ take place at an unusually high rate along the time window $[0,t]$: this suggests that $\mathbb{P}(\mathfrak{h}(x,t) > (v_{\Lambda,q}+\xi)t ) =e^{-O(t)}$. On the other hand for the height function $\mathfrak{h}(x,t)$ to assume values smaller, of order $t$, than the expected value, we need to require that through the entire backward light cone $\{ (y,s) : |x-y| < t-s \}$ nucleations have taken place with unusually slow rate, suggesting that $\mathbb{P}(\mathfrak{h}(x,t) < (v_{\Lambda,q}-\xi)t ) = e^{-O(t^2)}$, where $t^2$ signifies the area of the backward light cone. The following theorems confirm these qualitative heuristics, quantifying explicitly the decay rate functions.

{For convenience, we shall work with intensity $\Lambda=2(1-q)$ so that the law of large numbers $v_{\Lambda,q}=2$. In this case, upper tail will correspond to $\{\h(0,t)\ge \mu t\}$ events for $\mu\ge 2$, whereas lower tail will correspond to $\{\h(0,t)\le \mu t\}$ events for $\mu\in [0,2]$.}

\begin{theorem}[Upper-tail Large Deviation Principle]\label{thm:upper_tail}
    Let $\mathfrak{h}$ be the height function of the $q$-PNG with intensity $\Lambda=2(1-q)$ and droplet initial condition. Then, for $\mu \ge 2$ we have
    \begin{equation}
        -\lim_{t\to \infty} \frac1t \log \Pr(\h(0,t) \ge \mu t) =  \Phi_+(\mu),
    \end{equation}
    where the rate function is 
    \begin{equation} \label{eq:Phi+}
        \Phi_+(\mu) = 2\mu \, \arccosh(\mu/2)-2\sqrt{\mu^2-4}.
    \end{equation}
\end{theorem}


\begin{theorem}[Lower-tail Large Deviation Principle]\label{thm:lower_tail_intro}
    Let $\mathfrak{h}$ be the height function of the $q$-PNG with intensity $\Lambda=2(1-q)$ and droplet initial condition. Then, for $\mu \in [0,2]$, we have
    \begin{equation} \label{eq:LDP_h_lower_tail}
        -\lim_{t \to \infty} \frac{1}{t^2} \log \mathbb{P} \left( \mathfrak{h}(0,t) \le \mu t \right) = \Phi_-(\mu),
    \end{equation}
     where the rate function $\Phi_-: \R \to \R\cup \{+\infty\}$ is a decreasing, convex, continuous on $[0,\infty)$ function with $\Phi_-(\mu)=+\infty$ for $\mu<0$, $\Phi_-(0)=1-q$, and $\Phi_-(\mu)=0$ for $\mu\ge 2$. Furthermore, it 
     possesses the following expression,
    \begin{equation} \label{eq:Phi_deconvolution}
        \Phi_-(\mu) = \sup_{y\in \mathbb{R}} \left\{ \mathcal{F}(y) - \frac{\log q^{-1}}{2} (\mu-y)^2 \right\}.
    \end{equation}
    where the function $\mathcal{F}$ is described in \Cref{sec:phi-}.
\end{theorem}

\begin{remark} Few remarks related to the above theorems are in order.

\begin{enumerate}[leftmargin=18pt]
    \item The intensity $\Lambda$ of the $q$-PNG is assumed to be $2(1-q)$ for convenience. Under this intensity, we have the law of large numbers $\lim_{t\to \infty} \h(0,t)/t=2$, which is free of $q$. The Large Deviation Principle (LDP) for the general intensity case can be obtained with minor modifications in our arguments.
    \item In \Cref{cor:xzero}, we show that for each $t>|x|>0$, we have the following equality in distribution (in the sense of one-point marginals):
    $$\h(x,t) \stackrel{d}{=}\h(0,\sqrt{t^2-x^2}).$$
     This allows us to derive LDPs for the height function of a $q$-PNG model at a general location from our main theorems.   
     \item Note that the upper-tail rate function $\Phi_+$ is free of $q$, and in fact,  matches with the upper-tail rate function for the $(q=0)$ PNG model obtained in \cite{seppalainen_98_increasing}. This matching of the upper-tail rate function of a non-determinantal model with its determinantal counterpart was observed in the context of the KPZ equation \cite{dt21} and ASEP \cite{dz1} as well. A geometric explanation for this matching can be given using the theory of Gibbsian line ensembles. We refer to the recent work of Ganguly and Hegde \cite{ganguly2022sharp} where this line ensemble approach was carried out successfully in obtaining sharp estimates on the upper-tail of the KPZ equation and other general models under certain assumptions.  The lower-tail rate function, however, depends on the $q$ parameter (see \Cref{sec:phi-}).
\end{enumerate} 
\end{remark}

\subsubsection{Description of $\mathcal{F}$}\label{sec:phi-}
The explicit characterization of the function $\mathcal{F}$ requires the introduction of several objects. Throughout the paper, we will use the notation
\begin{equation}\label{eq:logq}
    \logq:=\log q^{-1}.
\end{equation} 
We will call $h:\mathbb{R}\to \mathbb{R}$ to be  $K$-Lipschitz if $|h(x)-h(y)|\le K|x-y|$ for all $x,y\in \mathbb{R}$. We define the set
\begin{equation}\label{def:y1bar}
    \begin{aligned}
    \overline{\mathcal{Y}}_1 & := \{ h:\mathbb{R} \to \mathbb{R}_+: h \text{ is 2-Lipschitz, non-decreasing in $\mathbb{R}_{-}$}, \\ & \hspace{3.5cm}\text{ decreasing on $\mathbb{R}_{+}$ and } \|h\|_{L^1}=1\}
\end{aligned}
\end{equation}
and the functional
\begin{equation} \label{eq:W_intro}
    \mathcal{W}^{(q)}(\kappa,h;x) := 1+\kappa \log \kappa + \kappa \mathsf{J}_{\logq}(h;x/\sqrt{\kappa}), \qquad \text{for } \kappa > 0, h\in \overline{\mathcal{Y}}_1, x\in \mathbb{R},
\end{equation}
where
\begin{equation}
\begin{split}
    \mathsf{J}_{\eta}(h;y) :=& - \frac{1}{2} + \eta \frac{[-y]_+^2}{2} + \frac{1}{2} \| \overline{\phi}_{\mathrm{VKLS}} - h \|_1^2
    \\
    & + 2 \int_{\mathbb{R}} h(\xi) \left( \mathbf{1}_{[\sqrt{2},+\infty)} (|\xi|) \arccosh\left| \frac{\xi}{\sqrt{2}} \right| +  \frac{\eta}{2} \mathbf{1}_{[\frac{y}{\sqrt{2}},+\infty)}(\xi)  \right)\diff \xi.
\end{split}
\end{equation}
Above the norm $\| \, \cdot \, \|_1$ is the Sobolev $H^{1/2}$ norm defined in \Cref{def:sob}, the function $\overline{\phi}_{\mathrm{VKLS}}\in \overline{\mathcal{Y}}_1$ is the Vershik-Kerov-Logan-Shepp optimal shape, explicitly given in \Cref{eq:VKLS_shape} and $[a]_+:=\max\{a,0\}$. We also define the minimizer of the functional $\mathcal{W}^{(q)}$ for any fixed $x$ as
\begin{equation}\label{def:f_intro}
    \mathcal{F}(x) := \inf_{\kappa > 0} \inf_{h \in \overline{\mathcal{Y}}_1}  \mathcal{W}^{(q)}(\kappa,h;x).
\end{equation}
The following theorem describes some of the main properties of $\mathcal{F}$ we prove in this paper. 

\begin{theorem}\label{prop:fall}
    Fix $q\in(0,1)$. The function $\mathcal{F}(x)$ is decreasing, non-negative, convex, and continuously differentiable with derivative $\mathcal{F}'(x)$ being $\logq$-Lipschitz. Moreover, there exists $x_q\in \R$ such that
    \begin{equation} \label{eq:F_parabola_intro}
        \mathcal{F}(x) = (1-q) + \logq \tfrac{x^2}{2} \quad\mbox{ for all } \ \ x\le x_q.
    \end{equation}
    Furthermore, we have
    \begin{equation}
         \label{eq:F_convolution}
        \mathcal{F}(x) = \inf_{y\in \mathbb{R}} \left\{ \Phi_-(y) + \frac{\logq}{2} (x-y)^2 \right\}.
    \end{equation}
\end{theorem}
\begin{remark}
    The functional $\mathcal{W}^{(q)}$ has an alternative representation reported in \eqref{eq:J_potential}.
\end{remark}

\begin{remark}
    The function $\mathcal{F}$ is the \emph{Moreau envelope} of the lower-tail rate function $\Phi_-$ with parameter $1/\logq$. The Moreau envelope of a function $f$ with parameter $\lambda$ is defined as the infimal convolution of $f$ with $\frac{x^2}{2\lambda}$ \cite{moreau}. It has applications in convex and variational analysis (see \cite{attouch,rock}) and can be understood as a general mechanism to \textit{smoothen} convex functions.
\end{remark}

\begin{remark}
The function $\mathcal{F}$ is, in fact, the lower tail rate function for the $q$-PNG height function with a random shift; see \eqref{eq.lwtailint}. For this reason, we will often refer to $\mathcal{F}$ as a lower-tail rate function as well.    In \Cref{subs:eq_diff} we conjecture an explicit formula for the function $\mathcal{F}$, given in terms of a solution of a certain non-linear differential equation.
\end{remark}

\subsection{Proof Ideas} We now describe the key ideas behind the proofs of our main theorems.  

\subsubsection{Connections between different solvable models}\label{sec:1.2.1} Our proof is facilitated by the connections between $q$-PNG and two important measures on the set of integer partitions: the Poissonized Plancherel measure and the cylindric Plancherel measure.

\begin{figure}[t]
    \begin{tikzpicture}[auto,
 		block_main/.style = {rectangle, draw=black, fill=white!95!black, text width=15em, text centered, minimum height=4.5em,font=\small},
 		block_KPZ/.style = {rectangle, draw=black, fill=white, text width=10em, text centered, minimum height=3.5em},
 		line/.style ={draw, thick, -latex', shorten >=0pt}]

 		\node [block_KPZ] (qpng) at (-3,-3) {$q$-PNG};
 		\node [block_KPZ] (cp) at (2,-3) {Cylindric Plancherel};
 		\node [block_KPZ] (det) at (7,-3) {Poissonized Plancherel};

   
 		\begin{scope}[every path/.style=line]
             \path (qpng)  -- (cp) node[midway,above] {$\chi$-shift};
             \path (cp)  -- (det) node[midway,above] {$S_{\zeta}$-shift};
 		\end{scope}
 	\end{tikzpicture}
 	\caption{Connections between models}
 	\label{fig:conn}
\end{figure}
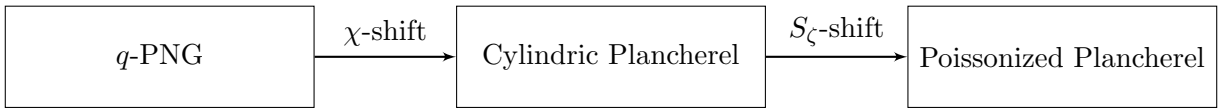
We observe that the height function $\h$ after a random shift by $\chi$, a $q$-geometric random variable (see \eqref{eq:qgeo}), is equal in distribution to the largest row of a random partition distributed according to a cylindric Plancherel measure. This distributional identity is proven in \Cref{thm:matching_qPNG_cylindric_plancherel} and essentially comes from a more general result established first in \cite{IMS_matching}.

The cylindric Plancherel measure and the more general Periodic Schur measure were introduced in the seminal paper by Borodin \cite{borodin2007periodic}. A remarkable property of these measures, as observed by Borodin, is that upon a random shift by $S_{\zeta}$, a $\mathrm{Theta}(q,\zeta)$-distributed random variable (see \eqref{eq:S_distribution}), they become determinantal point processes (\Cref{fig:conn}). The explicit correlation kernel $\mathsf{K}$ (defined in \eqref{skernel}) for the $S_\zeta$-shifted cylindric Plancherel measure allows us to derive Fredholm determinant formulas for the one-point probability distribution of the height function $\h$ after the combined $(\chi+S_\zeta)$-shift:
\begin{equation}
    \label{eq:rel_intro}
    \begin{aligned}
   \mathbb{P}(\h(0,t)+\chi+S_\zeta \le s) & =\mathbb{P}_{\mathsf{cPlan}(t(1-q))}(\lambda_1+S_\zeta \le s) \\ & =\det \left( 1 - \mathsf{K}_{\zeta,t(1-q)} \right)_{\ell^2(s+\frac{1}{2}, s+\frac{3}{2},\dots)}\!.\! 
\end{aligned}
\end{equation}
On the other hand, the one-point probability distribution of $(\chi+S_\zeta)$-shift of the height function $\mathfrak{h}$ is also equal to the expectation of a certain multiplicative functional of the Poissonized Plancherel measure, derived in \cite{aggarwal_borodin_wheeler_tPNG} as a special case of a more general result by Borodin \cite{borodin2016stochastic_MM} and also follows from the results in \cite{IMS_KPZ_free_fermions} (see \Cref{cor:iden}). This leads to a second formula for the probability distribution of $(\chi+S_\zeta)$-shifted height function $\h$:
{
\begin{equation}
\label{eq:rel_intro2}
   \mathbb{P}(\h(0,t)+\chi+S_\zeta \le s) =\mathbb{P}_{\mathsf{cPlan}(t(1-q))}(\lambda_1+S_\zeta \le s) =\mathbb{E}_{\mathsf{Plan}(t)} \bigg[\prod_{i \ge 1} \frac{1}{1+\zeta q^{s+i-\lambda_i}} \bigg]\!.\! 
\end{equation}
}
These formulas form the starting point of our analysis towards LDP results for $q$-PNG.

\subsubsection{Upper-Tail}  \label{sec:1.2.2}
In this subsection, we present a brief sketch of the proof for the upper-tail. The main component of our proof is the Fredholm determinant formula from \eqref{eq:rel_intro}.  Recall that a Fredholm determinant can be defined as a series:
\begin{align}
    \label{eq:fred_intro}
    \det(I-\mathsf{K}_{\zeta,t(1-q)})=1-\tr(\mathsf{K}_{\zeta,t(1-q)})+\sum_{L=2}^{\infty} (-1)^L\tr(\mathsf{K}_{\zeta,t(1-q)}^{\wedge L}),
\end{align}
where the notation $\mathsf{K}_{\zeta,t(1-q)}^{\wedge L}$ comes from the exterior algebra definition (see \Cref{sec.ho}). We shall collectively call the $L\ge 2$ series as ``higher-order term''. In the upper-tail regime, we expect the Fredholm determinant to behave perturbatively, i.e., leading order of $1-\det(I-\mathsf{K}_{\zeta,t(1-q)})$ is given by the trace term $\tr(\mathsf{K}_{\zeta,t(1-q)})$. Indeed, direct analysis of the trace yields
$$-\lim_{t\to\infty}\frac1t\log\Pr(\h(0,t)+\chi+S_{\zeta} \ge \mu t)
=\Psi_+(\mu):= \begin{dcases}
    \Phi_+(\mu) & \mu \in [2,q^{\frac12}+q^{-\frac12}] \\
    \mu \logq+2q^{-\frac12}-2q^{\frac12} & \mu \in [q^{\frac12}+q^{-\frac12},\infty)
\end{dcases}$$
where $\Phi_+$ is defined in \eqref{eq:Phi+}. However, it is not straightforward to extract the upper-tail rate function from the above result. From the precise tail behavior of $\chi+S_{\zeta}$  (notice that $\chi$ has an exponential right tail and no left tail, while $S_\zeta$ has Gaussian left and right tails), the previous limit suggests that, \emph{assuming that the upper-tail rate function $\widetilde{\Phi}_+$ of $\mathfrak{h}$ exists}, $\widetilde{\Phi}_+$ can be found as a solution of the relation 
\begin{equation} \label{eq:infimal_convolution_upper_tail_intro}
    \Psi_+(\mu) = \inf_{p\in \R} \left\{ \widetilde{\Phi}_+(\mu-p)+\mathcal{L}(p) \right\},
\end{equation}
where $\mathcal{L}(p):=\max\{p,0\}\cdot \logq$. {By \cite[p.57]{borwein2006convex}, taking Legendre transform of both sides we find that $\Psi_+^*=\widetilde{\Phi}_+^*+\mathcal{L}^*$. Since $\Psi_+$ and $\mathcal{L}$ are linear with slope $\logq$ after a certain point, we have $\Psi_+^*(x)=\mathcal{L}^*(x)=\infty$ for all $x> \logq$. Thus for all $x>\logq$, any $\widetilde{\Phi}_+^*(x) \in (-\infty,\infty]$ satisfies $\Psi_+^*(x)=\widetilde{\Phi}_+^*(x)+\mathcal{L}^*(x)$. This indicates that the above deconvolution problem does not have a unique solution. Indeed, there are \textit{infinitely many proper convex closed functions} including $\Phi_+$ defined in \eqref{eq:Phi+} that all satisfy the above deconvolution problem. This  suggests that direct analysis of the Fredholm determinant would not produce the exact rate function.}

To bypass the above problem, we utilize the Fredholm determinant in a different way following the strategy of \cite{dt21,dz1}. First, using the explicit distribution formula for $\chi+S_\zeta$, one can check that $\Pr(\h(0,t)+\chi+S_{\zeta/\sqrt{q}} \le 0)$ is equal to a certain $q$-Laplace transform of $\h(0,t)$, defined as $\Ex[\mathsf{F}_q(\zeta q^{-\h(0,t)})]$ where $\mathsf{F_q}(\zeta):=\prod_{k\ge 0} (1+\zeta q^k)^{-1}$. This $q$-Laplace transform may then be inverted in the following way to extract a formula for the moment-generating function of $\h(0,t)$. 
\begin{equation}\label{eq:fubini_intro}
    \Ex[e^{p\h(0,t)}]  =\frac{(-1)^n\int_0^{\infty} \zeta^{-\alpha}\frac{\diff^n}{\diff \zeta^n}\Ex[\mathsf{F}_q(\zeta q^{-\h(0,t)})] \diff \zeta}{(-1)^n\int_0^{\infty} \zeta^{-\alpha}\mathsf{F}_q^{(n)}(\zeta) \diff \zeta}, \mbox{ for }p>0,
\end{equation}
\noindent where $n:=\lfloor p\logq\rfloor+1$, and $\alpha:= p\logq-\lfloor p\logq\rfloor$. The $(-1)^n$ factor above ensures both the numerator and the denominator are positive. The above formula follows via Fubini's theorem and properties of $\mathsf{F}_q$ \cite[Lemma 1.8, Proposition 2.2]{dz1}. Armed with the Fredholm determinant formula from \eqref{eq:rel_intro}, we utilize the above relation to compute an exact expression for \textit{$p$-th Lyapunov exponent} for $\h(0,t)$, which is the limit of  $\log\Ex[e^{p\h(0,t)}]$ scaled by $t$:
\begin{align}
    \label{eq:lya_intro}
    \lim_{t\to \infty}\frac1t\log\Ex\big[e^{p\h(0,t)}\big]=4\sinh(p/2).
\end{align}
The upper-tail rate function can then be computed from the Lyapunov exponent by a standard Legendre-Fenchel transform technique.

Let us now mention briefly how we derive \eqref{eq:lya_intro} from \eqref{eq:fubini_intro}. We focus on the right-hand side of \eqref{eq:fubini_intro}. The denominator does not depend on $t$ and vanishes in the $-\frac1t\log$ limit. For the numerator, we wish to plug in the Fredholm determinant formula from \eqref{eq:fred_intro} (with $\zeta\mapsto \zeta/\sqrt{q}$) and analyze the derivatives of the Fredholm determinant series. However,  direct analysis of these derivatives is still quite delicate as the derivatives of the higher-order term (which also has a similar series representation) have an oscillatory behavior for large values of $\zeta$.  To circumvent this issue, we first split the range of integral in the numerator into two parts: $[0,\agamma]$ and $[\agamma,\infty)$ based on a carefully chosen $\agamma$ (see \eqref{def:gamma}). Using the decay properties of $\mathsf{F}_q$ and the precise value of $\agamma$,  the integral on the range $[\agamma,\infty)$ can easily be shown to be subdominant. For the $[0,\agamma]$ range integral, we now feed in the Fredholm determinant formula and write the integral as
a sum of the following two terms:
$$(-1)^{n+1}\hspace{-0.1cm}\int_0^{\agamma}\hspace{-0.1cm} \zeta^{-\alpha}\frac{\diff^n}{\diff\zeta^n}\tr(\mathsf{K}_{\zeta/\sqrt{q},t(1-q)})\diff \zeta, \ \ (-1)^{n}\hspace{-0.1cm}\int_0^{\agamma}\hspace{-0.1cm} \zeta^{-\alpha}\frac{\diff^n}{\diff\zeta^n}\sum_{L=2}^{\infty} (-1)^L\tr(\mathsf{K}_{\zeta/\sqrt{q},t(1-q)}^{\wedge L}) \diff \zeta.$$
Relying on the explicit expression of the kernel which involves Bessel functions, we develop precise estimates for the traces of the kernel and its derivatives in \Cref{sec:upper_tail}. This allows us to show that the first term above yields the correct Lyapunov exponent whereas the second term is subdominant.

\subsubsection{Lower-Tail}


For the characterization of lower-tail probabilities, rather than the explicit Fredholm determinant expression for the law of the height $\mathfrak{h}$, we use its connections with Poissonized
Plancherel and cylindric Plancherel measures.
We start with analyzing the multiplicative functional formula from \eqref{eq:rel_intro2}. 
The fact that it is worthwhile to probe into multiplicative functional formulas to derive lower-tail asymptotics for KPZ models was first noted by Corwin and Ghosal \cite{lwtail}, who obtained sharp lower-tail estimates for the KPZ equation.
But, as we will see below, in our case, the analysis of the right-hand side of \eqref{eq:rel_intro2} alone only gives us partial information. For a complete characterization of the lower-tail LDP, we will need new ideas based on the distributional equality between the law of the height $\mathfrak{h}$ and the first row of the cylindric Plancherel measure. 

Utilizing the precise form of the functional in \eqref{eq:rel_intro2} and leveraging regularity properties of the Poissonized Plancherel measure allow us to
to establish the Large Deviation Principle
\begin{equation}\label{eq.lwtailint}
    -\lim_{t\to +\infty} \tfrac{1}{t^2} \log \mathbb{P}(\h(0,t) + \chi + S_1 \le xt) = \mathcal{F}(x),
\end{equation}
where $\mathcal{F}(x)$ is defined in \eqref{def:f_intro}.
Observing the structure of the functional $\mathcal{W}^{(q)}$ defined in \eqref{eq:W_intro} we see that the $q$-dependent term $\frac{\logq}{2}[-y]_+^2 + \logq \int_{y/\sqrt{2}}^{+\infty}h(\xi) \diff \xi$ originates from the limit of the product $\prod_{i\ge 1}(1+q^{s+i-\lambda_i})^{-1}$ appearing in the right-hand side of \eqref{eq:rel_intro2}, while the remaining terms come from a Poissonization of the Vershik-Kerov limit of the Plancherel measure \cite{VershikKerov_LimShape1077}. From the precise tail behavior of $\chi+S_{1}$, the previous limit suggests that, \emph{(a) assuming that the lower-tail rate function $\widetilde{\Phi}_-$ of $\mathfrak{h}$ exists}, $\widetilde{\Phi}_-$ can be found as a solution of
\begin{equation} \label{eq:infimal_convolution_lower_tail_intro}
    \mathcal{F}(x) = \min_{y\in \mathbb{R}} \left\{ \widetilde{\Phi}_-(x-y) + \logq \frac{y^2}{2} \right\}.
\end{equation}
Unlike the case of the upper-tail discussed around \eqref{eq:infimal_convolution_upper_tail_intro}, because of the presence of the quadratic term $\logq y^2/2$, the problem \eqref{eq:infimal_convolution_lower_tail_intro}, admits a unique solution, \textit{(b) assuming that $\mathcal{F}$, $\widetilde{\Phi}_-$ are proper convex and closed}. The solution then is expressed \cite{hiriart1994deconvolution} as
\begin{equation*}
    \widetilde{\Phi}_-(\mu) = (\mathcal{F}^*-g^*)^*(\mu) = \sup_{y\in \mathbb{R}} \left\{ \mathcal{F}(y) - \frac{\eta_q}{2}(\mu-y)^2 \right\}.
\end{equation*}
Here, again, the superscript $^*$ denotes the Legendre transform. 

To rigorously support the aforementioned argument, it is necessary to establish the two assumptions (a) and (b) made earlier. This is where we rely on several innovative ideas originating from the asymptotic results for the Poissonized Plancherel measure and log-concavity properties of Schur polynomials.
We first look for regularity properties of the rate function $\mathcal{F}$. These regularity properties will allow us to first identify uniquely the function $\widetilde{\Phi}_-$ and further to prove that the lower-tail large deviation function for $\h$ exists and it is given by $\widetilde{\Phi}_-$. To this end, we find the distributional equality between the height function $\mathfrak{h}$ and the cylindric Plancherel measure to be crucial. In fact, using log-concavity properties of the (skew) Schur polynomials proved by Lam-Postnikov-Pylyavyskyy \cite{Lam_Postnikov_Pylyavskyy_concavity} we are able to conclude, through a Laplace principle-type argument, that both $\mathcal{F}$ and (any possible) $\widetilde{\Phi}_-$ are necessarily convex. The final important ingredient to the proof of \Cref{prop:fall} and hence of \Cref{thm:lower_tail_intro} is the explicit characterization of the function $\mathcal{F}(x)$ as $x\to - \infty$. The parabolic behavior \eqref{eq:F_parabola_intro} of the function $\mathcal{F}(x)$ for $x$ negative sufficiently large follows from an exact minimization of the functional $\mathcal{W}^{(q)}$, owing to classical results of Vershik-Kerov \cite{VershikKerov_LimShape1077,VershikKerov_LimShape1985}, Logan-Shepp \cite{logan_shepp1977variational} along with further non-trivial inequalities we establish. Interestingly, this explicit knowledge of $\mathcal{F}$ for large negative $x$ can be leveraged to extract uniform equicontinuity bounds of the pre-limit lower-tail probabilities for the height function $\mathfrak{h}$; see \Cref{prop:equicont}. Then, a compactness argument, in the style of Arzel{\'a}-Ascoli theorem, allows us to conclude that the limit in the left-hand side of \eqref{eq:LDP_h_lower_tail} exists and hence, by its convexity and Lipschitz properties, is equal to the \emph{infimal deconvolution} between the function $\mathcal{F}$ and a parabola, as in \eqref{eq:Phi_deconvolution}.

One of the key novelties of this paper is the deconvolution scheme that we propose to establish lower tail Large Deviation Principle. In our forthcoming work \cite{dlm24} we adapt this scheme to the stochastic six vertex model and we envision them to be adaptable to ASEP as well, where similar algebraic structures are present. 
The fact that the function $\mathcal{F}$ is the Moreau envelope of a convex, closed function, i.e. $\Phi_-$ implies, thanks to results from convex analysis \cite{attouch}, that $\mathcal{F}$ is also continuously differentiable. This is remarkable because only little explicit knowledge of the function $\mathcal{F}$ is required to establish this sort of regularity. In the presence of further regularity of $\Phi_-$, namely it being $C^2$, we could find an explicit expression for $\mathcal{F}$ (which would become $C^2$ in a half line $[\mu_q,+\infty)$, from properties of the Moreau envelope), which would follow from solving a certain highly non-linear differential equation derived (under the assumption that $\mathcal{F} \in C^2$) in \Cref{subs:eq_diff}.

\subsection{Comparison with other large deviations results for growth processes} We now review some of the results and available techniques to solve LDP problems for KPZ models. 

\subsubsection{Upper-Tail LDP literature} Zero-temperature models, such as PNG, Totally Asymmetric Simple Exclusion Process (TASEP), and first and last passage percolation (LPP), have sub- or super-additive structures. The existence of the upper-tail rate function in these models can be deduced easily from a standard subadditive argument \cite{kesten}. Thus, it is the explicit form of the upper-tail rate function that is interesting in these models. For the PNG model, i.e., $q$-PNG with $q=0$, \cite{seppalainen_98_increasing} first computed the upper-tail rate function explicitly. Sepp{\"a}l{\"a}inen's proof was based on a coupling between the superadditive process with a suitable particle system that admits certain known stationary initial conditions. This general coupling approach to extract the upper-tail rate function was later applied in other models such as TASEP \cite{sepp98mprf}, LPP with inhomogeneous exponential weights \cite{emrah}, LPP in the Bernoulli environment \cite{ciech}, and Brownian LPP \cite{jan19}.  A different proof for the PNG upper-tail rate function was given by \cite{deuschel_zeitouni_1999} based on Young diagrams.

The LDP problem for TASEP (equivalently Exponential LPP) was also explored by Johansson in his seminal paper \cite{johansson2000shape}. He related TASEP to the largest eigenvalue for the Laguerre unitary ensemble, which is a particular example of continuous Coulomb gas that arises in random matrix theory. Leveraging potential theory tools, Johansson developed a comprehensive framework for obtaining upper- and lower-tail LDP results for both discrete and continuous Coulomb gases. This framework has been successfully extended to integrable discretizations of Coulomb gases \cite{das2022large}.  Over the past two decades, it has been discovered that some of the solvable models mentioned above possess determinantal structures. A direct perturbative analysis of the Fredholm determinants provides an alternative route to extract the one-point upper-tail rate function in those models (see \cite{eichelsbacher2016precise} for unitary invariant ensembles). In a more recent work, \cite{qt21} employed exact solvability to prove multi-point LDP for TASEP.

In the case of positive temperature models, the available techniques are limited. Nonetheless, in certain instances, the techniques used in zero-temperature models can be applied to specific positive temperature models that possess a rich structure. Two such models are log-gamma polymer \cite{Seppalainen2012} and O'Connell-Yor polymer \cite{o2001brownian}. Indeed, exploiting the fact that these polymer models can be coupled with their stationary counterparts, the approach in \cite{seppalainen_98_increasing} was successfully implemented to obtain the upper-tail rate function in these models \cite{georgiou2013large,janjigian2015large}. The Lyapunov exponent approach that we used here to obtain the upper-tail rate function for $q$-PNG was first implemented in the context of the KPZ equation \cite{dt21} with droplet initial condition (see \cite{le2016large,kamenev2016short} for earlier physics works, and see the rigorous work \cite{gl20} for general initial data). The scheme was also later utilized in extracting upper-tail rate functions for half-space KPZ \cite{lin2020kpz} and ASEP \cite{dz1} (see \cite{dps} for earlier work on one-sided  upper-tail estimates for ASEP). Lastly, we mention the recent paper by \cite{ganguly2022sharp} which utilizes the line ensemble framework to address the upper-tail problem for the KPZ equation and related zero-temperature models satisfying suitable hypotheses. In fact, their techniques also provide a one-sided multi-point upper-tail LDP for the KPZ equation (see also \cite{yl23} for a multi-point upper-tail LDP for the KPZ equation in a different regime). Although the approach presented in \cite{ganguly2022sharp} shows promise for application in other solvable models,  many of the inputs of their paper have not yet been established for those models.

\subsubsection{Lower-Tail LDP literature} 

The explicit study of lower-tail probability for solvable growth processes in the KPZ class is harder, from a methodological perspective than that of fluctuations or upper-tail large deviations. The reason for this is that explicit formulas such as Fredholm determinant representations for the probability distribution of the height function {have oscillatory behavior in this regime} and hence are hard to analyze in this setting.

The first explicit result concerning the evaluation of the lower-tail rate function for a growth process is found in the seminal paper by Logan and Shepp \cite{logan_shepp1977variational}, where authors, using potential theoretic arguments, derived an upper bound (which easily becomes a sharp equality \cite{deuschel_zeitouni_1999}) for the probability law of the longest increasing subsequence of a random permutation. The calculations presented in \cite{logan_shepp1977variational} (which in the paper are attributed to B.F. Logan) can be also used to derive explicitly the lower-tail of the PNG \cite{seppalainen_98_increasing}, through a simple Poissonization argument. Large deviations of the TASEP were proved in the seminal paper by Johansson \cite{johansson2000shape}, where lower-tail were expressed in terms of a variational problem. The connection between TASEP and the edge of Laguerre Unitary Ensemble offers another route to derive the explicit lower-tail rate function borrowing results of \cite{Vivo_Majumdar_Bohigas_2007} (see \cite{majumdar2014top} for detailed calculations for the case of Gaussian $\beta$-ensembles). {For first passage percolation, \cite{basu2021upper} proved a $t^2$ speed one-sided LDP using geometric techniques.}

More recently, considerable attention has been given to the characterization of rare events for the KPZ equation. The approach of using the multiplicative functional formulas to study lower-tail asymptotics was first done in \cite{lwtail} where the authors established sharp lower-tail estimates for the KPZ equation under the droplet initial condition.  The full lower-tail LDP was resolved in \cite{tsai_lower_tail},  proving conjectures from the physics literature \cite{sasorov2017large,corwin2018coulomb,krajenbrink2018systematic}.
The different routes used in these physics works along with Tsai's approach were later shown to be closely related in \cite{kraj19}. A Riemann-Hilbert approach presented later in \cite{cafasso_claeys_KPZ}, highlights similarities between the lower-tail problem for the KPZ equation and that of zero temperature models. A key step to the derivation of the rate function amounts to finding a $g$-function (in the jargon of the Riemann-Hilbert Problem), which solves a one-cut singular integral equation (i.e.~the support of $g$ is connected). Besides the above methods, the physics work \cite{Le_Doussal_KP_large_deviations} discusses another route to obtain LDP for the KPZ equation by using its connection to KP equation \cite{quastel_remenik_2022_KP}. 
For the O'Connell-Yor and log-gamma polymer models, recently Landon and Sosoe \cite{landon2023upper,landon2022tail} developed a systematic approach to obtaining lower tail estimates by combining exact formulas and coupling arguments (adapted from \cite{emrah2021optimal}) and geometric arguments (adapted from \cite{ganguly_hegde2023optimal}). While their results provide the correct cubic-order exponents in the moderate deviation regime, it is not clear how to extend them to the LDP regime.

{Before our work, the lower-tail LDP problem had been successfully solved only for the KPZ equation among the {non-determinantal} models.} Compared to the existing literature, our situation, from a technical standpoint, is considerably different. As defined in \eqref{def:f_intro} the
rate function $\mathcal{F}$ is found through a minimization of a quadratic functional: in this case, the explicit minimizer could be found as the solution of a \emph{multi-cut} Cauchy integral equation, which gives rise to complicated transcendental relations seemingly hard to manipulate (discussed in \Cref{subsub:logan_shepp_generalization}). Without an explicit characterization of the optimizer of the  functional $\mathcal{W}^{(q)}$, we find through the use of log-concavity properties of Schur polynomials, that the function $\mathcal{F}$ is strictly convex, which allows us to compute rigorously its deconvolution and to establish Large Deviation Principle, hence avoiding to rely on explicit knowledge of $\mathcal{F}$.

\subsection{Further Approaches} \label{sub:furapp} In general the presence of a rich mathematical structure around the study of the $q$-PNG invites to several alternative approaches, which we do not pursue in this paper. We discuss approaches through connections to dimer models and probabilistic-geometric methods from the theory of last passage percolation here. We collect further directions, which include connections to potential theory and Riemann-Hilbert method in \Cref{sec:further_approaches}.

\subsubsection{Approach through Morales-Pak-Tassy's extension of the Hook integral}

As explained in \Cref{sec:1.2.1}, we have used in our main arguments a relation between the height function of the $q$-PNG and the first row of a partition taken with cylindric Plancherel law. This correspondence opens the possibility to study large deviations for the height function $\mathfrak{h}$ using asymptotic formulas for $f^{\lambda/\rho}$, the number of standard Young tableaux of skew shape $\lambda/\rho$. In \cite{morales_pak_tassy_2022}, authors managed to compute the asymptotic limit of the Naruse hook formula \eqref{eq:naruse_hook}, which expresses $f^{\lambda/\rho}$ in terms of the evaluation of a functional $c(\phi_{\lambda / \rho})$ which is a deformation of the hook integral \eqref{eq:hook_integral}. The expression of the functional $c$ involves solving a variational problem that arises from a connection between lozenge tilings on a hexagonal lattice and the Naruse hook formula \cite{morales_pak_panova_hook3} and for this reason, it appears significantly more involved to analyze explicitly. Variational problems associated with the asymptotic number of standard Young tableaux have appeared in literature also in \cite{sun2018dimer,gordenko2020limit}.
It would be  interesting to analyze these various variational problems to possibly extract further properties of the lower-tail rate function $\Phi_-$; we leave these further approaches for future works.

\subsubsection{Approach through geometric techniques from the theory of last passage percolation}

{Yet another possible way of studying lower-tail behaviors of the $q$-PNG model is through its connection to a cylindrical Poissonian last passage percolation. The connection between $q$-PushTASEP and a model of last passage percolation in a cylindric geometry was essentially discovered in \cite{IMS_skew_RSK,IMS_KPZ_free_fermions} and used by Corwin and Hegde to obtain bounds for the tail probabilities of the height function in a $q$-PushTASEP in the moderate deviation regime \cite{corwin_hegde2022lower}. For a thorough description of the last passage percolation model, the reader should consult \cite[Section 1.4, 1.5]{corwin_hegde2022lower}. Following the scaling limit presented in \Cref{thm:matching_qW_qPNG} which transforms the $q$-PushTASEP into the $q$-PNG we can transform the last passage percolation model with 
random geometric weights into a model with Poisson rates.}

{Although it is possible that techniques developed in \cite{corwin_hegde2022lower} could adapt well to our situation and possibly allow us to derive moderate deviation bounds for the tails of the distribution of $\mathfrak{h}$, we are also positive that sharper bounds, in the style of results developed by Ganguly and Hegde \cite{ganguly_hegde2023optimal}, could be obtained. We leave these interesting directions for future works.}

\subsection*{Organization} The rest of the article is organized as follows. In \Cref{sec:qPNG}, we describe the connections between $q$-PNG and other solvable models of interest. The upper-tail LDP and lower-tail LDP are proven in \Cref{sec:upper_tail} and \Cref{sec:lower_tail} respectively. In \Cref{sec:further_approaches}, we discuss possible different approaches that could lead to an explicit characterization for the lower-tail rate function.  

\subsection*{Notation and Conventions} Throughout the paper we fix a $q\in (0,1)$ and use the notation $\logq:=\log q^{-1}$. We use $\Con(x,y,z,\ldots)>0$ to denote a generic
deterministic positive finite constant that is dependent on the designated variables $x,y,z,\ldots.$ $\#A$ denotes the cardinality of a finite set $A$.

\section{$q$-Deformed Polynuclear Growth and other solvable models} \label{sec:qPNG}

In this section, we show how the $q$-Deformed Polynuclear Growth is related to several other solvable models in probability and algebraic combinatorics. In \Cref{sec:2.1} and \Cref{subs:relationship_qPNG_cplan}, we connect $q$-PNG to $q$-PushTASEP and Cylindric Plancherel measures respectively. In \Cref{sec:2.3}, we describe a sampling procedure for cylindric Plancherel measure which lead us to derive crucial moment bounds for the observables in that model. Finally, in \Cref{sec:2.4} we derive exact formulas descending from the above connections. These formulas will be a central tool for our analysis of probabilities of tail events in \Cref{sec:upper_tail,sec:lower_tail}.

\subsection{$q$-PNG as a limit of the $q$-PushTASEP} \label{sec:2.1}

The scope of this subsection is to relate the $q$-PNG described in the introduction with another solvable model, the $q$-PushTASEP \cite{BorodinPetrov2013NN,MatveevPetrov2014}. This connection had been predicted in \cite{aggarwal_borodin_wheeler_tPNG}, where the $q$-PNG was defined as a scaling limit of another stochastic model.

The $q$-PushTASEP is a discrete-time interacting particle system whose distribution
is related to the $q$-Whittaker measure \cite{BorodinCorwin2011Macdonald}.  
In order to describe the model we need to define a special probability distribution that generalizes the beta binomial law, given next. Here and below we will make use of the notion of $q$-Pochhammer symbol
\begin{equation*}
    (z;q)_n := (1-z) (1-q z) \cdots (1-q^{n-1} z),
\end{equation*}
 for $z\in\mathbb{C}$ and $n\in \mathbb{N}\cup \{+\infty\}$. 

\begin{definition}
    The $q$-deformed beta binomial distribution $\varphi_{q,\mu,\nu}(\cdot |m)$ is given by
    \begin{equation*}
        \varphi_{q,\mu,\nu}(k|m) = \mu^k \frac{(\nu/\mu;q)_k(\mu;q)_{m-k}}{(\nu;q)_m} \frac{(q;q)_m}{(q;q)_k (q;q)_{m-k}}, \qquad k\in\{0,\dots,m\}. 
    \end{equation*}
    By virtue of the $q$-Chu-Vandermonde identity \cite{ismailBook} we have $\sum_{k=1}^m \varphi_{q,\mu,\nu}(k|m)=1$, whenever the sum converges. Moreover, several choices of $q,\mu,\nu$ guarantee that $\varphi_{q,\mu,\nu}(k|m)$ is positive. 
\end{definition}

A relevant specialization of the $q$-deformed beta binomial distribution comes from setting $\nu=0$, $m=+\infty$. We say that a random variable $Y$ has $q$-Geometric distribution of parameter $\mu\in (0,1)$ if
\begin{equation}\label{eq:qgeo}
    \mathbb{P}(Y = k) = \varphi_{q,\mu,0}(k|+\infty){=\mu^k \frac{(\mu;q)_{\infty}}{(q;q)_k}, \qquad k\in\Z},
\end{equation}
and in this case we write $Y \sim q$-Geo$(\mu)$. We are now ready to define the $q$-PushTASEP.

\begin{definition}
    Let $N \in \mathbb{N}$ and consider locations {$y_1<\cdots<y_N$} with $y_i \in \mathbb{Z}$. The $N$-particles $q$-PushTASEP is the process $\{ \mathsf{x}_k(T):k=1,\dots,N, \, T=1,2,\dots \}$, where 
    \begin{itemize}
        \item at time $T=0$, we have $\mathsf{x}_k(0) = y_k$, for $k=1,\dots,N$;
        \item at time $T$ the $k$-th particle evolves as
        \begin{equation*}
            \mathsf{x}_k(T) = \mathsf{x}_k(T-1) + \mathsf{J}_{k,T} + \mathsf{P}_{k,T},
        \end{equation*}
        where $\mathsf{J}_{k,T} \sim q$-Geo$(a)$ with $a>0$ and $$\mathsf{P}_{k,T} \sim \varphi_{q^{-1},q^{\mathsf{x}_{k}(T-1)-\mathsf{x}_{k-1}(T-1)},0}(\cdot | \mathsf{x}_{k-1}(T) - \mathsf{x}_{k-1}(T-1)).$$
        Here we assume $\mathsf{x}_0(T)=-\infty$ by convention.
    \end{itemize}
    When $y_i=i$ for $i=1,\dots,N$ we call this the $q$-PushTASEP with step initial condition.
\end{definition}

{In words, in the $q$-pushTASEP particles evolve following a sequential update from the leftmost to the right. The jump of each particle consists in the sum of two independent components. One is a pushing effect, produced by the random variables $\mathsf{P}_{k,T}$, which depends on the movements of particles to the left. The other is a independent jump given by the random variables $\mathsf{J}_{k,T}$. The pushing effect in particular makes sure that the exclusion rule is preserved following each update.}

The $q$-PushTASEP model is of particular interest as it degenerates to various well-known models. Indeed, in \cite{MatveevPetrov2014}, it was shown that taking $q\to 1$ limit of the model, one obtains the log-gamma polymer model \cite{Seppalainen2012}. The log-gamma polymer model itself degenerates to continuum directed random polymer under intermediate disorder regime \cite{AlbertsKhaninQuastel2012} whose free energy is related to the KPZ equation \cite{KPZ1986}. Presently, we describe how $q$-PushTASEP degenerates to $q$-PNG as well.

To this end, we first construct a vertex model from the $q$-PushTASEP particle system (see \cite{BufetovMucciconiPetrov2018} also). At each location $(k,T)\in ([1,N]\cap \Z)\times \Z$, consider the random variables
    \begin{align*}
        &J= \mathsf{x}_{k-1}(T) - \mathsf{x}_{k-1}(T-1),
        \qquad
        &J'= \mathsf{x}_k(T) - \mathsf{x}_k(T-1),
        \\
        &G= \mathsf{x}_{k}(T-1)-\mathsf{x}_{k-1}(T-1) -1, \qquad &G'=\mathsf{x}_k(T)-\mathsf{x}_{k-1}(T) -1.
    \end{align*}
    {We interpret $G$ and $J$ to be the number of arrows entering from below and left into $(k,T)$ respectively, and $G'$ and $J'$ to be the number of arrows exiting above and to the right from $(k,T)$ respectively (see \Cref{fig:arrow}).
    \begin{figure}[t]
        \centering
        \setlength{\tabcolsep}{2em}
        \begin{tabular}{cc}
            \begin{tikzpicture}[anchor=base,baseline,scale=.5]
                \draw[] (0,0) -- (12,0);
                \draw[] (0,1) -- (12,1);
                \foreach \i in {1,...,11}
                {
                    \draw[] (\i,0) -- (\i,1);
                }
                \draw[fill] (1.5,.5) circle(.25);
                \draw[fill] (6.5,.5) circle(.25);
                \draw[densely dotted] (4.5,.5) circle(.25);
                \draw[densely dotted] (10.5,.5) circle(.25);
                \draw[>=stealth,->] (2.15,1.5) -- node[above] {$J$} (3.7,1.5);
                \draw[>=stealth,->] (6.85,1.5) -- node[above] {$J'$} (10.1,1.5);
                \draw[>=stealth,<->] (2.1,-.3) -- node[midway,below] {$G$} (5.9,-.3);
                \draw[>=stealth,<->] (5.1,-.6) -- node[midway,below] {$G'$} (9.9,-.6);
                \node[above] at (1.5,1) {\scriptsize $\mathsf{x}_{k-1}$};
                \node[above] at (4.5,.9) {\scriptsize $\mathsf{x}_{k-1}'$};
                \node[above] at (6.5,1) {\scriptsize $\mathsf{x}_{k}$};
                \node[above] at (10.5,.9) {\scriptsize $\mathsf{x}_{k}'$};
            \end{tikzpicture}
            &
            \begin{tikzpicture}[anchor=base,baseline,scale=.7]
                \draw [fill=black] (0,0) circle (5pt);
                \foreach \p in {(-2.1,-.15),(-2.1,0),(-2.1,.15)}
                {
                    \draw [>=stealth,thick,->]  \p --+ (1.8,0);
                }
                \node[left] at (-2.1,0) {\scriptsize $J$};
                \foreach \p in {(.3,.2),(.3,0.075),(.3,-.075),(.3,-.2)}
                {
                    \draw [>=stealth,thick,->]  \p --+ (1.8,0);
                }
                \node[right] at (2.1,0) {\scriptsize $J'$};
                \foreach \p in {(0.2,-2.1),(0.075,-2.1),(-0.075,-2.1),(-0.2,-2.1)}
                {
                    \draw [>=stealth,thick,->]  \p --+ (0,1.8);
                }
                \node[below] at (0,-2.1) {\scriptsize $G$};
                \foreach \p in {(0.3,.3),(0.15,.3),(0,.3),(-0.15,.3),(-0.3,.3)}
                {
                    \draw [>=stealth,thick,->]  \p --+ (0,1.8);
                }
                \node[above] at (0,2.1) {\scriptsize $G'$};
            \end{tikzpicture}
        \end{tabular}
        \caption{On the left panel a local update of $q$-pushTASEP particles $\mathsf{x}_{k-1}=\mathsf{x}_{k-1}(T-1)$, $\mathsf{x}_{k}=\mathsf{x}_{k}(T-1)$ into $\mathsf{x}_{k-1}'=\mathsf{x}_{k-1}(T)$, $\mathsf{x}_{k}'=\mathsf{x}_{k}(T)$. In the right panel the representation of the update as a vertex configuration.}
        \label{fig:arrow}
    \end{figure}
 }
 
 Note that the arrows satisfy certain conservation property: $J-G=J'-G'$. Let us write
$\mathsf{R}(j,g;j',g'):=\Pr\left[J'=j',G'=g' \,|\, J=j,G=g  \right].$
We have
    \begin{equation} \label{eq:stochastic R}
       \mathsf{R}(j,g;j',g')  = \sum_{k = 0}^{j'} \varphi_{q^{-1},q^g,0}(j'-k|j) \frac{a^k}{(q;q)_k} (a;q)_\infty. 
    \end{equation}
    Considering the specialization $a=\varepsilon^2 \theta^2$ we see that, up to order $\varepsilon^2$, the relevant cases are when $j,g,j',g'\in\{0,1\}$ and we find
    \begin{equation*}
       \mathsf{R}(0,0;0,0)  , \mathsf{R}(1,0;1,0) , \mathsf{R}(0,1;0,1) = 1+{O(\varepsilon^2)}.
    \end{equation*}
    Thus most of the vertices in the lattice will have no arrows. If there is a single arrow entering, it exits in the same direction with high probability. However, note that
    \begin{equation*}
    \mathsf{R}(1,1;0,0) =  1-q + O(\varepsilon^2),
        \qquad
 \mathsf{R}(1,1;1,1)  = q + O(\varepsilon^2).
    \end{equation*}
    Thus if two arrows enter at a point, they pass through each other with a probability close to $q$ and annihilate each other with a probability close to $1-q$. 
    Finally, we observe 
    \begin{equation*}
        \mathsf{R}(0,0;1,1) = \frac{\varepsilon^2 \theta^2}{1-q} + O(\varepsilon^4).
    \end{equation*}
Thus with $O(\e^2)$ probability a pair of exiting arrows is created, or nucleates. Taking $k,T$ of the order $\e^{-1}$, we thus expect $O(1)$ many nucleations, and they are distributed according to a Poisson point process in the $\e\downarrow 0$ limit. Given the definition of $q$-PNG from \Cref{def:qpng},  the above heuristic computation suggests that upon a $45^\circ$ rotation, the above-constructed vertex model is converging to $q$-PNG in the $\e\downarrow 0$ limit.


\begin{theorem} \label{thm:matching_qW_qPNG}
    Let $\{ \mathsf{x}_k(T):k=1,\dots,N, \, T=1,2,\dots \}$ be the $q$-PushTASEP with step initial condition. Then, under the scaling
    \begin{equation}\label{eq:scaling}
        a=\varepsilon^2 \theta^2,
        \qquad
        N = \lfloor \varepsilon^{-1} 2^{-1/2} (t+x) \rfloor,
        \qquad
        T = \lfloor \varepsilon^{-1} 2^{-1/2} (t-x) \rfloor,
    \end{equation} 
    we have, as a process in $(x,t)$ 
    \begin{equation*} 
        \mathsf{x}_{N}(T) - N \xrightarrow[\varepsilon \downarrow 0]{{d}} \mathfrak{h}(x,t)
    \end{equation*}
    in the sense of distribution w.r.t~the uniform-on-compact topology. Here $\mathfrak{h}$ denotes the height function of the $q$-PNG with intensity $\Lambda=\theta^2/(1-q)$ and droplet initial condition.
\end{theorem}

{
\begin{proof}[Sketch of proof]
    The argument outlined in the above paragraphs can be made rigorous by adapting the proof of \cite[proposition 4.4]{aggarwal_borodin_wheeler_tPNG} given in Appendix B of the same paper. To avoid repetition, we direct the readers to \cite{aggarwal_borodin_wheeler_tPNG} for details and provide below only an outline of the main ideas.
\\
\indent
    Denote by $\widetilde{\mathfrak{h}}(\eta,\tau) = \mathfrak{h}(2^{-1/2}(\eta-\tau),2^{-1/2}(\eta+\tau))$, the height function of the $q$-PNG in a space-time frame rotated by $45^\circ$ and by $\widetilde{q\mathrm{-PNG}}$ the push-forward measure defined by the change of coordinates $\mathfrak{h} \to \widetilde{\mathfrak{h}}$. 
\\
\indent
    As explained in the previous paragraphs the $q$-pushTASEP can be seen as a stochastic vertex model with stochastic vertex weights $\mathsf{R}$ as in \eqref{eq:stochastic R}. In this vertex model step initial conditions correspond to taking edges $(1,i)- (1,i+1)$ occupied by infinitely many arrows while edges $(i,0) -(i+1,0)$ empty, as shown in \Cref{fig:vertex model}. In this language the height function is $\mathsf{h}(k,T) = \mathsf{x}_k(T)-k$ which counts the number of edge arrows crossed by a segment with endpoints $(k+\frac{1}{2},T+\frac{1}{2})$ and $(\frac{1}{2},\frac{1}{2})$. Under the scaling of parameter $a=\varepsilon^2 \theta^2$ we denote the height function by $\mathsf{h}_\varepsilon$. 
\\
\indent
    For fixed real numbers $A,B>0$ we define the finite rectangular lattice 
    \begin{equation*}
        \Lambda^{(A,B)}_\varepsilon = \{0,\dots, \lfloor A/\varepsilon \rfloor\} \times \{0,\dots, \lfloor B/\varepsilon \rfloor\}.
    \end{equation*}
    We also define the events
    \begin{equation*}
        \mathfrak{E}^{(A,B)}_\varepsilon = \begin{minipage}{0.7\textwidth}
            for all $(i,j)\in \Lambda^{(A,B)}_\varepsilon$ there is at most one nucleation in the cross-shaped set $[i\varepsilon, (i+1)\varepsilon) \times [0,B] \cup [0,A] \times [j\varepsilon , (j+1)\varepsilon)$.
        \end{minipage}
    \end{equation*}
    and
    \begin{equation*}
        \mathsf{E}^{(A,B)}_\varepsilon = \begin{minipage}{0.7\textwidth}
            no edge in $\Lambda^{(A,B)}_\varepsilon$ has more than a single arrow and no more than one configuration \, $\begin{tikzpicture}[anchor=base,baseline=-3,scale=.5]
                \draw [fill=black] (0,0) circle (2pt);
                \draw[densely dotted] (-1,0) -- (-.2,0);
                \draw[densely dotted] (0,-1) -- (0,-.2);
                \draw[->] (.2,0) -- (1,0);
                \draw[->] (0,.2) -- (0,1);
                \end{tikzpicture}$ \,
            per row and column in $\Lambda^{(A,B)}_\varepsilon$.
        \end{minipage}
    \end{equation*}
    A simple calculation shows that 
    \begin{equation*}
        \mathbb{P}_{\widetilde{q\mathrm{-PNG}}} \left(  \mathfrak{E}^{(A,B)}_\varepsilon \right) = 1- O(\varepsilon AB)
        \qquad
        \mathbb{P}_{q\mathrm{-pushTASEP}} \left(  \mathsf{E}^{(A,B)}_\varepsilon \right) = 1- O(\varepsilon AB).
    \end{equation*}
    Conditioning the processes $\widetilde{q\mathrm{-PNG}}$ and $q\mathrm{-pushTASEP}$ respectively to events $\mathfrak{E}^{(A,B)}_\varepsilon$ and $\mathsf{E}^{(A,B)}_\varepsilon$ it is possible to couple them in such a way that 
    \begin{equation*}
        \widetilde{\mathfrak{h}}(\eta,\tau) = \mathsf{h}_\varepsilon(\eta/\varepsilon, \tau/\varepsilon) \qquad \text{for all $\eta,\tau$ such that } (\eta/\varepsilon, \tau/\varepsilon) \in \Lambda^{(A,B)}_\varepsilon,
    \end{equation*}
    with probability $1-O(\varepsilon AB)$. This shows that, letting $\varepsilon$ tend to $0$, we have $\mathsf{h}_\varepsilon \to \widetilde{\mathfrak{h}}$ in law over the compact set $[0,A]\times [0,B]$. Upon rotating back the space time coordinates $(\eta,\tau) \to (x,t)$, this proves the theorem. 
\end{proof}
}

{
\begin{figure}
    \centering
    \begin{tabular}{cc}
        \begin{tikzpicture}[scale=0.75]
            \draw[] (0,0) grid[step=.4] (8.3,8);
            \foreach \y in {1,...,20}
            {
                \node[right]  at (8,.4*\y-.2) {\tiny $\cdots$};
            }
            \foreach \x in {1,...,20}
            {
                \node[below] at (.4*\x-.2,0) {\tiny $\x$};
            }
            \foreach \x in {0,...,19}
            {
                \draw[fill]  (.4*\x+.2,.4*0+.2) circle(.1);
            }
            \foreach \x in {0,...,19}
            {
                \draw[fill]  (.4*\x+.2,.4*1+.2) circle(.1);
            }
            \foreach \x in {0,...,11,13,14,15,16,17,18,19}
            {
                \draw[fill]  (.4*\x+.2,.4*2+.2) circle(.1);
            }
            \foreach \x in {0,...,11,13,14,15,16,17,18,19}
            {
                \draw[fill]  (.4*\x+.2,.4*3+.2) circle(.1);
            }
            \foreach \x in {0,...,11,13,14,15,16,17,18,19}
            {
                \draw[fill]  (.4*\x+.2,.4*4+.2) circle(.1);
            }
            \foreach \x in {0,...,11,13,14,15,16,17,18,19}
            {
                \draw[fill]  (.4*\x+.2,.4*5+.2) circle(.1);
            }
            \foreach \x in {0,1,2,3,5,6,7,8,9,10,11,12,14,15,16,17,18,19}
            {
                \draw[fill]  (.4*\x+.2,.4*6+.2) circle(.1);
            }
            \foreach \x in {0,1,2,3,5,6,7,8,9,10,11,12,14,15,16,17,18,19}
            {
                \draw[fill]  (.4*\x+.2,.4*7+.2) circle(.1);
            }
            \foreach \x in {0,1,2,3,5,6,7,8,9,10,11,12,14,15,16,17,18}
            {
                \draw[fill]  (.4*\x+.2,.4*8+.2) circle(.1);
            }
            \foreach \x in {0,1,2,3,5,6,7,8,9,10,11,12,14,15,16,17,18}
            {
                \draw[fill]  (.4*\x+.2,.4*9+.2) circle(.1);
            }
            \foreach \x in {0,1,2,3,5,6,7,8,9,10,11,12,14,15,16,17,18}
            {
                \draw[fill]  (.4*\x+.2,.4*10+.2) circle(.1);
            }
            \foreach \x in {0,1,2,3,5,6,7,8,9,10,11,12,14,15,16,17,18}
            {
                \draw[fill]  (.4*\x+.2,.4*11+.2) circle(.1);
            }
            \foreach \x in {0,1,2,3,5,6,7,8,9,10,11,12,14,15,16,17,18}
            {
                \draw[fill]  (.4*\x+.2,.4*12+.2) circle(.1);
            }
            \foreach \x in {0,1,2,3,5,6,7,8,9,10,11,12,14,15,16,17,18}
            {
                \draw[fill]  (.4*\x+.2,.4*12+.2) circle(.1);
            }
            \foreach \x in {0,1,2,3,5,6,7,8,9,10,11,12,14,15,17,18,19}
            {
                \draw[fill]  (.4*\x+.2,.4*13+.2) circle(.1);
            }
            \foreach \x in {0,1,2,3,5,6,7,8,9,10,11,12,14,15,17,18,19}
            {
                \draw[fill]  (.4*\x+.2,.4*14+.2) circle(.1);
            }
            \foreach \x in {0,1,2,3,5,6,7,8,9,10,11,12,14,15,17,18,19}
            {
                \draw[fill]  (.4*\x+.2,.4*15+.2) circle(.1);
            }
            \foreach \x in {0,1,2,3,5,6,7,8,9,10,11,12,14,15,17,18,19}
            {
                \draw[fill]  (.4*\x+.2,.4*16+.2) circle(.1);
            }
            \foreach \x in {0,2,3,4,6,7,8,9,10,11,12,13,14,15,17,18,19}
            {
                \draw[fill]  (.4*\x+.2,.4*17+.2) circle(.1);
            }
            \foreach \x in {0,2,3,4,6,7,8,9,10,11,12,13,14,15,17,18,19}
            {
                \draw[fill]  (.4*\x+.2,.4*18+.2) circle(.1);
            }
            \foreach \x in {0,2,3,4,6,7,8,9,10,11,12,13,14,15,17,18,19}
            {
                \draw[fill]  (.4*\x+.2,.4*19+.2) circle(.1);
            }
        \end{tikzpicture}
        &
        \begin{tikzpicture}[scale=0.75]
            \draw[densely dotted] (0,0) grid[step=.4] (8,8);
            \foreach \i in {0,...,20}
            {
                \foreach \j in {0,...,20}
                {
                    \draw[fill] (0.4*\i,0.4*\j) circle(.02);
                }
            }
            \foreach \y in {0,...,19}
            {
                \draw[thick,->] (0,.4*\y+0.05) --+ (0,0.3);
                \node[left] at (0,.4*\y+0.2) {\tiny $\infty$};
            }
            \foreach \x in {12,...,20}
            {
                \draw [>=stealth,thick,->]  (.4*\x+.05,.4*2) --+ (0.3,0);
            }
            \foreach \x in {4,...,20}
            {
                \draw [>=stealth,thick,->]  (.4*\x+.05,.4*6) --+ (0.3,0);
            }
            \foreach \x in {1,...,11}
            {
                \draw [>=stealth,thick,->]  (.4*\x+.05,.4*17) --+ (0.3,0);
            }
            \foreach \x in {14,...,16}
            {
                \draw [>=stealth,thick,->]  (.4*\x+.05,.4*13) --+ (0.3,0);
            }
            \foreach \x in {17,...,20}
            {
                \draw [>=stealth,thick,->]  (.4*\x+.05,.4*8) --+ (0.3,0);
            }
            \foreach \y in {2,...,16}
            {
                \draw [>=stealth,thick,->]  (.4*12,.4*\y+.05) --+ (0,0.3);
            }
            \foreach \y in {6,...,20}
            {
                \draw [>=stealth,thick,->]  (.4*4,.4*\y+.05) --+ (0,0.3);
            }
            \foreach \y in {17,...,20}
            {
                \draw [>=stealth,thick,->]  (.4*1,.4*\y+.05) --+ (0,0.3);
            }
            \foreach \y in {13,...,20}
            {
                \draw [>=stealth,thick,->]  (.4*14,.4*\y+.05) --+ (0,0.3);
            }
            \foreach \y in {8,...,12}
            {
                \draw [>=stealth,thick,->]  (.4*17,.4*\y+.05) --+ (0,0.3);
            }
            \foreach \x in {1,...,21}
            {
                \node[below] at (.4*\x-.4,0) {\tiny $\x$};
            }
        \end{tikzpicture}
    \end{tabular}
    \caption{On the left panel a typical evolution of a $q$-pushTASEP under the scaling \eqref{eq:scaling} with $\varepsilon$ small. Here discrete time runs upward and the bottom row consists in a step initial condition. The left panel depicts the vertex configuration corresponding to the dynamics on the left panel.}
    \label{fig:vertex model}
\end{figure}
}

\subsection{Relationships between $q$-PNG and Cylindric Plancherel measures} \label{subs:relationship_qPNG_cplan}

The connection between the $q$-PNG and the $q$-PushTASEP elaborated in the previous subsection allows us to establish a relation between the $q$-PNG and cylindric Plancherel measure. This is because of a more general connection relating the periodic Schur measure \cite{borodin2007periodic} and the $q$-PushTASEP discovered in \cite{IMS_KPZ_free_fermions}. In order to discuss these developments we need to introduce some notation.

A partition $\lambda$ is a decreasing sequence of non-negative integers $\lambda_1 \ge \lambda_2 \ge \cdots$ which eventually become zero $\lambda_m=\lambda_{m+1}=\cdots=0$. The size of a partition $\lambda$ is the sum of its elements $|\lambda|=\lambda_1+ \lambda_2 + \cdots$.
For any $n\in \mathbb{N}\cup\{0\}$ we denote 
\begin{equation} \label{eq:p_n}
    \mathsf{p}_n = \#\{ \lambda \text{ partition } : |\lambda|=n\}.
\end{equation}
It is well-known (see \cite[Eq.~(1.15)]{romik_2015} for example) that $\mathsf{p}_n$ is of exponential order in the square root of $n$. Indeed, there exists a constant $C>0$ such that 
\begin{equation}
    \label{eq:pnbd}
    \mathsf{p}_n \le e^{C\sqrt{n}}
\end{equation}
holds for all $n$. Graphically a partition is represented by its Young diagram, obtained drawing, one above the other, left justified rows of cells of lengths given by elements $\lambda_1,\lambda_2,\dots$ as done below in \Cref{fig:partition_tableau}. We are going to identify a partition with its Young diagram and for us, the two notions will be equivalent. Given a partition $\lambda$ we also define its transpose $\lambda' = (\lambda_1',\lambda_2',\dots)$, where $\lambda_i'=\#\{ j : \lambda_j \ge i\}$; the Young diagram of $\lambda'$ is obtained from that of $\lambda$ by a reflection with respect to the diagonal. A partition $\rho$ is contained in another partition $\lambda$ if $\rho_i \le \lambda_i$ for all $i=1,2,\dots$ and in this case we write $\rho \subset \lambda$. This relation also implies that all cells of the Young diagram of $\rho$ belong also to that of $\lambda$.  Whenever $\rho \subset \lambda$ we define the skew partition $\lambda/\rho$ which is the set of cells of the Young diagram of $\lambda$ which do not belong to $\rho$. The size of a skew partition (or equivalently of a skew Young diagram) $\lambda/\rho$ is $|\lambda / \rho| = |\lambda| - |\rho|$.

Labeling cells of skew Young diagrams with natural numbers defines Young tableaux. We say that a Young tableau $T$ of shape $\lambda/\rho$ is \emph{partial} if $T$ is a labeling of cells of $\lambda/\rho$, where values are increasing row-wise and column-wise and each value appears at most once. A tableau $T$ is \emph{standard} if it is a partial tableau where values range in $\{1,\dots , |\lambda/\rho|\}$. Partial and standard Young tableaux are also defined when $\rho = \varnothing$, and in this case we say that their shape is straight (as opposed to skew). The content of a tableau $T$ is the set of its labels
\begin{equation}\label{def:conttab}
    \mathrm{cont} (T) = \{ T(i,j) : (i,j) \in \mathrm{shape}(T) \}. 
\end{equation}

\begin{figure}[htbp]
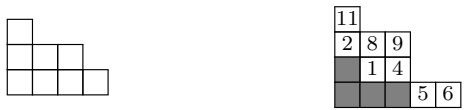

    \centering
    \begin{equation*} \ytableausetup{centertableaux,smalltableaux}
    \ydiagram{1,3,4}
    \hspace{3cm}
    \begin{ytableau}
    11\\
    2 & 8 & 9\\
    *(gray) & 1 & 4 \\
    *(gray) & *(gray) & *(gray) & 5 & 6
    \end{ytableau}
\end{equation*}
    \caption{On the left the Young diagram corresponds to the partition $(4,3,1)$. On the right a partial Young tableau of skew shape $(5,3,3,1)/(3,1)$ and content $\{1,2,4,5,6,8,9,11\}$. }
    \label{fig:partition_tableau}
\end{figure}

For a skew partition $\lambda/\rho$ we define the number
\begin{equation}\label{eq:flr}
    f^{\lambda/\rho} := \# \{ T  :  T \text{ is a standard Young tableau of shape } \lambda/\rho \}.
\end{equation}
Counting standard Young tableau of a given shape has been a problem of great interest in algebraic combinatorics and representation theory. In the case of straight shapes the celebrated hook length formula \cite{frame1954hook} provides the closed expression
\begin{equation} \label{eq:hook_length_formula}
    f^\lambda = \frac{|\lambda|!}{ \prod_{c \in \lambda} h_\lambda(c) },
\end{equation}
where $h_\lambda(c)-\lambda_i - i + \lambda_j' -j +1$ is the \emph{hook length} of the cell $c = (i,j)$ in $\lambda$.
A similar formula is available for the number of standard Young tableaux of skew shape:
\begin{equation} \label{eq:naruse_hook}
    f^{\lambda/\rho} = |\lambda/\rho|! \sum_{C\in \mathcal{E}(\lambda/\rho)} \prod_{c \in \lambda \setminus C } \frac{1}{h_\lambda(c)}.
\end{equation}
where the sum is performed over the $\mathcal{E}(\lambda / \rho)$, the set of all \emph{excited Young diagrams} of $\lambda / \rho$ (which, for the sake of brevity we will not define here). This formula was discovered by Naruse \cite{naruse2014schubert} and a proof can be found in \cite{Morales_Pak_Panova_hook}. There exist also other formulas for $f^{\lambda/\rho}$ in terms of Littlewood–Richardson coefficients, determinants \cite{Stanley1999}, and reverse semistandard tableaux \cite{okounkov1997shifted}.

With these definitions in place, we are ready to define an important measure on the set of partitions.

\begin{definition}
    Fix any $\gamma>0$. The cylindric Plancherel measure with intensity $\gamma$ is the probability measure on the set of skew partitions
    \begin{equation*}
        \mathbb{P}_{\mathsf{cPlan}(\gamma)}(\lambda/\rho) = q^{|\rho|} \gamma^{2 |\lambda/\rho|} \left( \frac{f^{\lambda/\rho}}{|\lambda/\rho|!} \right)^2 e^{-\gamma^2/(1-q)} (q;q)_\infty.
    \end{equation*}
    It was shown in \cite[Example 3.4]{borodin2007periodic} that the measure of the whole set of skew partitions is one, i.e. $\sum_{\rho \subset \lambda} \mathbb{P}_{\mathsf{cPlan}(\gamma)}(\lambda/\rho) =1$.
\end{definition}
The cylindric Plancherel measure is a particular case of the periodic Schur measure, which is also a probability measure on the set of skew partitions proportional to
 \begin{equation} \label{eq:periodic_schur}
     q^{\rho} s_{\lambda/\rho}(a_1,\dots, a_N) s_{\lambda/\rho}(b_1,\dots, b_T) \left(\prod_{i=1}^N \prod_{j=1}^T(a_ib_j;q)_\infty \right) (q;q)_\infty.
 \end{equation}
 It was introduced by Borodin in \cite{borodin2007periodic} and more recently revisited by Betea-Bouttier \cite{betea_bouttier_periodic}. In \eqref{eq:periodic_schur} functions $s_{\lambda/\rho}$ are the Schur polynomials \cite{macdonald1988new} and $a_1,\dots,a_N,b_1,\dots,b_T \in [0,1)$. The cylindric Plancherel measure is recovered from \eqref{eq:periodic_schur} setting $N=T=n$, $a_i=b_j=\gamma/n$ for all $i,j$ and taking the limit $n\to +\infty$. Under this limit we have
 \begin{equation}\label{eq:Schur_exponential}
     s_{\lambda/\rho} \xrightarrow[n \to +\infty]{} \gamma^{|\lambda/\rho|} \frac{f^{\lambda/\rho}}{|\lambda/\rho|!} 
     \qquad
     \text{and}
     \qquad
     \prod_{i=1}^N \prod_{j=1}^T(a_ib_j;q)_\infty \xrightarrow[n \to +\infty]{}  e^{-\gamma^2/(1-q)}.
 \end{equation}
 The limit $q\to 0$ of the cylindric Plancherel measure recovers the Poissonized Plancherel measure \cite{borodin2000b}, which we denote by
 \begin{equation*}
     \mathbb{P}_{\mathsf{Plan}(\gamma)}(\lambda) = \gamma^{2 |\lambda|} \left( \frac{f^{\lambda}}{|\lambda|!} \right)^2 e^{-\gamma^2}.
 \end{equation*}
Note that the Poissonized Plancherel measure is supported on straight partitions $\lambda$. On the other hand, taking $\gamma\to 0$, the cylindric Plancherel measure becomes the volume measure
\begin{equation} \label{eq:volume_measure}
    \mathbb{P}_{\mathsf{vol}}(\rho) = q^{|\rho|} (q;q)_\infty.
\end{equation}

The following theorem states an equivalence in law between the height function of the $q$-PNG and the length of the first row in a cylindric Plancherel measure.

\begin{theorem} \label{thm:matching_qPNG_cylindric_plancherel}
    Fix any $\theta>0$. Let $\mathfrak{h}$ be the height function of the $q$-PNG with intensity $\Lambda=\theta^2/(1-q)$ and droplet initial condition and let $\chi \sim q$-$\mathrm{Geo}(q)$ be an independent random variable. Then, for all $t>|x|>0$ we have
    \begin{equation} \label{eq:matching_height_qPNG_cylindric_Plancherel}
         \mathfrak{h}(x,t) + \chi \stackrel{d}{=} \lambda_1,
    \end{equation}
    where $\lambda_1$ is the first row of a random partition $\lambda$ where $\lambda/\rho \sim \mathbb{P}_{\mathsf{cPlan}(\theta \sqrt{(t^2-x^2)/2})}$.
\end{theorem}

\begin{proof}
    In \cite{IMS_KPZ_free_fermions} it was shown that 
    \begin{equation} \label{eq:matching_qPushTASEP_periodicSchur}
        \mathsf{x}_N(T)-N +\chi = \bar{\lambda}_1,
    \end{equation} 
    where $\mathsf{x}_N(T)$ is the $N$-th particle in the $q$-PushTASEP with step initial conditions and $\bar{\lambda}_1$ is the first row of a partition $\bar{\lambda}$ distributed according to the periodic Schur measure \eqref{eq:periodic_schur} with parameters $a_i=b_j=\sqrt{a}$ for all $i,j$. This is a consequence of \cite[Theorem 4.10]{IMS_KPZ_free_fermions} and \cite[Proposition 3.1]{IMS_KPZ_free_fermions}, which are respectively rephrasing of \cite[Theorem 1.3]{IMS_matching} and \cite[Section 6.3]{MatveevPetrov2014}. Taking the scaling limit \eqref{eq:scaling}, in the periodic Schur measure on the right-hand side of \eqref{eq:matching_qPushTASEP_periodicSchur} we see that
    \begin{equation*}
        s_{\lambda / \rho}(a_1,\dots, a_N) \xrightarrow[\varepsilon \downarrow 0]{} \left[\frac{\theta(t+x)}{\sqrt{2}} \right]^{|\lambda/\rho|} \frac{f^{\lambda/\rho}}{|\lambda/\rho|!},
        \ \ \
        s_{\lambda / \rho}(b_1,\dots, b_T) \xrightarrow[\varepsilon \downarrow 0]{} \left[\frac{\theta(t-x)}{\sqrt{2}} \right]^{|\lambda/\rho|} \frac{f^{\lambda/\rho}}{|\lambda/\rho|!}
    \end{equation*}
    and hence the limiting law of the partition $\bar{\lambda}$ becomes $\mathbb{P}_{\mathsf{cPlan}(\theta \sqrt{(t^2-x^2)/2})}$. Since by \Cref{thm:matching_qW_qPNG} under the same scaling limit $\mathsf{x}_N(T)-N$ converges to $\mathfrak{h}(x,t)$, this proves \eqref{eq:matching_height_qPNG_cylindric_Plancherel}.
\end{proof}

\begin{corollary}\label{cor:xzero}
   Fix any $\theta>0$. Let $\mathfrak{h}$ be the height function of the $q$-PNG with intensity $\Lambda=\theta^2/(1-q)$ and droplet initial condition.  Fix any $x\in \R$ and $t>0$ such that $t>|x|>0$. We have the following equality in distribution (in the sense of one-point marginals):
    \begin{equation*}
        \h(x,t)\stackrel{d}{=}\h\big(0,\sqrt{t^2-x^2}\big).
    \end{equation*}
\end{corollary}

Since our main results deal with only one-point marginals of the height function of $q$-PNG, we shall consider the height function at the origin, i.e.~$\h(0,t)$, for the remainder of the paper.

\subsection{Sampling the cylindric Plancherel measure}
\label{sec:2.3}
The cylindric Plancherel measure can be sampled leveraging a combinatorial construction discovered by Sagan and Stanley in \cite{sagan1990robinson}, which is a generalization of the celebrated Robinson-Schensted correspondence \cite{robinson1938representations,Schensted1961} (see e.g. \cite{Stanley1999}). Although we will not discuss precisely this construction, whose details are not used in this paper, we like to give a general idea of this sampling mechanism. {The procedure we present below is a generalization of the canonical use of the Robinson-Schensted correspondence to sample the Plancherel measure \cite{Schensted1961}; see also \cite[Section 1.8]{romik_2015}.} 

We use the notion of \emph{partial permutation matrices}, that are square matrices $M=(M_{i,j})_{i,j=1}^n$ such that $M_{i,j}\in \{0,1\}$ and for all $i,j$, we have $\sum_k M_{i,k} , \sum_k M_{k,j} \in \{ 0,1\}$. In other words, partial permutation matrices $M$ have at most one non-zero element per row and column. We define
\begin{equation}
    \mathrm{cont}(M) := \left\{ i:\sum_{k} M_{i,k} =1 \right\},
    \qquad \# M := \sum_{i,j=1}^n M_{i,j}.
\end{equation}
The Sagan-Stanley correspondence can be formulated as a bijection
\begin{equation} \label{eq:sagan_stanley}
    (P,Q; M ) \longleftrightarrow (\overline{P},\overline{Q})
\end{equation}
where $P,Q$ are a pair of partial tableaux of the same shape $\lambda/\rho$, $M$ is a partial permutation matrix such that $\mathrm{cont}(P)\cap \mathrm{cont}(M)= \mathrm{cont}(Q)\cap \mathrm{cont}(M^T)=\varnothing$ and $\overline{P},\overline{Q}$ are a pair of partial tableaux of the same shape $\mu/\lambda$, where
\begin{equation} \label{eq:properties_SS}
\begin{aligned}
    & \mathrm{cont}(P) \cup \mathrm{cont}(M) = \mathrm{cont}(\overline{P}),
    \quad
    \mathrm{cont}(Q) \cup \mathrm{cont}(M^T) = \mathrm{cont}(\overline{Q}), \\
    & |\mu/\lambda| = |\lambda /\rho| + \# M. 
\end{aligned}
\end{equation}
where recall the content of a tableau was defined in \eqref{def:conttab}. The correspondence \eqref{eq:sagan_stanley} can be used iteratively to build random skew partitions from random permutation matrices. For this let us consider $\{\mathscr{M}_{k}\}_{k \ge 0}$, a sequence of independent Poisson point processes on the square $(0,1) \times (0,1)$, where $\mathscr{M}_k$ has intensity $\gamma^2 q^k$ and $\nu$, an independent random partition taken with law $\mathbb{P}_\mathsf{vol}$ as in \eqref{eq:volume_measure}. We define 
\begin{equation} \label{eq:law_A_k}
    \mathcal{A}_k := \#\mathscr{M}_k,
    \qquad \qquad \mathcal{A}_k \sim \mathrm{Poi}(q^k \gamma^2)
\end{equation}
and we also define random variables 
\begin{equation} \label{eq:n_and_N}
    \mathsf{N} = \max \left\{ k: \mathscr{A}_k > 0 \right\}  
    \qquad \mathsf{n} = \sum_{k\ge 0} \mathscr{A}_k.
\end{equation}
By the Borel-Cantelli theorem, since the intensities of the Poisson random variables $\mathscr{A}_k$ decay exponentially we have that $\mathsf{N}$ is almost surely finite, whereas a simple calculation shows that
\begin{equation} \label{eq:n_Poisson}
    \mathsf{n} \sim \mathrm{Poi}(\gamma^2/(1-q)).
\end{equation}
Let us write $\mathscr{M}$ to denote the point process obtained by superimposing $\{\mathscr{M}_k\}_{k\ge 0}$. From any realization of the Poisson point processes $\{\mathscr{M}_{k}\}_{k \ge 0}$ and hence of $\mathsf{N}=N, \mathsf{n}= n$, we construct a sequence of $n \times n$ partial permutation matrices $\{M_k\}_{k \ge 0}$ as follows:

For each point $p=(p_x,p_y)\in \mathscr{M}_k$, we set $M_k(i_p,j_p)=1$ where
\begin{equation*}
    i_p:=\sum_{k'\ge 0} \#\{ \mathscr{M}_{k'} \cap ([0,1] \times [p_y,1])\},
    \qquad
    j_p:=\sum_{k'\ge0} \#\{\mathscr{M}_{k'} \cap ([0,p_x] \times [0,1])\}.
\end{equation*}
The remaining entries of $M_k$ are set to be zero (see \Cref{fig:mkmatrix}).
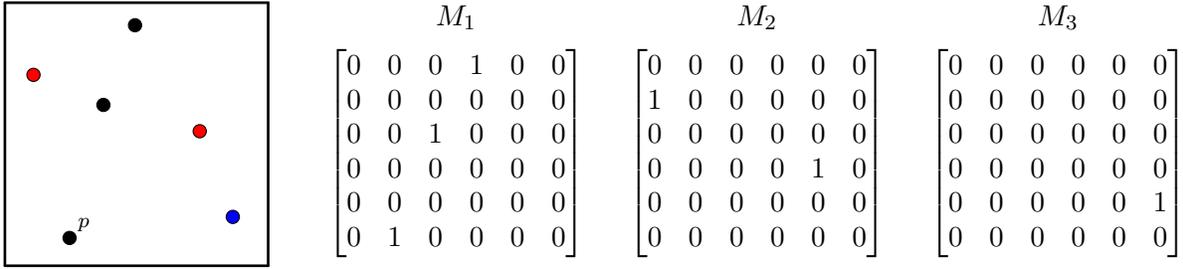
\begin{figure}[t]
    \centering
\begin{tikzpicture}[line cap=round,line join=round,>=triangle 45,x=0.5cm,y=0.5cm]
\draw [line width=1pt] (7,9)-- (7,2);
\draw [line width=1pt] (7,9)-- (14,9);
\draw [line width=1pt] (14,9)-- (14,2);
\draw [line width=1pt] (14,2)-- (7,2);

\node at (19,8.6) {$M_0$};
\node at (25,8.6) {$M_1$};
\node at (31,8.6) {$M_2$};

\node at (19,5) {$\begin{bmatrix} 0 & 0 & 0 & 1 & 0 & 0 \\ 0 & 0 & 0 & 0 & 0 & 0 \\ 0 & 0 & 1 & 0 & 0 & 0 \\ 0 & 0 & 0 & 0 & 0 & 0 \\ 0 & 0 & 0 & 0 & 0 & 0 \\ 0 & 1 & 0 & 0 & 0 & 0 \end{bmatrix}$};
\node at (25,5) {$\begin{bmatrix} 0 & 0 & 0 & 0 & 0 & 0 \\ 1 & 0 & 0 & 0 & 0 & 0 \\ 0 & 0 & 0 & 0 & 0 & 0 \\ 0 & 0 & 0 & 0 & 1 & 0 \\ 0 & 0 & 0 & 0 & 0 & 0 \\ 0 & 0 & 0 & 0 & 0 & 0 \end{bmatrix}$};
\node at (31,5) {$\begin{bmatrix} 0 & 0 & 0 & 0 & 0 & 0 \\ 0 & 0 & 0 & 0 & 0 & 0 \\ 0 & 0 & 0 & 0 & 0 & 0 \\ 0 & 0 & 0 & 0 & 0 & 0 \\ 0 & 0 & 0 & 0 & 0 & 1 \\ 0 & 0 & 0 & 0 & 0 & 0 \end{bmatrix}$};
\begin{scriptsize}
\draw [fill=red] (7.76,7.08) circle (2.5pt);
\draw [fill=black] (10.46,8.4) circle (2.5pt);
\draw [fill=red] (12.18,5.58) circle (2.5pt);
\draw [fill=black] (8.72,2.74) circle (2.5pt);
\draw [fill=black] (9.62,6.28) circle (2.5pt);
\draw [fill=blue] (13.06,3.3) circle (2.5pt);
\node at (9.1,3.1) {$p$};
\end{scriptsize}
\end{tikzpicture}
    \caption{A realization of the point processes $\{\mathscr{M}_k\}_{k\ge 0}$ and corresponding matrices $\{M_k\}_{k\ge 0}$. Black, red, and blue points are points in $\mathscr{M}_0$, $\mathscr{M}_1$, and $\mathscr{M}_2$ respectively. $\mathscr{M}_k=\varnothing$ for $k\ge 3$. 
    The row number $i$ counts the number of points weakly above $p$, while the column number $j$ counts the number weakly to the left of $p$. Thus $M_0(6,2)=1$ for the point $p$ marked in the figure. }
    \label{fig:mkmatrix}
\end{figure}
Notice that almost surely none of the $n$ points of the Poisson point processes will share an $x$ or $y$ coordinate. Thus almost surely, $M_k$'s are partial permutation matrices with $\# M_k=\#\mathscr{M}_k$. After constructing the sequence of matrices $M_k$, of which only the matrices $M_0,\dots,M_N$ will be not identically zero we start building up random partitions. Let $P_0,Q_0$ be the pair of skew tableaux of shape $\nu/\nu$ (i.e. with no labeled cells) and construct, through the correspondence \eqref{eq:sagan_stanley}, the sequence of pairs of partial tableaux $(P_i,Q_i)$, for $i=0,\ldots,N$ such that
\begin{equation*} 
    (P_{i},Q_{i},M_{N-i}) \longleftrightarrow (P_{i+1},Q_{i+1}).
\end{equation*}
We set $(P,Q)=(P_{N+1}, Q_{N+1})$ and it is clear, by construction that this is a pair of standard tableau with the same shape, which we denote by $\lambda/\rho$, where 
\begin{equation} \label{eq:size_partitions}
    |\lambda/\rho| = n  \qquad \text{and} \qquad |\rho| = |\nu| + \sum_{k \ge 0} k \cdot (\# M_k).  
\end{equation}
The construction just described, along with the fact that the Sagan-Stanley correspondence is a bijection, proving the following proposition.
\begin{proposition}
    The random skew partition $\lambda/\rho$ constructed through the procedure described above is distributed according to the cylindric Plancherel measure $\mathbb{P}_{\mathsf{cPlan}(\gamma)}$.
\end{proposition}

The above sampling scheme allows us to derive exponential moment bounds for the size of random partitions $\lambda$ and $\rho$, where $\lambda/\rho$ is distributed according to the cylindric Plancherel measure. We report them in the following proposition.

\begin{proposition}
Fix $\gamma>0$. Suppose $\lambda/\rho \sim \Pr_{\mathsf{cPlan}(\gamma)}$. Then
    \begin{equation} \label{eq:law_size_skew_shape}
        |\lambda / \rho| \sim \mathrm{Poi}(\gamma^2/(1-q)),
    \end{equation}
    {For all $\theta\in (0,\log q^{-1})$, there exists a constant $\Con=\Con(q,\theta)>0$ such that \begin{align}
        \label{eq:expmombd}
        \Ex[e^{\theta |\rho|}]\le \Con \exp\left(\frac{\gamma^2qe^\theta}{1-qe^{\theta}}\right).
    \end{align}}
\end{proposition}
\begin{proof}
   {Note that \eqref{eq:law_size_skew_shape} is a direct consequence of \eqref{eq:n_Poisson} and the first equality in \eqref{eq:size_partitions}. Let us focus on proving \eqref{eq:expmombd}. From the second equality in \eqref{eq:size_partitions}, we deduce that
    $|\rho|=|\nu|+\sum_{k\ge 0}k\mathcal{A}_k$
    where $\nu \sim \mathbb{P}_{\mathsf{vol}}$ and $\mathcal{A}_k\sim \operatorname{Poi}(q^k\gamma^2)$ are all independent.  Using the explicit moment generating function for poisson random variables for $\theta\in (0,\log q^{-1})$, we obtain
    \begin{equation} \label{eq:generating_f_G}
        \mathbb{E} \left( e^{\theta\sum_{k\ge 1}k\, \mathscr{A}_k} \right) = \prod_{k\ge 1} \Ex[e^{\theta k\mathcal{A}_k}] = e^{\gamma^2\left( \frac{qe^{\theta}}{1-qe^{\theta}} - \frac{q}{1-q} \right)} \le e^{\frac{q\gamma^2 e^{\theta}}{1-qe^{\theta}}}.
    \end{equation}
    Since $\Pr(|\nu|=k)=(q;q)_\infty \mathsf{p}_k q^k$ and $\mathsf{p}_k\le e^{\Con \sqrt{k}}$ from \eqref{eq:pnbd}, for any $\theta\in (0,\log q^{-1})$ we have $\Ex[e^{\theta|\nu|}]\le \Con(q,\theta)<\infty.$ Combining this with \eqref{eq:generating_f_G}, we arrive at \eqref{eq:expmombd}.}  
\end{proof}

\subsection{Exact formulas for the height function of $q$-PNG}\label{sec:2.4}

The connection between the cylindric Plancherel measure and the height function of the $q$-Deformed Polynuclear Growth with droplet initial condition can be leveraged to write down explicit formulas for the distribution of $\mathfrak{h}$. Formulas analogous to those described below were found in \cite{aggarwal_borodin_wheeler_tPNG} through a similar matching with the Poissonized Plancherel measure.

As found in the seminal paper \cite{borodin2007periodic} the cylindric Plancherel measure becomes, after a certain random shift, a determinantal point process with an explicit correlation kernel. Given $\zeta>0$, we say that a random variable $S_\zeta$ has $\mathrm{Theta}(q,\zeta)$ distribution if
\begin{equation}\label{eq:S_distribution}
    \mathbb{P}(S_\zeta=k) = \frac{q^{k^2/2} \zeta^k}{\vartheta(q,\zeta)}  \quad \mbox{ for }  {k\in \Z},
\end{equation}
where the normalization constant can be evaluated as the Jacobi triple product
\begin{equation*}
    \vartheta(q,\zeta) = (q,-\sqrt{q} \zeta, -\sqrt{q}/\zeta;q)_\infty.
\end{equation*}
Here we are using the convention that $(a_1, \dots,a_m;q)_\infty = (a_1,;q)_\infty  \cdots (a_m;q)_\infty$.

\begin{proposition}[\cite{borodin2007periodic}] \label{prop:shift_mixed_det}
   Fix any $\gamma,\zeta>0$. Let $\lambda / \rho \sim \mathbb{P}_{\mathsf{cPlan}(\gamma)}$ and $S_\zeta\sim \mathrm{Theta}(q,\zeta)$ be independent random variables. Consider the point process
    \begin{equation*}
        \mathfrak{S}(\lambda,S_\zeta) = \left\{ \lambda_i -i + S_\zeta + \frac{1}{2} : i=1,2,\dots  \right\}.
    \end{equation*}
    Then $\mathfrak{S}(\lambda,S_\zeta)$ is a determinantal point process on $\mathbb{Z}':=\mathbb{Z}+1/2$ with correlation kernel
   \begin{align}\label{skernel}
    \mathsf{K}_{\zeta,\gamma}(a,b) := \sum_{\ell\in \mathbb{Z}'}\frac{1}{1+\zeta^{-1}q^{\ell}}J_{a+\ell}\left(\frac{2\gamma}{1-q}\right)J_{b+\ell}\left(\frac{2\gamma}{1-q}\right),
\end{align}
    where $J_m$ are the Bessel functions of the first kind.
\end{proposition}
By the theory of determinantal point processes \cite{Soshnikov2000, Johansson2005lectures, Borodin2009}, we have the following:
\begin{corollary}
 Fix any $\gamma,\zeta>0$.  Let $\lambda/\rho \sim \mathbb{P}_{\mathsf{cPlan}(\gamma)}$ and $S_\zeta\sim \mathrm{Theta}(q,\zeta)$ be independent random variables. Then, for any $s\in \mathbb{Z}$ we have
    \begin{equation}\label{eq:cor2}
        \mathbb{P}(\lambda_1+S_\zeta \le s) = \det \left( I - \mathsf{K}_{\zeta,\gamma} \right)_{\ell^2(s+\frac{1}{2}, s+\frac{3}{2},\dots)}.
    \end{equation}
\end{corollary}
For the next result, we define the function $\mathsf{F}_q: [0,\infty)\to [0,\infty)$ as
\begin{equation}\label{def:fq}
    \mathsf{F}_q(\zeta) := \prod_{k \ge 0} \frac{1}{1+\zeta q^{k}}.
\end{equation}
\begin{theorem}\label{l:fred} Fix any $\theta,\zeta>0$. Let $\mathfrak{h}$ be the height function of the $q$-PNG with intensity $\Lambda=\theta^2/(1-q)$ and droplet initial condition. We have
\begin{align}\label{e:fred}
\Ex\left[ \mathsf{F}_q(\zeta q^{-\h(0,t)})\right]  =\det(I-\mathsf{K}_{\zeta/\sqrt{q},\theta t/ \sqrt{2}})_{\ell^2(\mathbb{N}')}.
\end{align}
where $\mathbb{N}':=\{1/2,3/2,5/2,\ldots\}$.
\end{theorem}
\begin{proof}
Let $\chi\sim q$-Geo$(q)$ and $S_\zeta\sim$Theta$(\zeta,q)$ independent of $q$-PNG. By \cite[Eq, (5.4)]{IMS_KPZ_free_fermions} 
\begin{equation*}
    \mathbb{P}(\chi + S_\zeta \le n) = \frac{1}{(-\zeta q^{n+1/2};q)_\infty} = \mathsf{F}_q(\zeta q^{n+1/2}),
\end{equation*}
for any $n\in \mathbb{Z}$ and hence we have    
\begin{align}\label{eq:iden1}
    \Pr(\h(0,t) + \chi+S_\zeta \le 0) = \Ex\left[ \mathsf{F}_q(\zeta q^{1/2-\h(0,t)})\right] .
\end{align}
By \Cref{thm:matching_qPNG_cylindric_plancherel} we know $\h(0,t)+\chi$ is equivalent in distribution to the first row $\lambda_1$ of a partition $\lambda \sim \mathbb{P}_{\mathsf{cPlan}(\theta t/\sqrt{2})}$. Therefore, $\h(0,t)+\chi+S_\zeta$ is equal in law to $\lambda_1+S_\zeta$ whose probability distribution is given in \Cref{prop:shift_mixed_det}. Taking $\zeta \mapsto \zeta/\sqrt{q}$, $\gamma\mapsto \theta t/\sqrt{2}$ and $s\mapsto 0$ in \eqref{eq:cor2}, in view of \eqref{eq:iden1}, we arrive at \eqref{e:fred}. This completes the proof.
\end{proof}

We next give another formula for the height function of the $q$-PNG stemming from a relation between the periodic Schur and the (usual) Schur measure observed in \cite{IMS_KPZ_free_fermions}. This formula was also derived in \cite{aggarwal_borodin_wheeler_tPNG} as a special case of a more general result by Borodin \cite{borodin2016stochastic_MM}.
\begin{theorem}\label{cor:iden}
   Fix $\theta,\zeta>0$. Let $\mathfrak{h}$ be the height function of the $q$-PNG with intensity $\Lambda=\theta^2/(1-q)$ and let $\chi \sim q$-$\mathrm{Geo}(q)$, $S_\zeta\sim\mathrm{Theta}(q,\zeta)$ be independent random variables, independent of $\h$ as well. Then for $s\in \Z$ we have
    \begin{equation} \label{eq:S+h=expextation}
        \mathbb{P}(\mathfrak{h}(0,t)+\chi+S_\zeta \le s) = \mathbb{E}_{\mathsf{Plan}(\theta t/\sqrt{2}(1-q))} \bigg[ \prod_{i \ge 1} \frac{1}{1+\zeta q^{s+i-\lambda_i}} \bigg].
    \end{equation}
\end{theorem}
\begin{proof}
    It was shown in \cite[Corollary 4.3 and Remark 4.4]{IMS_KPZ_free_fermions} that
    \begin{equation} \label{eq:S+lambda1=expextation}
        \mathbb{P}(\bar{\lambda}_1+S_\zeta\le s) = \mathbb{E} \left( \prod_{i \ge 1} \frac{1}{1+\zeta q^{s+i-\tilde{\lambda}_i}} \right),
    \end{equation}
    where in the left-hand side the partition $\bar{\lambda}$ is taken with respect to the periodic Schur measure \eqref{eq:periodic_schur}, while in the right-hand side the partition $\tilde{\lambda}$ obeys a Schur measure (i.e. \eqref{eq:periodic_schur} with $q=0$) with specializations in geometric progression $(a_1,qa_1,q^2a_1, \dots, a_N,qa_N,q^2a_N,\dots)$ and $(b_1,qb_1,q^2b_1, \dots, b_T,qb_T,q^2b_T,\dots)$. We now take the scaling \eqref{eq:scaling} after setting $a_i=b_j=\sqrt{a}$ and $x=0$ to deduce \eqref{eq:S+h=expextation} from \eqref{eq:S+lambda1=expextation}. The limit of the left-hand side was already taken in \eqref{thm:matching_qPNG_cylindric_plancherel}. Considering the limits
    \begin{equation*}
    \begin{split}
        s_\lambda(a_1,qa_1,q^2a_1, \dots, a_N,qa_N,q^2a_N,\dots) &\xrightarrow[\varepsilon \downarrow 0]{} \left( \frac{\theta t}{\sqrt{2} (1-q)} \right)^{|\lambda|} \frac{f^{\lambda}}{|\lambda|!},
        \\
        s_{\lambda}(b_1,qb_1,q^2b_1, \dots, b_T,qb_T,q^2b_T,\dots) &\xrightarrow[\varepsilon \downarrow 0]{} \left( \frac{\theta t}{\sqrt{2}(1-q)} \right)^{|\lambda|} \frac{f^{\lambda}}{|\lambda|!},
    \end{split}
    \end{equation*}
    we see that the right-hand side of \eqref{eq:S+lambda1=expextation} converges to the right-hand side of \eqref{eq:S+h=expextation}.  
\end{proof}

\section{Upper-Tail LDP for $q$-PNG} 
\label{sec:upper_tail}

The goal of this section is to prove the upper-tail Large Deviation Principle for the height function $\h$ of $q$-PNG. In \Cref{sec:3.1}, we elaborate on the brief sketch given in \Cref{sec:1.2.2} and reduce our proof to establishing asymptotics and estimates for the leading term and higher-order term respectively. In \Cref{sec.bessel}, we utilize Bessel function tail behavior to extract various estimates related to the trace of the kernel and its derivatives. The leading term and the higher-order term are analyzed in  \Cref{sec.trace} and \Cref{sec.ho} respectively.

\subsection{An outline of proof of \Cref{thm:upper_tail}}  \label{sec:3.1}

 In this subsection, we present the main argument for the proof of the Large Deviation Principle for the upper-tail of the height function $\mathfrak{h}(0,t)$. For this, we define the function 
 \begin{equation}\label{def:wp}
     \Upsilon(p) := 4 \sinh(p/2), \qquad \text{for } p>0,
 \end{equation}
 whose Legendre transform is the rate function $\Phi_+$ defined in \eqref{eq:Phi+}. Indeed, it can be easily checked that
 \begin{equation}\label{eq:phi_W}
     \Phi_+(\mu) = \sup_{p>0} \big( p \mu - \Upsilon(p) \big).
 \end{equation}
We are now ready to state the main theorem of this section, which computes the large time asymptotics of the moment-generating function of $\mathfrak{h}$.
\begin{theorem}\label{t:uldp} Fix any $p>0$. Let $\mathfrak{h}$ be the height function of the $q$-PNG with intensity $\Lambda=2(1-q)$ and droplet initial condition.  We have
    \begin{align}\label{e:lyap}
        \lim_{t\to \infty} \frac1t \log \Ex \left[ e^{p\h(0,t)} \right] = \Upsilon(p),
    \end{align}
    where $\Upsilon(\cdot)$ is defined in \eqref{def:wp}
\end{theorem}

Let us complete the proof of \Cref{thm:upper_tail} assuming the \Cref{t:uldp}.

\begin{proof}[Proof of \Cref{thm:upper_tail}] Deriving large deviation rate functions from the asymptotic limit of moment generating function, i.e. from the Lyapunov exponent, is common practice and in the context of the KPZ equation, this idea has been worked out in \cite{dt21}. Indeed appealing to a general result,\cite[Proposition 1.13]{gl20}, we obtain that the upper-tail LDP of $\h(0,t)$ is given by the Legendre transform of $\Upsilon(\cdot)$, which is $\Phi_+$ due to the relation in \eqref{eq:phi_W}. This proves the upper-tail LDP.
\end{proof}

We now give the proof of \Cref{t:uldp} elaborating on the strategy outlined in \Cref{sec:1.2.2}. 

\begin{proof}[Proof of \Cref{t:uldp} modulo \Cref{p:trace} and \Cref{p:ho} below] 

We will need to define some parameters which we will use throughout the rest of the section. For $p>0$ and $q\in (0,1)$, we set 
\begin{equation} \label{eq:s_n_alpha}
    s=\frac{p}{\log q^{-1}}>0,
    \qquad
    n=\lfloor s \rfloor+1,
    \qquad
    \alpha=s-\lfloor s \rfloor.
\end{equation} 
Set $\delta:=(\Upsilon(p)-2p)/4>0$ (in fact any $\delta\in (0,\Upsilon(p)-2p)$ will work). For each $t>0$, we define 
\begin{align}\label{def:gamma}
    \agamma=\agamma(t):=e^{-(\Upsilon(p)-\delta)t/s}.
\end{align}

We are going to express the moment generating function of the height function $\mathfrak{h}(0,t)$ using a  ``Fubini trick" from  \cite[Lemma 1.8, Proposition 2.2]{dz1} as
\begin{equation}\label{eta0}
\begin{split}
    \Ex[e^{p\h(0,t)}]  =\Ex[q^{-s\h(0,t)}] 
    =\frac{(-1)^n\int_0^{\infty} \zeta^{-\alpha}\frac{\diff^n}{\diff \zeta^n}\Ex[\mathsf{F}_q(\zeta q^{-\h(0,t)})] \diff \zeta}{(-1)^n\int_0^{\infty} \zeta^{-\alpha}\mathsf{F}_q^{(n)}(\zeta) \diff \zeta},
\end{split}
\end{equation}
where $\mathsf{F}_q$ is defined in \eqref{def:fq}.
Let us analyze the right-hand side of \eqref{eta0}. First, from \cite[Proposition 2.2 (b)]{dz1} we know $(-1)^n\int_0^{\infty}\zeta^{-\alpha}\mathsf{F}_q^{(n)}(\zeta) \diff \zeta$ is strictly positive, finite and free of $t$. Thus we may ignore the denominator while computing $\frac1t\log$ limit of the right-hand side of \eqref{eta0}. Moving to the numerator, we are going to split the integration over two intervals as 
\begin{equation} \label{eq:split_integral}
    (-1)^n\int_0^{\infty} \zeta^{-\alpha}\frac{\diff^n}{\diff \zeta^n}\Ex[\mathsf{F}_q(\zeta q^{-\h(0,t)})] \diff \zeta = \left(  \int_0^{\agamma} + \int_{\agamma}^\infty \right) (-1)^n\zeta^{-\alpha}\frac{\diff^n}{\diff \zeta^n}\Ex[\mathsf{F}_q(\zeta q^{-\h(0,t)})] \diff \zeta, 
\end{equation}
where the parameter $\agamma$ was defined in \eqref{def:gamma}. The second integral in the right-hand side of \eqref{eq:split_integral} can be bounded, in absolute value, as follows.
\begin{equation*}
    \begin{split}
        \int_{\agamma}^{\infty} \left| \zeta^{-\alpha}\frac{\diff^n}{\diff \zeta^n}\Ex[\mathsf{F}_q(\zeta q^{-\h(0,t)})]  \right| \diff \zeta & =
        \int_{\agamma}^{\infty} \left|\zeta^{-\alpha}\Ex\left[q^{-n\h(0,t)}\mathsf{F}_q^{(n)}(\zeta q^{-\h(0,t)})\right]\right| \diff \zeta 
        \\ & \le \sup_{x>0} |x^{n}\mathsf{F}_q^{(n)}(x)|\cdot \int_{\agamma}^{\infty} \zeta^{-n-\alpha} \diff \zeta 
         = \frac{\agamma^{-s}}s\sup_{x>0} |x^{n}\mathsf{F}_q^{(n)}(x)|,
    \end{split}
\end{equation*}
where in the last line we have used that $s=n+\alpha-1$. 
By \cite[Proposition 2.2 (c)]{dz1}, we know that $\sup_{x>0} |x^{n}\mathsf{F}_q^{(n)}(x)|$ is finite (and independent of $t$), whereas, by the choice of $\agamma$ \eqref{def:gamma} we have $\agamma^{-s}=e^{(\Upsilon(p)-\delta)t}$. This shows that for any fixed $p>0$ we have 
\begin{equation} \label{eq:bound_tail}
    \left| \int_{\agamma}^\infty  \zeta^{-\alpha}\frac{\diff^n}{\diff \zeta^n}\Ex[\mathsf{F}_q(\zeta q^{-\h(0,t)})]d \zeta \right| \le C \, e^{(\Upsilon(p) - \delta) t} 
\end{equation}
for some constant $C=C(p)$. 

Let us now examine the first integral on the right-hand side of \eqref{eq:split_integral}. To analyze this remaining term, we use the exact representation of $\Ex[\mathsf{F}_q(\zeta q^{-\h(0,t)})]$ from \eqref{e:fred}, which allows us to write
\begin{equation}
    \label{eq:fred}
    \Ex[\mathsf{F}_q(\zeta q^{-\h(0,t)})]=\det(I-K_{\zeta,t})_{\ell^2(\mathbb{N}')}
\end{equation}
where the correlation kernel $K_{\zeta,t}$ equals $\mathsf{K}_{\zeta/\sqrt{q},\theta t/\sqrt{2}}$ defined in \eqref{skernel} with $\theta=\sqrt{2}(1-q)$ (as we work $q$-PNG with intensity $\Lambda=2(1-q)$). In other words, we set
\begin{equation}\label{kernel}
K_{\zeta,t}(a,b):=\mathsf{K}_{\zeta/\sqrt{q},t(1-q)}(a,b)=\sum_{\ell\in \Z'}v_{q,\ell}(\zeta)J_{a+\ell}(2t)J_{b+\ell}(2t), \quad v_{q,\ell}(\zeta):=\frac1{1+\zeta^{-1}q^{\ell+\frac12}},
\end{equation}
for $a,b\in \Z'$. For the remainder of this section, we shall always work with $\ell^2(\mathbb{N}')$ space, and drop it from the Fredholm determinant notation in \eqref{eq:fred}. We now state two propositions, whose proofs are postponed to \Cref{sec.trace} and \Cref{sec.ho} below.

\begin{proposition}[Trace asymptotics]\label{p:trace} For each $t>0$, $\tr(K_{\zeta,t}):=\sum_{a\in \mathbb{N}'} K_{\zeta,t}(a,a)$ is differentiable at each $\zeta\in (0,1)$. For each $p>0$, we have
\begin{align} \label{eq:limit_W(p)}
    \lim_{t\to\infty} \frac1t\log \left[(-1)^{n+1}\int_0^{\agamma} \zeta^{-\alpha}\frac{\diff^n}{\diff\zeta^n}\tr(K_{\zeta,t})\diff \zeta\right] = \Upsilon(p).
\end{align}
\end{proposition}

\begin{proposition}[Higher-order estimates] \label{p:ho} For each $p>0$, there exists a constant $\Con=\Con(p)>0$ such that for all $t$ large enough we have
    \begin{align}\label{e.ho}
        \int_0^{\agamma} \zeta^{-\alpha}\left|\frac{\diff^n}{\diff \zeta^n}\big[\det(I-K_{\zeta,t})+\operatorname{tr}(K_{\zeta,t})\big]\right| \diff\zeta \le \Con \, e^{\Upsilon(p)t-\tfrac1\Con t}.
    \end{align}
\end{proposition}
\noindent Making use of the identity in \eqref{eq:fred}, we write the first integral on the right-hand side of \eqref{eq:split_integral} as
\begin{equation*}
\begin{split}
      (-1)^n\int_0^{\agamma}\hspace{-0.1cm} \zeta^{-\alpha}\frac{\diff^n}{\diff\zeta^n}\Ex[\mathsf{F}_q(\zeta q^{-\h(0,t)})] \diff \zeta =& (-1)^{n+1}\int_0^{\agamma} \zeta^{-\alpha}\frac{\diff^n}{\diff \zeta^n}\tr(K_{\zeta,t})\diff\zeta \\
      & +(-1)^n\int_0^{\agamma} \zeta^{-\alpha}\frac{\diff^n}{\diff\zeta^n}\big[\det(I-K_{\zeta,t})+\operatorname{tr}(K_{\zeta,t})\big] \diff \zeta    
\end{split}
\end{equation*}
and employing the convergence result \eqref{eq:limit_W(p)} along with the bound \eqref{e.ho} we get
\begin{align*}
   \lim_{t\to\infty} \frac1t\log \left[(-1)^n\int\limits_0^{\agamma} \zeta^{-\alpha}\frac{\diff^n}{\diff \zeta^n}\Ex\left[\mathsf{F}_q(\zeta q^{-\h(0,t)})\right] \diff\zeta\right]=\Upsilon(p).
\end{align*}
Combining the previous limit with the bound \eqref{eq:bound_tail} we complete the proof of \eqref{e:lyap}.
\end{proof}

\subsection{Bessel estimates}\label{sec.bessel} In this section we collect various estimates related to Bessel functions that will be useful in our later analysis. The key reason why $\Phi_+$, defined in \eqref{eq:Phi+}, appears as the upper-tail rate function is that it governs the tail asymptotics of Bessel functions. Indeed, from classical estimates for Bessel functions (see Lemma 9.1 and Eq.~(9.17) in \cite{baik2016-book} for example) we have the following result.
\begin{lemma}\label{l.bessel} For each $n\in \Z_{>0}$ and for all $0<2t<n$ we have
	\begin{align}\label{e:jest}
		J_{n}^2(2t)\le \frac{\pi}{8\sqrt{n^2-4t^2}} e^{-t\Phi_+(\tfrac{n}{t})},
	\end{align}
where $J_m$ are the Bessel functions of the first kind. Furthermore, for each fixed $u>2$ we have
 \begin{align}\label{e:jest2}
		\lim_{t\to \infty} 2\pi \cdot t \sqrt{u^2-4} \cdot  J_{\lfloor ut \rfloor}^2(2t)e^{t\Phi_+(u)}=1.
	\end{align}
\end{lemma}
For each $v>0$, we define the function
\begin{align}\label{def:U}
    U_v(x):=vx-\Phi_+(x).
\end{align}
The following lemma collects useful properties of $\Phi_+$, $\Upsilon$, and $U_v$ functions.
\begin{lemma}[Properties of $\Phi_+$, $\Upsilon$  and $U_v$ functions]\label{l:propphi} We have the following.
\begin{enumerate}[label=(\alph*), leftmargin=18pt]
    \item $\Phi_+, \Phi'_+$ are strictly increasing on $[2,\infty)$ with $\Phi_+'(x)=2\log\frac{x+\sqrt{x^2-4}}{2}$ and $\Phi_+''(x)=\frac{2}{\sqrt{x^2-4}}$.
    \item $U_v(x)$ has a unique maximizer on $[2,\infty)$ given by $x_v^*:=2\cosh(v/2)$ with maximum $$U_v(x_v^*)=\Upsilon(v)=4\sinh(v/2).$$
    \item $\Upsilon(v)/v$ is strictly increasing with
    \begin{align*}
        \lim_{v\downarrow 0} \frac{\Upsilon(v)}{v}=2.
    \end{align*}
    \item For all $v>0$, $v x_v^*=2v\cosh(v/2)>\Upsilon(v)$.
\end{enumerate}
\end{lemma}
The above lemma can be checked easily and hence its proof is skipped.

In our analysis in subsequent sections, we shall often encounter the double sum $$\sum_{a\in \mathbb{N}'}\sum_{\ell\in \mathbb{Z}'}J_{a+\ell}^2(2t) e^{v\ell}$$ 
for $v>0$, which is related to the trace of the kernel $K_{\zeta,t}$ in \eqref{kernel} and its derivative (which we will define in the next subsection). We shall now devote a few lemmas to understanding the asymptotics as $t\to\infty$ of the above double sum restricted to various subsets of $\mathbb{N}'\times \mathbb{Z}'$.

\begin{lemma}\label{l:jupbd1} Fix any $D\in \R$ and $v>0$. For each $t>0$, we have
\begin{align*}
    \sum_{a\in \mathbb{N}'}\sum_{\ell\in \mathbb{Z}': a+\ell\le D}J_{a+\ell}^2(2t) e^{v\ell} \le \frac{e^{Dv+2v}}{(e^v-1)^2}.
\end{align*}
\end{lemma}
\begin{proof} Recall that for integers $n>0$, we have $J_{-n}(x)=(-1)^nJ_n(x)$. Combining this with the uniform estimate for nonnegative order Bessel functions from \cite{landau}, we have that $\sup_{x\in \R} |J_{m}^2(x)| \le 1$ for all $m\in \Z$. Hence,
    \begin{align*}
    \sum_{a\in \mathbb{N}'}\sum_{\ell\in \mathbb{Z}': a+\ell\le D}J_{a+\ell}^2(2t) e^{v\ell} \le \sum_{a\in \mathbb{N}'}\sum_{\ell\in \mathbb{Z}': a+\ell\le D} e^{v\ell}.
\end{align*}
The required estimate now follows from geometric series estimates.
\end{proof}

\begin{lemma}\label{l:jupbd2} Let $r>0$ and fix any $(2+r)<\gamma_1<\gamma_2\le \infty$. Fix any $v\ge 0$ and $t\ge 1$. There exists a $\Con=\Con(r,v)>0$ such that 
	\begin{align*}
		\sum_{a\in \mathbb{N}'}\sum_{\ell\in \mathbb{Z}': a+\ell\in [\gamma_1t,\gamma_2t]}J_{a+\ell}^2(2t) e^{v\ell} \le \Con \cdot  e^{t \cdot \max_{x\in [\gamma_1,\gamma_2]}U_{v}(x)},
	\end{align*}
	where the function $U_v$ is defined in \eqref{def:U}.
\end{lemma}
\begin{proof} Using the estimate from \eqref{e:jest} we get
	\begin{equation}
 \begin{split}
        &\sum_{a\in \mathbb{N}'}\sum_{\ell\in \mathbb{Z}': a+\ell\in [\gamma_1t,\gamma_2t]}J_{a+\ell}^2(2t) e^{v\ell} 
        \\
        &\le \frac{\pi}{8t\sqrt{4r+r^2}}\sum_{a\in \mathbb{N}'} e^{-va}\sum_{\ell\in \mathbb{Z}': a+\ell\in [\gamma_1t,\gamma_2t]} e^{-t\Phi_+(\tfrac{a+\ell}{t})+v(a+\ell)} \\ 
        & \le \frac{\pi}{8t\sqrt{4r+r^2}}\frac{1}{1-e^{-v}}\sum_{k\in \mathbb{Z} \cap [\gamma_1t, \gamma_2t]} e^{t U_v(k/t)}.
  \label{e:jupbd}
  \end{split}
  \end{equation}
  Now we can estimate the sum in the right-hand side of \eqref{e:jupbd} using Laplace's method, which can be applied since $U_v$ is concave as shown in \Cref{l:propphi}. In case $\gamma_1 \le x_v^* \le \gamma_2$, one can easily show that
  \begin{equation*}
      \sum_{k\in \mathbb{Z} \cap [\gamma_1t, \gamma_2t]} e^{t U_v(k/t)} \le \Con \sqrt{t } e^{t U_v(x_v^*)}.
  \end{equation*}
  On the other hand, when $\gamma_2<x_v^*$, we have
  \begin{equation*}
    \begin{split}
        \sum_{k\in \mathbb{Z} \cap [\gamma_1t, \gamma_2t]} e^{t U_v(k/t)} &\le  e^{t U_v(\gamma_2)} \sum_{k\in \mathbb{Z} \cap [\gamma_1t, \gamma_2t]} e^{t U_v'(\gamma_2) (k/t -\gamma_2)} 
        \\
        &
        \le \Con \, e^{t U_v(\gamma_2)},
    \end{split}
    \end{equation*}
    while, if $x_v^* \le \gamma_1$, we have
    \begin{equation*}
    \begin{split}
        \sum_{k\in \mathbb{Z} \cap [\gamma_1t, \gamma_2t]} e^{t U_v(k/t)} &\le  e^{t U_v(\gamma_1)} \sum_{k\in \mathbb{Z} \cap [\gamma_1t, \gamma_2t]} e^{t U_v'(\gamma_1) (k/t -\gamma_1)} 
        \\
        &
        \le \Con \, e^{t U_v(\gamma_1)}.
    \end{split}
    \end{equation*} 
	Plugging these estimates back in \eqref{e:jupbd} yields the desired result. 
\end{proof}

\begin{lemma}\label{l:jupbd3} Fix any $r,v>0$ and $t\ge 1$. Fix any $\sigma\in (2+r,\infty]$. There exists a $\Con=\Con(r,v)>0$ such that
	\begin{align}\label{e:jupbd3}
		\sum_{a\in \mathbb{N}'}\sum_{\ell\in \mathbb{Z}': \ell\le \sigma t}J_{a+\ell}^2(2t) e^{v\ell} \le \Con  \cdot  e^{t \cdot U_{v}\big(\min \{\sigma,x_v^*\}\big)}.
	\end{align}
	where the function $U_v$ is defined in \eqref{def:U}.
\end{lemma}
\begin{proof} Recall that $\Phi_+$ is increasing and $\Phi_+(2)=\Phi_+'(2)=0$. We can thus choose a number $\xi>0$ such that $\xi \le \min\{r,x_v^*-2\}$ and $v-\Phi_+'(2+\xi) >0$. Take  $0 <\delta \le \min\{v^{-1}\xi(v-\Phi_+'(2+\xi)),\xi\}$. Note that as $\Phi_+'$ is increasing we have
	\begin{align*}
		\Phi_+(2+\xi)=\int_2^{2+\xi} \Phi_+'(y) \diff y \le \xi\Phi_+'(2+\xi). 
	\end{align*}
 Thus for any $v>0$, we have
\begin{align*}
	U_v(2+\xi)=v(2+\xi)-\Phi_+(2+\xi) \ge v(2+\xi)-\xi\Phi_+'(2+\xi) \ge v(2+\delta).
\end{align*}
Based on this $\delta$, we split the double sum in \eqref{e:jupbd3} into three parts given by the following three index sets:
\begin{align*}
    T_1 & :=\{(a,\ell)\in \mathbb{N}'\times \mathbb{Z}'\mid a+\ell\le (2+\delta)t\}, \\
    T_2 & :=\{(a,\ell)\in \mathbb{N}'\times \mathbb{Z}'\mid a+\ell> (2+\delta)t, \ell \le (2+\delta)t\}, \\
    T_3 & :=\{(a,\ell)\in \mathbb{N}'\times \mathbb{Z}'\mid a+\ell> (2+\delta)t,  (2+\delta)t< \ell\le \sigma t\}.
\end{align*}
We write
\begin{align*}
    \sum_{a\in \mathbb{N}'}\sum_{\ell\in \mathbb{Z}': \ell\le \sigma t}J_{a+\ell}^2(2t) e^{v\ell}=\left(\sum_{(a,\ell)\in T_1}+\sum_{(a,\ell)\in T_2}+\sum_{(a,\ell)\in T_3}\right)J_{a+\ell}^2(2t) e^{v\ell}.
\end{align*}
For the first sum, using \Cref{l:jupbd1} we have
	\begin{align*}
	\sum_{(a,\ell)\in T_1}J_{a+\ell}^2(2t) e^{v\ell} \le \Con \cdot e^{v(2+\delta)t}.
\end{align*}
For the second sum using \eqref{e:jest} we get
\begin{align}
	\sum_{(a,\ell)\in T_2}J_{a+\ell}^2(2t) e^{v\ell}& \le \frac{\pi}{8t\sqrt{4\delta+\delta^2}}\sum_{\ell\in \mathbb{Z}': \ell\le (2+\delta) t} e^{v\ell}\sum_{a\in \mathbb{N}', a+\ell\ge (2+\delta)t} e^{-t\Phi_+(\tfrac{a+\ell}{t})}. \label{e.387}
\end{align}
Using the fact that $\Phi_+$ is convex we see that for $a+\ell \ge (2+\delta)t$ we have
\begin{align*}
    t\Phi_+(\tfrac{a+\ell}{t}) \ge t\Phi_+(2+\delta)+((a+\ell)-2t-\delta t)\Phi_+'(2+\delta).
\end{align*}
This forces, after a change of variable $k=a+\ell$ in \eqref{e.387},
\begin{align*}
    \mbox{r.h.s.~of \eqref{e.387}} \le \Con\cdot e^{-t\Phi_+(2+\delta)} \sum_{\ell\in \mathbb{Z}': \ell\le (2+\delta) t} e^{v\ell}\sum_{k\in \mathbb{N}'} e^{-k\Phi_+'(2+\delta)}. 
\end{align*}
where we recognize that both the above sums are just geometric series. As $\Phi_+'(2+\delta)>0$, overall we have 
\begin{align*}
    \sum_{(a,\ell)\in T_2}J_{a+\ell}^2(2t) e^{v\ell} \le  \Con\cdot  e^{t\cdot U_v(2+\delta)}.
\end{align*}
Finally, for the third sum we have
\begin{align}\label{e.lbd1}
	\sum_{(a,t)\in T_3}J_{a+\ell}^2(2t) e^{v\ell} & \le \frac{\pi}{8t\sqrt{4\delta+\delta^2}}\sum_{\ell\in \mathbb{Z}': (2+\delta)t\le \ell\le \sigma t} e^{v\ell}\sum_{a\in \mathbb{N}', a+\ell\ge (2+\delta)t} e^{-t\Phi_+(\tfrac{a+\ell}{t})}.
\end{align}
Using the estimate $\Phi_+(\tfrac{a+\ell}{t}) \ge  \Phi_+(\tfrac{\ell}{t})+\tfrac{a}{t}\Phi_+'(2+\delta)$ for the term in the exponent, which holds for $(a,\ell) \in T_3$, we bound the rightmost sum to get
\begin{align*}
    \mbox{r.h.s.~of \eqref{e.lbd1}} \le \Con \cdot \frac{1}{t}\sum_{\ell\in \mathbb{Z}': (2+\delta)t\le \ell\le \sigma t} e^{v\ell} e^{-t\Phi_+(\tfrac{\ell}{t})}.
\end{align*}
The last sum can be bounded in a similar manner as done in the proof of \Cref{l:jupbd2}, to yield that the above term is at most  $\Con\cdot e^{t\cdot U_v(\min\{\sigma,x_v^*\})}$.
Since $$U_v(\min\{\sigma,x_v^*\}) \ge U_v(2+\zeta) \ge \max\{U_v(2+\delta), v(2+\delta)\},$$ we see that the third sum dominates. This gives us the desired result.
\end{proof}

\subsection{Trace Analysis} \label{sec.trace}  The main goal of this section  is to prove \Cref{p:trace}. We first recall a few basic definitions and results from operator theory. 

For an operator 
 with kernel $T:\mathbb{Z}'\times \mathbb{Z}'\to \R$, we define its norm as
 \begin{align*}
\norm{T}:=\tr(\sqrt{T^*T}).
\end{align*}
We say $T$ is trace-class when $\norm{T}<\infty.$ The operator $T$ is said to be positive if
 $$\sum_{a,b\in \mathbb{Z}'} T(a,b)f(a)f(b) \ge 0$$
 for all $f:\mathbb{Z}'\to \R$. For a positive operator, we have $\norm{T}=\tr(T)$.

To prove that the kernels are positive and trace class we will often rely on the standard square root trick result which we recall for the reader's convenience.

\begin{lemma}\label{l:sqtrick} Consider kernel $R:\mathbb{N}' \times \mathbb{Z}' \to \mathbb{R}$, with $\sum_{a\in \mathbb{N}'}\sum_{\ell\in \mathbb{Z}'} R^2(a,\ell)<\infty.$ The operator $K$ with the kernel
\begin{align*}
    K(a,b):=\sum_{\ell\in \mathbb{Z}'} R(a,\ell)R(b,\ell)
\end{align*}
is positive and trace-class, with $\norm{K}=\tr(K)=\sum_{a\in \mathbb{N}'}\sum_{\ell\in \mathbb{Z}'} R^2(a,\ell).$
\end{lemma}

A continuous version of this result appears in \cite[Lemma 2.1]{dt21}. The proof in the reference can be adapted easily to show the above lemma and hence we do not report the details.

Recall the kernel $K_{\zeta,t}$ from \eqref{kernel}. A formal computation suggests that the `derivatives' of the kernel $K_{\zeta,t}$ are given by
\begin{equation}
\label{e:deriK}
    \begin{aligned}
    K_{\zeta,t}^{(n)}(a,b) & :=\sum_{\ell\in \mathbb{Z}'}\left( \frac{\diff^n}{\diff \zeta^n} v_{q,\ell}(\zeta)\right) J_{a+\ell}(2t)J_{b+\ell}(2t) \\ & = (-1)^{n+1}n!\sum_{\ell\in \mathbb{Z}'}\frac{q^{\ell+\frac12}}{(\zeta+q^{\ell+\frac12})^{n+1}}J_{a+\ell}(2t)J_{b+\ell}(2t).
\end{aligned}
\end{equation}
The following lemma shows the derivative of $K_{\zeta,t}^{(n-1)}$ is indeed $K_{\zeta,t}^{(n)}$ in the trace norm sense. For convenience, we write $K_{\zeta,t}^{(0)}:=K_{\zeta,t}$.
\begin{lemma}\label{l:deriK} The kernel $K_{\zeta,t}$ defines a positive trace class operator on $\ell^2(\mathbb{N}')$. For each $n\ge 1$
\begin{enumerate}[label=(\alph*),leftmargin=18pt]
    \item $(-1)^{n+1}K_{\zeta,t}^{(n)}$ defines a positive trace class operator on $\ell^2(\mathbb{N}')$.
    \item $K_{\zeta,t}^{(n-1)}$ is differentiable in $\zeta$ at each $\zeta>0$ in the trace norm, with derivative being equal to $K_{\zeta,t}^{(n)}$, i.e.,
    \begin{align*}
        \lim_{\zeta'\to\zeta} \left\|\frac{K_{\zeta',t}^{(n-1)}-K_{\zeta,t}^{(n-1)}}{\zeta'-\zeta}-K_{\zeta,t}^{(n)} \right\|=0.
    \end{align*}
\end{enumerate}
\end{lemma}

We now turn to the proof of \Cref{l:deriK}.

\begin{proof}[Proof of  \Cref{l:deriK}] From  \Cref{l:jupbd1,l:jupbd2} it follows that for each $t>0$ and $M>1$
  \begin{align}\label{e:Mbd}
    \sum_{a\in \mathbb{N}'}\sum_{\ell\in \mathbb{Z}'}J_{a+\ell}^2(2t) M^{\ell} = \sum_{a\in \mathbb{N}'}\sum_{\ell\in \mathbb{Z}':a+\ell\le 3t}J_{a+\ell}^2(2t) M^{\ell}+\sum_{a\in \mathbb{N}'}\sum_{\ell\in \mathbb{Z}':a+\ell>3t}J_{a+\ell}^2(2t) M^{\ell}<\infty.
\end{align}
This forces
\begin{align*}
    \sum_{a\in \mathbb{N}'}\sum_{\ell\in \mathbb{Z}'}\frac{q^{\ell+\frac12}J_{a+\ell}^2(2t)}{(\zeta+q^{\ell+\frac12})^{n+1}}<\infty, \quad \sum_{a\in \mathbb{N}'}\sum_{\ell\in \mathbb{Z}'}\frac{J_{a+\ell}^2(2t)}{1+\zeta^{-1}q^{\ell+\frac12}}<\infty,
\end{align*}
for each $\zeta,t>0$, $q\in (0,1)$ and $n\ge 1$. In view of  \Cref{l:sqtrick}, we have that $K_{\zeta,t}$ and $(-1)^{n+1}K_{\zeta,t}^{(n)}$ are positive and trace-class. We focus on proving the differentiability of $K_{\zeta,t}^{(n-1)}$. Towards this end, set $D_{\zeta',\zeta}:=\frac{K_{\zeta',t}^{(n-1)}-K_{\zeta,t}^{(n-1)}}{\zeta'-\zeta}-K_{\zeta,t}^{(n)}.$ Assume $\zeta'>\zeta$. We have
\begin{align*}
    D_{\zeta',\zeta}(a,b)=\sum_{\ell\in \mathbb{Z}'}J_{a+\ell}(2t)J_{b+\ell}(2t)\int_{\zeta}^{\zeta'}\frac{(\zeta'-r)}{2(\zeta'-\zeta)}\cdot\frac{(-1)^{n+2}(n+1)!q^{\ell+\frac12}}{(r+q^{\ell+\frac12})^{n+2}} \diff r.
\end{align*}
Applying \Cref{l:sqtrick} we see that 
\begin{align*}
    \left\|D_{\zeta',\zeta}\right\|=\sum_{a\in \mathbb{N}'}\sum_{\ell\in \mathbb{Z}'}J_{a+\ell}^2(2t)\int_{\zeta}^{\zeta'}\frac{(\zeta'-r)}{2(\zeta'-\zeta)}\cdot\frac{(n+1)!q^{\ell+\frac12}}{(r+q^{\ell+\frac12})^{n+2}}\diff r.
\end{align*}
Note that $\frac{(\zeta'-r)}{2(\zeta'-\zeta)}\le \frac12$. Thus, by \eqref{e:Mbd} with $M=q^{-n-1}$ we have 
\begin{align*}
    \left\|D_{\zeta',\zeta}\right\|\le (\zeta'-\zeta)\sum_{a\in \mathbb{N}'}\sum_{\ell\in \mathbb{Z}'}J_{a+\ell}^2(2t)\cdot (n+1)!q^{-(n+1)(\ell+\frac12)}<\infty.
    \end{align*}
    By Dominated Convergence Theorem we get that $\left\|D_{\zeta',\zeta}\right\|\to 0$ as $\zeta'\downarrow 0$. The proof for $\zeta'<\zeta$ is analogous. This completes the proof.
\end{proof}

We now come to the proof of \Cref{p:trace}.

\begin{proof}[Proof of \Cref{p:trace}]
    
 We are going to show that there exists a constant $C=C(p)>1$ such that for all $t$ large enough we have
    \begin{align} \label{eq:proof_traceorder}
        & \frac{1}{C\, t}\cdot e^{\Upsilon(p)t} \le (-1)^{n+1}\int_{0}^{\agamma} \zeta^{-\alpha}\frac{\diff^n}{\diff\zeta^n}\operatorname{tr}(K_{\zeta,t}) \diff\zeta \le C \, e^{\Upsilon(p)t}.
    \end{align}
In view of \Cref{l:deriK} we have
\begin{align*}
    \frac{\diff^n}{\diff \zeta^n}\operatorname{tr}(K_{\zeta,t})=\operatorname{tr}(K_{\zeta,t}^{(n)})=(-1)^{n+1}n!\sum_{a\in \mathbb{N}'}\sum_{\ell\in \mathbb{Z}'}\frac{q^{\ell+\frac12}J_{a+\ell}^2(2t)}{(\zeta+q^{\ell+\frac12})^{n+1}}.
\end{align*}
Pushing the integral inside the double sum we have 
    \begin{align}\label{eta1}
        (-1)^{n+1}\int_{0}^{\agamma} \hspace{-0.2cm}\zeta^{-\alpha}\frac{\diff^n}{\diff \zeta^n}\operatorname{tr}(K_{\zeta,t}) \diff \zeta= n!\sum_{a\in \mathbb{N}'}\sum_{\ell\in \mathbb{Z}'}q^{\ell+\frac12}J_{a+\ell}^2(2t)\int_{0}^{\agamma}\hspace{-0.2cm} \frac{\zeta^{-\alpha}}{(\zeta+q^{\ell+\frac12})^{n+1}} \diff \zeta.
    \end{align}
  Using the substitution $u=1/(1+\zeta^{-1}q^{\ell+\frac12})$ and setting $y_{\ell}=1/(1+\agamma^{-1}q^{\ell+\frac12})$ the integral in the right-hand side become
    \begin{align*}
        \int_{0}^{\agamma} \frac{\zeta^{-\alpha}}{(\zeta+q^{\ell+\frac12})^{n+1}} \diff \zeta & = q^{-(n+\alpha)(\ell+\frac12)}\int_0^{y_\ell} u^{-\alpha}(1-u)^{n+\alpha-1}\diff u. 
    \end{align*}
    Thus,
   \begin{equation}\label{eta11}
   \begin{aligned}
       & (-1)^{n+1}\int_{0}^{\agamma} \zeta^{-\alpha}\frac{\diff^n}{\diff\zeta^n}\operatorname{tr}(K_{\zeta,t}) \diff\zeta \\ & \hspace{2cm}= n!\sum_{a\in \mathbb{N}'}\sum_{\ell\in \mathbb{Z}'}e^{p(\ell+\frac12)}J_{a+\ell}^2(2t)\int_0^{y_\ell} u^{-\alpha}(1-u)^{n+\alpha-1}\diff u.
   \end{aligned}       
    \end{equation}  
We now seek to find an upper and lower bound for the right-hand side in the above equality. For the upper bound we extend the range of integration to $[0,1]$ and evaluate the integral, which we recognize to be a Beta function, to get
    \begin{align}\label{eta12}
        (-1)^{n+1}\int_{0}^{\agamma} \zeta^{-\alpha}\frac{\diff^n}{\diff\zeta^n}\operatorname{tr}(K_{\zeta,t}) \diff\zeta \le \Gamma(1-\alpha)\Gamma(n+\alpha)\sum_{a\in \mathbb{N}'}\sum_{\ell\in \mathbb{Z}'}e^{p(\ell+\frac12)}J_{a+\ell}^2(2t).
    \end{align}
Let us choose a constant $b=b(p)$ such that 
\begin{equation*}
    2<b< \min \left\{ \frac{\Upsilon(p)}{p}, 2\cosh(p/2)\right\},
\end{equation*}
which exists by virtue of \Cref{l:propphi}. Splitting the summation over $\ell$ in \eqref{eta12}, we obtain the bounds
\begin{align} \label{eq:eta12_1}
      \sum_{a\in \mathbb{N}'}\sum_{\ell\in \mathbb{Z}': a+\ell\le bt}J_{a+\ell}^2(2t)e^{p(\ell+\frac12)} \le C \cdot  e^{pbt} \le  C \cdot  e^{\Upsilon(p)t}
\end{align}
from \Cref{l:jupbd1} and
\begin{align}\label{eq:eta12_2}
     \sum_{a\in \mathbb{N}'}\sum_{\ell\in \mathbb{Z}': a+\ell\ge bt}J_{a+\ell}^2(2t)e^{p(\ell+\frac12)}  & \le C\cdot e^{\Upsilon(p)t}
\end{align}
from  \Cref{l:jupbd2}. Combining the inequalities \eqref{eq:eta12_1}, \eqref{eq:eta12_2} with \eqref{eta12} proves the upper bound in \eqref{eq:proof_traceorder}.
For the lower bound, we set $d=\frac1t\lfloor 2t\cosh(p/2)\rfloor$ and focus only on a single term in the double sum in the right-hand side of \eqref{eta11}, with $a=\frac12$ and $\ell=d \, 
t-\frac12$. We have
\begin{align}\label{etail12}
    \mbox{r.h.s.~of \eqref{eta11}} \ge n! \cdot e^{pdt}J_{dt}^2(2t)\int_0^{y_{*}} u^{-\alpha}(1-u)^{n+\alpha-1} \diff u,
\end{align}
where 
\begin{align*}
    y_*:=y_{dt-\frac12}  =(1+\agamma^{-1}q^{dt})^{-1} =(1+e^{(\Upsilon(p)-\delta-pd)t/s})^{-1}.
\end{align*}
For large enough $t$, we have $\Upsilon(p)-\delta-pd < 0$, which forces $y_*\to 1$ as $t\to \infty$. Thus, for all large enough $t$, one can ensure
\begin{align} \label{eq:etail12_1}
    \int_0^{y_{*}} u^{-\alpha}(1-u)^{n+\alpha-1} \diff u \ge C^{-1}\int_0^{1} u^{-\alpha}(1-u)^{n+\alpha-1} \diff u = C^{-1}\cdot \frac{\Gamma(1-\alpha)\Gamma(n+\alpha)}{n!}.
\end{align}
Applying \eqref{e:jest2} we see that for large enough $t$ we have
\begin{align} \label{eq:etail12_2}
     e^{pdt}J_{dt}^2(2t) \ge \frac{1}{C \, t} \cdot e^{-t\Phi_+(2\cosh(p/2))+pdt}.
\end{align}
{Note that 
\begin{align*}
    pdt-t\Phi_+(2\cosh(p/2)) & \le p+t[2p\cosh(p/2)-\Phi_+(2\cosh(p/2))] \\ & =p+tU_p(2\cosh(p/2))=p+\Upsilon(p)t.
\end{align*}
 Thus,
\begin{align} \label{eq:etail12_3}
     e^{pdt}J_{dt}^2(2t) \ge \frac{1}{C \, t} \cdot e^{-t\Phi_+(2\cosh(p/2))+pdt} \ge \frac{1}{\widetilde{C} \, t} \cdot e^{\Upsilon(p)t}.
\end{align}}
Inserting the estimates \eqref{eq:etail12_1}, \eqref{eq:etail12_3} in the right-hand side of \eqref{etail12} proves the desired lower bound in \eqref{eq:proof_traceorder}. This completes the proof.
\end{proof}

We conclude this subsection by proving additional bounds on the trace of the kernel $K_{\zeta,t}$ and its derivatives. These will be useful for our analysis in \Cref{sec.ho} below.

\begin{lemma}[Bounds on traces]\label{l:trbd} Fix any $r>0$, $\sigma>2+r$ and $n\ge 1$.  There exists a constant $\Con=\Con(n,r,q)>0$ such that
    \begin{align}
        |\tr(K_{q^{\sigma t},t})| & \le \Con \cdot e^{ t\cdot U_{\logq}\big(\min\{\sigma,x_{\logq}^*\}\big)-t{\logq}\sigma  }
        \label{eq:bound_trace_1}
        \\ 
        |\tr(K_{q^{\sigma t},t}^{(n)})| & \le \Con  \cdot e^{ t\cdot U_{n{\logq}}\big(\min\{\sigma,x_{n{\logq}}^*\}\big)}
        \label{eq:bound_trace_2}
    \end{align}
where $\logq:=\log q^{-1}$ and the function $U_v$ was defined in \eqref{def:U} and $x_p^*=2\cosh(p/2)$ was defined in \Cref{l:propphi}. 
\end{lemma}

\begin{proof}  We use the fact that $1+q^{\ell-\sigma t+\frac12} \ge  \ind_{\ell>\sigma t} +q^{\ell-\sigma t+\frac12} \cdot \ind_{\ell\le\sigma t}$ to get
\begin{equation}
    \begin{split}
        |\tr(K_{q^{\sigma t},t})| & =\sum_{a\in \N'}\sum_{\ell \in \Z'} \frac{J_{a+\ell}^2(2t)}{1+q^{\ell-\sigma t+\frac12}}  \\ & \le  q^{\sigma t}\sum_{a\in \N'}\sum_{\ell \in \Z', \ell\le \sigma t } q^{-\ell-\frac12}J_{a+\ell}^2(2t)+\sum_{a\in \N'}\sum_{\ell \in \Z',\ell>\sigma t} J_{a+\ell}^2(2t). \label{e:jt}
    \end{split}
\end{equation}
The first term in the right-hand side of \eqref{e:jt} can be bounded using \Cref{l:jupbd3}, while the second term can be estimated using \eqref{e:jest}. Together we have
\begin{align*}
    \mbox{r.h.s.~of \eqref{e:jt}} \le \Con \cdot  \bigg[ e^{t\big[U_{\logq}(\min\{\sigma,x_{\logq}^*\})-{\logq}\sigma\big]}+e^{-t \cdot \Phi_+(\sigma)}\bigg].
\end{align*}
Since $U_{\logq}(\min\{\sigma,x_{\logq}^*\})-{\logq}\sigma \ge U_{\logq}(\sigma)-{\logq}\sigma=-\Phi_+(\sigma)$, the previous inequality implies \eqref{eq:bound_trace_1}. In order to show \eqref{eq:bound_trace_2}, we use the fact that $q^{\sigma t}+q^{\ell+\frac12} \ge q^{\sigma t} \cdot \ind_{\ell>\sigma t} +q^{\ell+\frac12} \cdot \ind_{\ell\le\sigma t}$ to get 
\begin{equation} \label{e:jt1}
\begin{split}
    & |\operatorname{tr}(K_{q^{\sigma t},t}^{(n)})| \\ & =n!\sum_{a\in \mathbb{N}'}\sum_{\ell\in \mathbb{Z}'}\hspace{-0.2cm}\frac{q^{\ell+\frac12}J_{a+\ell}^2(2t)}{(q^{\sigma t}+q^{\ell+\frac12})^{n+1}} \\ & \le  n!\!\cdot\!q^{-\frac{n}2} \sum_{a\in \N'}\sum_{\ell \in \Z', \ell\le \sigma t } q^{-n\ell}J_{a+\ell}^2(2t)+n!\!\cdot\! q^{-\sigma (n+1) t}\sum_{a\in \N'}\sum_{\ell \in \Z',\ell>\sigma t}\hspace{-0.2cm} q^{\ell+\frac12} J_{a+\ell}^2(2t). 
    \end{split}
\end{equation}
Once again, using \Cref{l:jupbd3} and \eqref{e:jest} we arrive the bound
\begin{align*}
    \eqref{e:jt1} \le \Con \cdot \bigg[ e^{ t\cdot U_{n{\logq}}(\min\{\sigma,x_{n{\logq}}^*\}) }+e^{ t [-\Phi_+(\sigma)+n{\logq}\sigma] }\bigg].
\end{align*}
As $U_{n{\logq}}(\min\{\sigma,x_{n{\logq}}^*\}) \ge -\Phi_+(\sigma)+n{\logq}\sigma$, this proves \eqref{eq:bound_trace_2}, completing the proof.
\end{proof}

\subsection{Higher-order terms}\label{sec.ho} The goal of this section is to prove \Cref{p:ho}. To deal with the expression $\det(I-K_{\zeta,t})+\tr(K_{\zeta,t})$, we will use the exterior algebra definition for Fredholm determinants \cite{simon1977notes}, which we recall here for convenience. 

For a trace-class operator $T$ on a Hilbert space $H$, consider the $L$-th exterior power $\wedge_{i=1}^L H$ and the operator $T^{\wedge L}$ defined by $T^{\wedge L}(v_1 \wedge \cdots \wedge v_L) := (Tv_1) \wedge \cdots \wedge (Tv_L)$. The operator $T^{\wedge L}$ is trace-class on $\wedge_{i=1}^L H$ and we have
\begin{align*}
        \det(I-T)=1+\sum_{L=1}^{\infty} (-1)^{L}\tr(T^{\wedge L}).
    \end{align*}
    In light of the above formula, we have
  \begin{align*}
        \det(I-K_{\zeta,t})+\tr(K_{\zeta,t})=1+\sum_{L=2}^{\infty} (-1)^{L}\tr(K_{\zeta,t}^{\wedge L}).
    \end{align*}  
From \eqref{eq:fred} and the calculation in \eqref{eta0}, we know $\det(I-K_{\zeta,t})$ is infinitely differentiable and from \Cref{l:deriK} we have $\tr(K_{\zeta,t})$ is also infinitely differentiable. Thus taking derivatives on both sides of the above equation we get 
\begin{align*}
    \frac{\diff^n}{\diff \zeta^n}[\det(I-K_{\zeta,t})+\tr(K_{\zeta,t})]=\frac{\diff^n}{\diff \zeta^n}\left[\sum_{L=2}^{\infty} (-1)^{L}\tr(K_{\zeta,t}^{\wedge L})\right].
\end{align*}
We claim that the sum and derivative can be interchanged, i.e.,
\begin{align}\label{cd:int}
    \frac{\diff^n}{\diff \zeta^n}\left[\sum_{L=2}^{\infty} (-1)^{L}\tr(K_{\zeta,t}^{\wedge L})\right]=\sum_{L=2}^{\infty} (-1)^{L}\frac{\diff^n}{\diff \zeta^n}\tr(K_{\zeta,t}^{\wedge L}).
\end{align}
We shall justify the interchange in \Cref{l:hotrace}. Assuming this, we insert the above formula in the right-hand side of \eqref{e.ho} to get 
\begin{align}\label{e.interchange}
    \int_0^{\agamma} \zeta^{-\alpha}\left|\frac{\diff^n}{\diff\zeta^n}\big[\det(I-K_{\zeta,t})+\operatorname{tr}(K_{\zeta,t})\big]\right| \diff\zeta \le \sum_{L=2}^{\infty} \int_0^{\agamma} \zeta^{-\alpha}\left|\frac{\diff^n}{\diff \zeta^n}\tr(K_{\zeta,t}^{\wedge L})\right| \diff \zeta.
\end{align}
The derivatives of $\tr(K_{\zeta,t}^{\wedge L})$ can be estimated using terms of the form $\tr(K_{\zeta,t}^{(m)}).$ To state precisely these estimates, we introduce a few pieces of notation. For any $n,L\in \mathbb{N}$ define the set of compositions of $n$ of length $L$
\begin{align*}
    \mc(L,n)\coloneqq\{\vec{m} \in (\Z_{\ge0})^L \mid m_1+m_2+\cdots+m_L=n\}
\end{align*}
and the multinomial coefficient
\begin{equation*}
    \binom{n}{\vec{m}}:=\frac{n!}{m_1! m_2!\cdots m_L!}.
\end{equation*}
We set $r:=(\Upsilon(p)/p-2)/8>0$ so that
\begin{align}\label{e.qgamma}
    q^{(2+4r)t} =e^{-(2+4r)pt/s} > e^{-(\Upsilon(p)-\delta)t/s}=\agamma.
\end{align}
where $\agamma$ is defined in \eqref{def:gamma}. Note that the above relation ensures that the bounds in \Cref{l:trbd} are valid for all $K_{\zeta,t}$ with $\zeta\in [0,\tau]$.

We now extract a convenient expression for the derivatives of $\tr(K_{\zeta,t}^{\wedge L})$ in the following lemma.

\begin{lemma}\label{l:interch1} Fix an orthonormal basis $\{e_i\}_{i\ge 1}$ of $\ell^2(\mathbb{N}')$. For each $t>0$, $\tr(K_{\zeta,t}^{\wedge L})$ is differentiable in $\zeta$ at each $\zeta\in (0,\agamma]$. Furthermore, we have
\begin{align}\label{e.interch1}
    \frac{\diff^n}{\diff\zeta^n}\tr(K_{\zeta,t}^{\wedge L}) = \sum_{i_1<\cdots<i_L}\sum_{\vec{m}\in \mc(L,n)} \binom{n}{\vec{m}} \det\big(\langle e_{i_k}, K_{\zeta,t}^{(m_j)}e_{i_j}\rangle \big)_{j,k=1}^{L}
\end{align}
where $K_{\zeta,t}^{(n)}$ defined in \eqref{e:deriK}.
\end{lemma}
\begin{proof} Let us first note that
\begin{equation}
\label{e.intr01}
    \begin{aligned}
    \tr(K_{\zeta,t}^{\wedge L}) & = \sum_{i_1<\cdots<i_L}\langle e_{i_1} \wedge \cdots \wedge e_{i_L}, K_{\zeta,t} e_{i_1} \wedge \cdots \wedge K_{\zeta,t}e_{i_L}\rangle \\ & = \sum_{i_1<\cdots<i_L}\det\big(\langle e_{i_k},K_{\zeta,t}e_{i_j}\rangle \big)_{j,k=1}^L. 
\end{aligned}
\end{equation}

We wish to take the derivative of the terms inside the above sum. By the product rule for derivatives, we have
\begin{align*}
    \frac{\diff^n}{\diff\zeta^n} \det\big(\langle e_{i_k},K_{\zeta,t}e_{i_\ell}\rangle \big)_{k,\ell=1}^L = \sum_{\vec{m}\in \mc(L,n)} \binom{n}{\vec{m}} \det\bigg(\frac{d^{m_j}}{d\zeta^{m_j}}\langle e_{i_k}, K_{\zeta,t}e_{i_j}\rangle\bigg)_{j,k=1}^{L}.
\end{align*}
Thanks to \Cref{l:deriK}, we may pass the derivative on r.h.s.~of the above equation inside to get 
\begin{align}\label{e.intr02}
    \frac{\diff^n}{\diff\zeta^n} \det\big(\langle e_{i_k},K_{\zeta,t}e_{i_\ell}\rangle \big)_{k,\ell=1}^L = \sum_{\vec{m}\in \mc(L,n)} \binom{n}{\vec{m}} \det\big(\langle e_{i_k}, K_{\zeta,t}^{(m_j)}e_{i_j}\rangle\big)_{j,k=1}^{L}.
\end{align} 
Given the identities in \eqref{e.intr01} and \eqref{e.intr02}, \eqref{e.interch1} follows by taking derivative on both sides of \eqref{e.intr01} and justifying the interchange of derivatives and the infinite sum $\sum_{i_1<\cdots<i_L}$. To establish the interchange, employing \cite[Proposition 4.2]{dt21}, it suffices to show that
\begin{align*}
   \sum_{i_1<\cdots<i_L} \sum_{\vec{m}\in \mc(L,n)} \binom{n}{\vec{m}} \det\big(\langle e_{i_k}, K_{\zeta,t}^{(m_j)}e_{i_j}\rangle\big)_{j,k=1}^{L}
\end{align*}
converges absolutely and uniformly for $\zeta\in [0,\agamma]$. To this end, we note that $K_{\zeta,t}^{(m_j)}$s are self-adjoint Hilbert-Schimdt operators on $\ell^2(\mathbb{N}')$. Thus, appealing to \cite[Lemma 4.3]{dt21} we have
\begin{align*}
    \sum_{i_1<\cdots<i_L} \sum_{\vec{m}\in \mc(L,n)} \binom{n}{\vec{m}} \left|\det\big(\langle e_{i_k}, K_{\zeta,t}^{(m_j)}e_{i_j}\rangle\big)_{j,k=1}^{L}\right| \le L!\sum_{\vec{m}\in \mc(L,n)} \binom{n}{\vec{m}} \prod_{i=1}^L \left|\tr( K_{\zeta,t}^{(m_i)})\right|.
\end{align*}
Note that The bounds from \Cref{l:trbd} ensure that the right-hand side of the above equation converges uniformly for $\zeta\in [0,\agamma]$. This completes the proof.
\end{proof}

The following lemma provides an upper bound for the derivatives of $\tr(K_{\zeta,t}^{\wedge L})$.

\begin{lemma}\label{l:hotrace} For $L\ge 2$, we have
  \begin{align*}
      \left|\frac{\diff^n}{\diff\zeta^n}\tr(K_{\zeta,t}^{\wedge L}) \right| \le \sum_{\vec{m}\in \mc(L,n)} \binom{n}{\vec{m}} \frac{(\#\mathrm{supp}(\vec{m}))!}{(L-\#\mathrm{supp}(\vec{m}))!}\prod_{i=1}^L |\tr(K_{\zeta,t}^{(m_i)})|
  \end{align*}  
  where $\#\mathrm{supp}(\vec{m})=\#\{i\mid m_i>0\}$. Furthermore, we have
  \begin{align}\label{cd:int2}
    \frac{\diff^n}{\diff \zeta^n}\left[\sum_{L=2}^{\infty} (-1)^{L}\tr(K_{\zeta,t}^{\wedge L})\right]=\sum_{L=2}^{\infty} (-1)^{L}\frac{\diff^n}{\diff \zeta^n}\tr(K_{\zeta,t}^{\wedge L}).
\end{align}
\end{lemma}
\begin{proof} The proof of the above lemma can be completed employing  \Cref{l:interch1} and using the argument presented in the proof of \cite[Proposition 4.5]{dt21}. The only thing that we need to check is that there exists $\Con=\Con(n,t)>0$ such that for all $\vec{m}\in \mc({L,n})$ and for all $\zeta\in [0,\agamma]$ we have 
\begin{align}\label{e:trbd3}
    \left|\tr( K_{\zeta,t}^{(m_i)})\right| \le \Con. 
\end{align}
This last inequality is immediate from the trace bounds in \Cref{l:trbd}, completing the proof.
\end{proof}

Given the above lemma, in view of the estimate in \eqref{e.interchange}, we now need to estimate integrals of $\zeta^{-\alpha}$ times products of traces of different derivative kernels. This is achieved in the following lemma.

\begin{lemma}\label{l.keytech} Fix $t>1$ and $L\ge 2$. There exists a constant $\Con=\Con(p,q)>0$, such that for all $\vec{m}\in \mc(L,n)$ we have
    \begin{align*}
        \int_0^{\agamma} \zeta^{-\alpha}\prod_{i=1}^L |\tr(K_{\zeta,t}^{(m_i)})| \diff \zeta \le \Con ^{L} \cdot t \cdot e^{\Upsilon(p)t-\tfrac1{\Con}t}.
    \end{align*}
\end{lemma}
\begin{proof}
    Since $q^{(2+4r)t} \ge \agamma$ from \eqref{e.qgamma}, we shall instead provide a bound for
\begin{align*}
       A & := \int_0^{q^{(2+4r)t}} \zeta^{-\alpha}\prod_{i=1}^L |\tr(K_{\zeta,t}^{(m_i)})| \diff \zeta.
    \end{align*} 
Let us set $w:=\#\mathrm{supp}(\vec{m})$ and assume $m_1,m_2,\ldots,m_w>0$. We use the substitution $\zeta =q^{\sigma t}$ so that $\diff\zeta=tq^{ \sigma t}\log q\,\diff\sigma$. Using the bounds from \Cref{l:trbd} we have 
   \begin{align}\label{e.tech00}
       A & =\logq \int_{2+4r}^{\infty} q^{-\sigma\alpha t+\sigma t}\prod_{i=1}^L |\tr(K_{q^{\sigma t},t}^{(m_i)})| \diff\sigma \le \Con^L \cdot t\int_{2+4r}^{\infty} e^{M(\sigma) t} \diff \sigma
    \end{align} 
   where 
   \begin{align*}
       M(\sigma):=(\alpha-L+w-1)\sigma {\logq} +(L-w)U_{\logq}\big(\min\{\sigma,x_{\logq}^*\}\big)+\sum_{j=1}^w U_{m_j{\logq}}\big(\min\{\sigma,x_{m_j{\logq}}^*\}\big),
   \end{align*}
   and $\logq$ is defined in \eqref{eq:logq}.    We shall now estimate the integral by considering several cases.
\begin{enumerate}[leftmargin=10pt]
\setlength\itemsep{0.8em}
    \item \textbf{When $\sigma\in (2+4r,x_{{\logq}}^*)$}, we have
    \begin{align*}
       M(\sigma) & =(\alpha-L+w-1)\sigma {\logq} +(L-w)U_{\logq}\big(\sigma)+\sum_{j=1}^w U_{m_j{\logq}}\big(\sigma)\\ & =(\alpha-L+w-1)\sigma {\logq} +(L-w)[\logq \sigma-\Phi_+(\sigma)]+\sum_{j=1}^w [m_j\logq \sigma-\Phi_+(\sigma)]  \\ & =p\sigma  -L\Phi_+\big(\sigma) = L \cdot U_{p/L}(\sigma), 
   \end{align*}
   where we used the fact $m_1+\cdots+m_w=n=s-\alpha+1$ and $p=s\logq$.
   By properties of the $U_v$ function (see \Cref{l:propphi}) we have $U_{p/L}(\sigma)  \le 4\sinh(p/2L)$.   Thus we have
   \begin{align*}
       \int_{2+4r}^{x_{\logq}^*} e^{M(\sigma)t}\diff\sigma \le x_{\logq}^* \cdot e^{4L\sinh(p/2L)t}.
   \end{align*}
   Set $g_1(L):=4L\sinh(p/2L)$. Observe that $g_1'(L)=4\cosh(p/L)(\tanh(p/L)-p/L)<0$. This implies $g_1(L)$ is strictly decreasing. As $L\ge 2$, we have $g_1(L)<g_1(1)=4\sinh(p/2)=\Upsilon(p)$. Thus one can find a constant $\Con>0$ free of $L$, such that $x_{\logq}^* \cdot e^{4L\sinh(p/2L)t} \le \Con e^{\Upsilon(p)t-\frac1\Con t}$, which is precisely the bound we are looking for. This completes our work for this range of $\sigma$.
   
   \item  \textbf{When $w\ge 2$, $\sigma \ge x_{\logq}^*$}, we have $n\ge 2$ and
   \begin{equation}
       \label{e.tech01}
        \begin{aligned}
       M(\sigma) & \le (\alpha-1)\sigma {\logq} -(L-w)\Phi_+(x_{\logq}^*)+\sum_{j=1}^w U_{m_j{\logq}}(x_{m_j{\logq}}^*) \\ & \le  (\alpha-1)\sigma {\logq} +\sum_{j=1}^w 4\sinh(m_j{\logq}/2).  
   \end{aligned}
   \end{equation}
{We claim that for $a_1,\ldots,a_k>0$ we have
\begin{align}\label{e.clsinh}
    \sum_{i=1}^k\sinh(a_i) \le \sinh\bigg(\sum_{i=1}^k a_i\bigg).
\end{align}
Let us quickly explain why \eqref{e.clsinh} is true. Suppose $a\ge 0$. Set $P(x):=\sinh(x+a)-\sinh(x)-\sinh(a)$. Note that $P'(x)=\cosh(x+a)-\cosh(x) \ge 0$ for all $x \ge 0$. Thus $P(\cdot)$ is increasing on $[0,\infty)$ and hence $P(x)\ge P(0)=0$ for all $x\ge 0$. Hence for $x,a\ge 0$, we have $\sinh(a)+\sinh(x)\le \sinh(x+a)$. Iterating this inequality $k$ times we get \eqref{e.clsinh}.} Now using \eqref{e.clsinh}  we get that
   \begin{align*}
       \mbox{r.h.s.~of \eqref{e.tech01}} & \le (\alpha-1)\sigma {\logq} +4\sinh((n-m_1){\logq}/2)+4\sinh(m_1{\logq}/2) \\ & \le  (\alpha-1)\sigma {\logq}+4\sinh((n-1){\logq}/2)+4\sinh({\logq}/2) \\ & = (\alpha-1)\sigma {\logq}+\Upsilon((n-1){\logq})+\Upsilon({\logq}),
   \end{align*}
   where the last inequality follows from the fact that for $m_1\in [1,n-1]$, $\sinh((n-m_1){\logq}/2)+\sinh(m_1{\logq}/2)$ is maximized at $m_1=1$ or $n-1$.  Using the above bound,  we get
   \begin{align}
       \nonumber \int_{x_{\logq}^*}^{\infty} e^{M(\sigma)t} \diff\sigma & \le \Con \cdot (1-\alpha)^{-1} \cdot e^{(\alpha-1)x_{\logq}^*t{\logq}} \cdot e^{\Upsilon((n-1){\logq})t+\Upsilon({\logq})t}  \\ & \le \nonumber \Con \cdot (1-\alpha)^{-1} \cdot e^{[\Upsilon((n-1){\logq})+\Upsilon({\logq})+x_{\logq}^*(\alpha-1){\logq}]t} \\ & \le \Con \cdot (1-\alpha)^{-1} \cdot e^{\Upsilon(p)t} \cdot e^{[\Upsilon({\logq})+\Upsilon((n-1){\logq})+x_{\logq}^*(\alpha-1){\logq}-\Upsilon(p)]t}. \label{e.tech02}
   \end{align}
   Recall $p=(n+\alpha-1){\logq}$. Set $g_2(\alpha):=\Upsilon({\logq})+\Upsilon((n-1){\logq})+x_{\logq}^*(\alpha-1){\logq}-\Upsilon((n+\alpha-1){\logq}).$ As $n\ge 2$, we have
   \begin{align*}
       g_2'(\alpha)=x_{\logq}^*{\logq}-\Upsilon'((n+\alpha-1){\logq}){\logq} =2{\logq}[\cosh({\logq}/2)-\cosh((n+\alpha-1){\logq}/2)] \le 0.
   \end{align*}
   This forces $g_2(\alpha)\le g_2(0) =\Upsilon({\logq})-x_{\logq}^*{\logq}<0$ by property (d) in \Cref{l:propphi}. Thus in summary there exists a constant $\Con>0$ such that $e^{g_2(\alpha)t} \le e^{-\frac1\Con t}$. Plugging this back in \eqref{e.tech02} leads to the desired estimate. This completes our work in this step. 
   
   \item \textbf{When $w=1$, $\sigma \ge x_{n{\logq}}^*$,} we have
   \begin{align*}
       M(\sigma)=(\alpha-L)\sigma {\logq} +(L-1)U_{\logq}(x_{\logq}^*)+U_{n{\logq}}(x_{n{\logq}}^*).
   \end{align*}
   We thus have
   \begin{align}\label{e.tech03}
       \int_{x_{n{\logq}}^*}^{\infty} e^{M(\sigma)t} \diff\sigma & \le \Con \cdot (L-\alpha)^{-1} \cdot e^{(\alpha-L)x_{n{\logq}}^*t{\logq}} \cdot e^{((L-1)U_{\logq}(x_{\logq}^*)+U_{n{\logq}}(x_{n{\logq}}^*))t}.
   \end{align}
   Note that the exponent above can be bounded from above as
\begin{align*}
    & \alpha x_{n{\logq}}^*{\logq}+L(x_{\logq}^*-x_{n{\logq}}^*){\logq}-(L-1)\Phi_+(x_{\logq}^*)-{\logq}x_{\logq}^*+U_{n{\logq}}(x_{n{\logq}}^*) \\ & \le \alpha x_{n{\logq}}^*{\logq}+2(x_{\logq}^*-x_{n{\logq}}^*){\logq}-\Phi_+(x_{\logq}^*)-{\logq}x_{\logq}^*+U_{n{\logq}}(x_{n{\logq}}^*) \\ & = \Upsilon(p)+g_3(\alpha),
\end{align*}
  where $g_3(\alpha)=(\alpha -2)x_{n{\logq}}^*{\logq}+\Upsilon({\logq})+\Upsilon(n{\logq})-\Upsilon((n+\alpha-1){\logq})$. We have
   $g_3'(\alpha)>0$. So, $g_3(\alpha) \le g_3(1)=-x_{n{\logq}}^*{\logq}+\Upsilon({\logq})\le \Upsilon({\logq})-x_{\logq}^*{\logq}<0.$ Hence, just as in part (b), we have shown that there exists a constant $\Con>0$ such that $e^{g_3(\alpha)t} \le e^{-\frac1\Con t}$. Plugging this back in \eqref{e.tech03} leads to the desired estimate. This completes our work in this part.

\item \textbf{When $w=1$, $\sigma \in (x_{\logq}^*,x_{n{\logq}}^*)$, and $n+\alpha-1 \le L$,} we have
   \begin{align*}
       M(\sigma)=(\alpha-L)\sigma {\logq} +(L-1)U_{\logq}(x_{\logq}^*)+U_{n{\logq}}(\sigma).
   \end{align*}
   In this range, we have $M'(\sigma)=(n+\alpha -L){\logq}-\Phi_+'(\sigma) \le -\Phi_+'(x_{\logq}^*)+{\logq}=0$, which forces $M$ to be decreasing. Thus, 
   \begin{align*}
       M(\sigma) & \le M(x_{\logq}^*)=(\alpha-L)x_{\logq}^*{\logq} +(L-1)U_{\logq}(x_{\logq}^*)+U_{n{\logq}}(x_{\logq}^*) \\ & = (n+\alpha -1)x_{\logq}^*{\logq} -L\Phi_+(x_{\logq}^*) \le sx_{\logq}^*{\logq} -2\Phi_+(x_{\logq}^*).
   \end{align*}
   Let us set $g_4(s):=\Upsilon(s{\logq})-sx_{\logq}^*{\logq}+2\Phi_+(x_{\logq}^*)$. Note that $g_4'(1)=0$ and $g_4''(s) \ge 0$. Thus, $g_4(s) \ge g_4(1)=\Upsilon({\logq})-x_{\logq}^*{\logq}+2\Phi_+(x_{\logq}^*)=\Phi_+(x_{\logq}^*)$. So, $sx_{\logq}^*{\logq} -2\Phi_+(x_{\logq}^*) \le \Upsilon(p)-\Phi(x_{\logq}^*)$. As $\Phi_+(y)>0$ for all $y>0$, this forces
   \begin{align*}
       \int_{x_{\logq}^*}^{x_{n{\logq}}^*} e^{M(\sigma)t} \diff\sigma \le x_{n{\logq}}^* e^{\Upsilon(p)t}\cdot e^{-\Phi_+(x_{\logq}^*)t} \le \Con e^{\Upsilon(p)t-\frac1\Con t},
   \end{align*}
   for some $\Con>0$ depending only on $q$ (as $\logq=\log q^{-1})$. This completes our work for this part.
   
   \item \textbf{When $w=1$, $\sigma \in (x_{\logq}^*,x_{n{\logq}}^*)$, and $n+\alpha-1 > L$,} we have
   \begin{align*}
       M(\sigma) & =(\alpha-L)\sigma {\logq} +(L-1)U_{\logq}(x_{\logq}^*)+U_{n{\logq}}(\sigma) = U_{(n+\alpha-L){\logq}}(\sigma)+(L-1)U_{\logq}(x_{\logq}^*).
   \end{align*}
   In the above range, $M(\sigma)$ attains a maximium at $x_{(n+\alpha-L){\logq}}^*$. Thus, by setting $x=L-1$, we get
   \begin{align*}
       M(\sigma) & \le M(x_{(n+\alpha-L){\logq}}^*)=\Upsilon((s-x){\logq})+x\Upsilon({\logq}):=g_5(x).
   \end{align*}
   Note that $g_5'(x)=-\Upsilon'((s-x){\logq})+\Upsilon({\logq}) \le -{\logq}\Upsilon'({\logq})+\Upsilon({\logq})<0$ as $\Upsilon'(p)$ is increasing and $s-x=n+\alpha-L\ge 1$. So, $g_5(x) \le g_5(1)=\Upsilon((s-1){\logq})+\Upsilon({\logq})$. Now $\Upsilon((s-1){\logq})-\Upsilon(s{\logq})$ is decreasing in $s$, hence $\Upsilon((s-1){\logq})-\Upsilon(s{\logq}) \le \Upsilon({\logq})-\Upsilon(2{\logq})$. Hence $$g_5(x) \le 2\Upsilon({\logq})-\Upsilon(2{\logq})+\Upsilon(s{\logq})=2\Upsilon({\logq})-\Upsilon(2{\logq})+\Upsilon(p).$$
   Since $2\Upsilon({\logq})-\Upsilon(2{\logq})<0$, we have that
   \begin{align*}
       \int_{x_{\logq}^*}^{x_{n{\logq}}^*} e^{M(\sigma)t} \diff\sigma \le x_{n{\logq}}^* e^{\Upsilon(p)t}\cdot e^{(2\Upsilon({\logq})-\Upsilon(2{\logq}))t} \le \Con e^{\Upsilon(p)t-\frac1\Con t},
   \end{align*}
completing our work for this part.
\end{enumerate}

Combining all the above parts, we have thus shown that
\begin{align*}
    \int_{2+4r}^{\infty} e^{M(\sigma)t} \diff\sigma \le  \Con e^{\Upsilon(p)t-\frac1\Con t}
\end{align*}
for some $\Con>0$ depending on $p$ and $q$. Inserting this bound back in \eqref{e.tech00} and adjusting the constant $\Con$, we get that $A \le \Con^{L}\cdot t\cdot e^{\Upsilon(p)t-\frac1\Con t}$. Again, as $q^{(2+4r)t}\ge \agamma$, this establishes \Cref{l.keytech}. 
\end{proof}

Combining preliminary results enumerated in the lemmas above, we are now ready to prove \Cref{p:ho}.

\begin{proof}[Proof of \Cref{p:ho}]  
In view of the estimates from \Cref{l:hotrace} and \Cref{l.keytech}, we have
\begin{align*}
    \mbox{r.h.s.~of \eqref{e.interchange}} \le \Con \cdot t\cdot e^{\Upsilon(p)t-\frac1\Con t} \sum_{L=2}^{\infty} \sum_{\vec{m}\in \mc(L,n)} \binom{n}{\vec{m}} \frac{(\#\mathrm{supp}(\vec{m}))!\Con^{L}}{(L-\#\mathrm{supp}(\vec{m}))!}, 
\end{align*}
where $C>0$ depends only on $p$. The double sum is computed in proof of Proposition 4.7 in \cite{dt21}. In particular, the double sum is finite and its value depends only on $\Con$ and $n$. Thus. adjusting the constant $\Con$ further, we arrive at the desired estimate in \eqref{e.ho}. This completes the proof.
\end{proof}

\section{Lower-Tail LDP for $q$-PNG} \label{sec:lower_tail}

In this section, we prove the lower-tail Large Deviation Principle for the height function $\h$ of $q$-PNG: \Cref{thm:upper_tail}. In \Cref{subs:prelim_lower_tail}, we introduce continual Young diagrams and several important functionals and discuss a few basic properties of them. In \Cref{sec:cont}, we establish continuity-type results for these functionals. In \Cref{sec:exists} we use probabilistic arguments to derive the precise lower-tail rate function for the largest row of the shifted cylindric Plancherel measure and sharp estimates for the lower-tail of the unshifted ones. We discuss log-concavity properties of Schur polynomials in  \Cref{sec:ltconvex} and prove convexity of $\mathcal{F}$ defined in \eqref{def:f_intro}. We complete the proofs of our main theorems related to the lower-tail in \Cref{sec:4.5}.

\subsection{Preliminaries} \label{subs:prelim_lower_tail}

A key ingredient for the study of the lower-tail rate function of the $q$-PNG is the relation \eqref{eq:S+h=expextation} which matches the probability distribution of a random shift of $\mathfrak{h}$ with a multiplicative expectation of the Poissonized Plancherel measure. For this reason, in this subsection, we recall results concerning the asymptotics of the Plancherel measure.

We introduce the set of continual Young diagrams 
\begin{equation*}
	\mathcal{Y} = \{ \varphi:[0,+\infty) \to [0,+\infty): \varphi \text{ is decreasing and } \| \varphi \|_{L^1} < +\infty \}
\end{equation*}
and its subset of shapes with unit integral
\begin{equation*}
	\mathcal{Y}_1 = \left\{ \phi \in \mathcal{Y} : \| \phi \|_{L^1} =1 \right\}.
\end{equation*}
Given a continual Young diagram $\varphi$ we define its representation in \emph{Russian notation} $\widetilde{\varphi}$, which is the function
\begin{equation*}
	\widetilde{\varphi}(u) = v
	\qquad \Longleftrightarrow \qquad \frac{v-u}{\sqrt{2}} = \varphi\left( \frac{v+u}{\sqrt{2}}  \right).
\end{equation*}
In words, the function $\widetilde{\varphi}$ is obtained rotating by $45^\circ$ counterclockwise the graph of $\varphi$ and as such have $\widetilde{\varphi}(x) \ge |x|$ for all $x\in \mathbb{R}$; see \Cref{fig:young_diagrams}. Motivated by this we further define 
\begin{equation} \label{eq:phi_bar}
	\overline{\varphi}(x) = \widetilde{\varphi}(x) - |x|.
\end{equation}
Translating the properties of the function $\varphi$, the function $\overline{\varphi}$ is easily seen to belong to the set 
\begin{equation*}
	\overline{\mathcal{Y}} = \{ h \in L^1(\mathbb{R}): h \text{ is absolutely continuous, nonnegative}, \, \mathrm{sign}(x) h'(x) \in [-2,0] \text{ a.e.}\}.
\end{equation*}

\begin{figure}%
\centering
\parbox{3.5cm}{
\includegraphics[height=3.2cm]{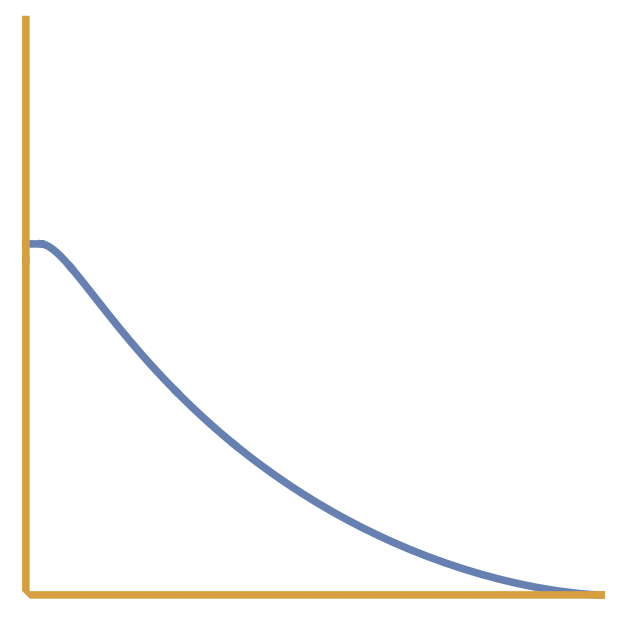}
}%
\parbox{5.5cm}{
\includegraphics[height=2.5cm]{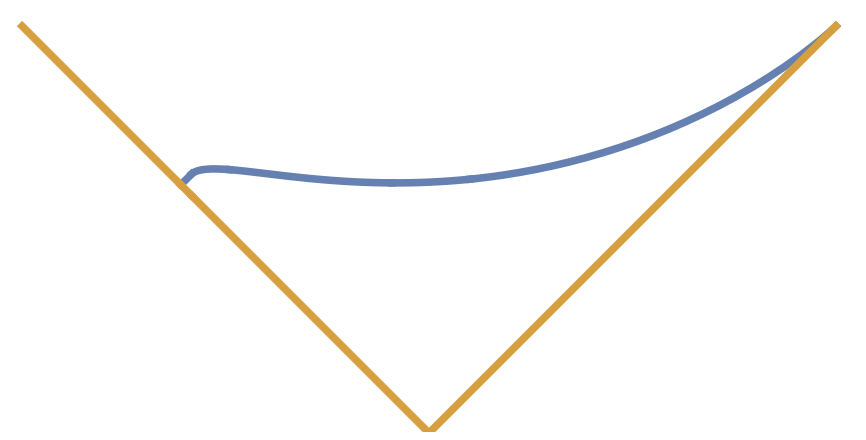}
}%
\parbox{6cm}{
\includegraphics[width=5.6cm]{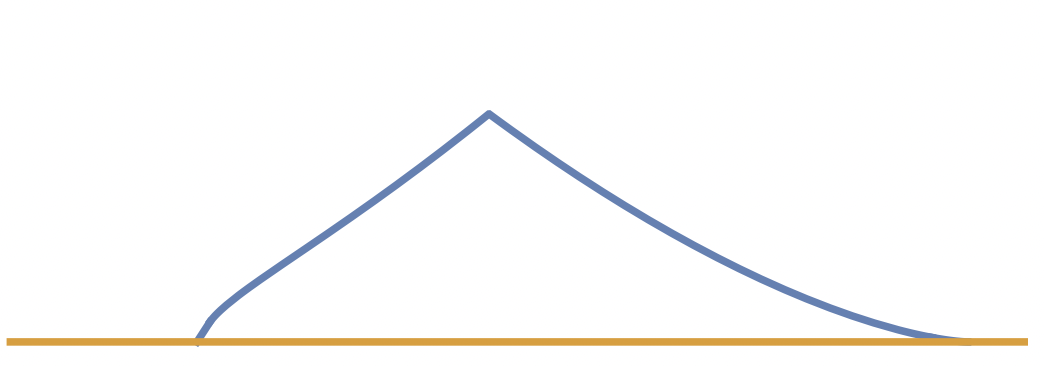}
}
\caption{On the left panel a continual Young diagram $\varphi \in \mathcal{Y}$. On the central panel its representation in Russian notation $\widetilde{\varphi}$. On the right panel the shape $\overline{\varphi} \in \overline{\mathcal{Y}}$.}%
\label{fig:young_diagrams}%
\end{figure}

For a shape $\phi \in \mathcal{Y}_1$ we define its inverse as $\phi^{-1}(y):= \inf\{x\ge 0 \mid \phi(x)\le y \}$. Note that $\phi^{-1} \in \mathcal{Y}_1$. We define the Hook integral of $\phi$ as
\begin{equation} \label{eq:hook_integral}
	I_{\mathrm{hook}}(\phi) := \int_0^{\infty} \diff x \int_0^{\phi(x)} \diff y \log \mathsf{h}_\phi(x,y),
\end{equation}
where
\begin{equation*}
	\mathsf{h}_\phi(x,y) := \phi(x) + \phi^{-1}(y) - x- y.
\end{equation*}
In the seminal paper \cite{VershikKerov_LimShape1077} Vershik and Kerov proved through a series of algebraic manipulations that the hook integral possesses the equivalent representation
\begin{equation} \label{eq:vershik_kerov_expression}
	I_{\mathrm{hook}} (\phi) = - \frac{1}{2} + \frac{1}{2} \| \overline{\phi}_{\mathrm{VKLS}} - \overline{\phi} \|_1^2 + 2 \int_{|y|>\sqrt{2}} \overline{\phi}(y)   \arccosh\left| \frac{y}{\sqrt{2}} \right| \diff y,
\end{equation}
where $\overline{\phi}_{\mathrm{VKLS}}$ is the Logan-Shepp-Vershik-Kerov optimal shape 
\begin{equation} \label{eq:VKLS_shape}
	\overline{\phi}_{\mathrm{VKLS}} (x) = 
	\begin{cases}
		\frac{2}{\pi} \left[\sqrt{2-x^2}+x \arcsin \left(\frac{x}{\sqrt{2}}\right)\right]-| x| \qquad &\text{for } |x|\le\sqrt{2},
		\\
		0  &\text{for } |x|>\sqrt{2},
	\end{cases}
\end{equation}
$\| \cdot \|_1^2$ denotes the {square of the} Sobolev $H^{1/2}$ norm 
\begin{equation}\label{def:sob}
	\| \psi \|_1^2 := \int_\mathbb{R} |\omega| |\hat{\psi}(\omega)|^2 \diff \omega
\end{equation}
where $\hat{\psi}$ is the Fourier transform of $\psi$ defined as {$$\hat{\psi}(\omega)=\int_{\R} e^{-\mathbf{i}\omega x}\psi(x)\diff x .$$} 
An analogous representation to \eqref{eq:vershik_kerov_expression} of the hook functional $I_\mathrm{hook}$ was given in \cite{logan_shepp1977variational}; see also \cite{romik_2015}. {From \eqref{eq:vershik_kerov_expression} it is easy to see that 
\begin{equation}
    I_{\mathrm{hook}} (\phi)\geq -\frac{1}{2},
\end{equation}
for any $\phi\in \mathcal{Y}_1$ with the equality holds if and only if $\overline{\phi}=\overline{\phi}_{\mathrm{VKLS}} $}.

The following proposition states the convergence, at the exponential scale, of the Plancherel measure to the hook functional. Such convergence is uniform with respect to the limit shape. 

\begin{proposition}[\cite{romik_2015}, Theorem 1.14] \label{prop:plancherel_hook_integral} Recall $f^{\lambda}$ from \eqref{eq:flr}. As $n\to\infty$, uniformly over all partitions $\lambda\vdash n$  we have
	\begin{equation*} 
		\frac{1}{n} \log \bigg( \frac{1}{n!}(f^\lambda)^2 \bigg)= -1 - 2 I_\mathrm{hook}(\phi_\lambda)+O\bigg(\frac{\log n}{\sqrt{n}}\bigg),
	\end{equation*}
	where $\phi_\lambda \in \mathcal{Y}_1$ is defined by
	\begin{equation}\label{eq:defphi}
		\phi_\lambda(x) = \frac{1}{\sqrt{n}}\lambda_{\lfloor x \sqrt{n} \rfloor+1}.
	\end{equation}
\end{proposition}

In order to state our theorem we fix some more notation. For a left continuous decreasing function $\phi : \mathbb{R}_+ \to \mathbb{R}_+$ with unit integral and $q \in (0,1)$ we define
\begin{equation}\label{def:Vq}
	\mathcal{V}^{(q)}(x;\phi) = 
	\logq\int_0^\infty  [\phi(y) - y -x]_+ \diff y,
\end{equation}

where $[a]_+=\max\{a,0\}$ and $\logq:=\log q^{-1}$.
Define also the functional 
\begin{equation} \label{eq:W_functional}
	\mathcal{W}^{(q)}(\kappa,\phi;x) =  1+\kappa \log \kappa + 2\kappa I_{\mathrm{hook}} (\phi) + \kappa \mathcal{V}^{(q)}(x/\sqrt{\kappa};\phi), \quad \kappa>0, \phi\in \mathcal{Y}_1, x\in \R,
\end{equation}
and 
\begin{equation} \label{eq:f_Plancherel}
	\mathcal{F}(x) := \inf_{\kappa > 0} \inf_{\phi \in \mathcal{Y}_1}  \left\{ \mathcal{W}^{(q)}(\kappa,\phi;x)  \right\}.
\end{equation}
{Note that the $\mathcal{W}^{(q)}$ functional defined above and the $\mathcal{W}^{(q)}$ functional defined in \eqref{eq:W_intro} are essentially the same; the only difference is that the second coordinate of the one defined in the introduction took shapes in $\overline{\mathcal{Y}}$ with unit integral as input. Thus the function $\mathcal{F}$ defined above is precisely the same as the one defined in \eqref{def:f_intro}.}

We shall show in the next subsection that $\mathcal{F}$ is the lower-tail rate function for the first row of shifted cylindric Plancherel measure. Presently, we end this subsection by discussing a few properties of $\mathcal{W}^{(q)}$ and $\mathcal{F}$.

\begin{proposition} \label{lem:range_kappa}  For each $x\in \R$, the minimizer in the optimization problem in \eqref{eq:f_Plancherel} is attained at some $\kappa_*\in (0,\kappa^*)$, where $\kappa^*$ satisfies $1+\kappa^* \log \kappa^* - \kappa^* = \logq$.
\end{proposition}

\begin{proof} Fix any $x\in \R$ and any shape $\phi \in \mathcal{Y}_1$. We first provide upper and lower bounds for the function $\mathcal{V}^{(q)}(x;\phi)$. Note that when $x\ge 0$ we have  $$0 \le \mathcal{V}^{(q)}(x;\phi)= \logq\int_0^{\infty} [\phi(y)-y-x]_+\diff y \le  \logq \int_0^\infty \phi(y)\diff y =\logq.$$
For $x<0$, we split the integral into two parts 
\begin{align}
	\nonumber
	\int_0^{\infty} [\phi(y)-y-x]_+\diff y & = \int_0^{-x} [\phi(y)-y-x]_+\diff y+\int_{-x}^{\infty} [\phi(y)-y-x]_+\diff y \\ & =\frac{x^2}2+\int_0^{-x} \phi(y)\diff y+\int_{-x}^{\infty} [\phi(y)-y-x]_+\diff y. \label{eq:obsbd}
\end{align}
Since $\phi$ is nonnegative, we thus have that $\int_0^{\infty} [\phi(y)-y-x]_+\diff y\ge x^2/2$. On the other hand for the upper bound notice that $\int_{-x}^{\infty} [\phi(y)-y-x]_+\diff y \le \int_{-x}^{\infty} \phi(y)\diff y.$
Since $\phi$ integrates to $1$, we thus have $\int_0^{\infty} [\phi(y)-y-x]_+\diff y \le \frac{x^2}{2}+1.$ Combining the $x\ge 0$ and $x<0$ case, we deduce that
\begin{equation}\label{eq:vqbound}
	\logq \frac{[-x]_+^2}{2} \le \mathcal{V}^{(q)}(x;\phi) \le \logq \left( \frac{[-x]_+^2}{2} + 1 \right).
\end{equation}
	Then we can bound from both sides the functional $\mathcal{W}^{(q)}$ as
	\begin{equation*}
		1 + \kappa \log \kappa - \kappa + \logq \frac{[-x]_+^2}{2} \le \mathcal{W}^{(q)}(\kappa , \phi ;x) \le 1 + \kappa \log \kappa + 2 \kappa I_\mathrm{hook}(\phi) + \logq \left( \frac{[-x]_+^2}{2} + \kappa \right),
	\end{equation*}
	where in the lower bound we also used the fact that $I_\mathrm{hook}(\phi) \ge -1/2$. For reference let us also evaluate the functional $\mathcal{W}^{(q)}$ at the special values $\kappa =1$ and $\phi = \phi_\mathrm{VKLS}$ as
	\begin{equation*}
		\mathcal{W}^{(q)}(1,\phi_\mathrm{VKLS};x) \le  \logq \left( \frac{[-x]_+^2}{2} + 1 \right).
	\end{equation*}
	Then, for any $\kappa$ such that $1+\kappa \log \kappa - \kappa \ge \logq$ we have
	\begin{equation*}
		\mathcal{W}^{(q)}(1,\phi_\mathrm{VKLS};x) \le 1+\kappa \log \kappa - \kappa +  \logq \frac{[-x]_+^2}{2} \le \mathcal{W}^{(q)}(\kappa,\phi;x).
	\end{equation*}
	This proves that in order for $\kappa$ to be a minimizer, we must have $\kappa \le \kappa^*$. Let us also show that $\kappa = 0$ cannot be a minimizer. For any shape $\phi$ we have
	\begin{equation} \label{eq:W_kappa=0}
		\mathcal{W}^{(q)}(0,\phi;x) = 1 + \logq \frac{[-x]_+^2}{2}.
	\end{equation}
	On the other hand, we have
	\begin{equation*}
		\mathcal{W}^{(q)}(\kappa,\phi_\mathrm{VKLS};x) \le 1 + \kappa \log \kappa + \kappa (\logq-1) + \logq \frac{[-x]_+^2}{2} 
	\end{equation*}
	and the right-hand side can easily be checked to be a decreasing function for $\kappa \in (0,q)$, which at $\kappa = 0$ equals the right-hand side of \eqref{eq:W_kappa=0}. This implies that for $\kappa$ to minimize the functional $\mathcal{W}^{(q)}$ it needs to be $\kappa >0$, completing the proof. 
\end{proof}
\begin{remark} If $\phi$ is bounded with compact support,  then for $x\ll 0$ (large negative values), the two integrals in \eqref{eq:obsbd} are $1$ and $0$ respectively. This forces
	\begin{equation*}
		\mathcal{W}^{(q)}(\kappa,\phi;x) = 1 + \kappa \log \kappa + 2 \kappa I_\mathrm{hook}(\phi) + \logq \left( \frac{x^2}{2} + \kappa \right).
	\end{equation*}
Since $\phi_{\mathrm{VKLS}}$ is bounded with compact support, we thus have 
\begin{equation}\label{eq:wbd}
	\mathcal{W}^{(q)}(\kappa,\phi_{\mathrm{VKLS}};x) = 1 + \kappa \log \kappa + \kappa (\logq-1) + \logq \frac{x^2}{2}
\end{equation}
for $x\ll 0$.
\end{remark}
We now discuss a few properties of $\mathcal{F}$ that can be deduced from the definition.

\begin{proposition}\label{prop:fprop} $\mathcal{F}$ is non-negative and  decreasing. We have $\mathcal{F}(x)>0$ if and only if $x< 2$. There exists $x_q < 0$ such that
		\begin{equation} \label{eq:f_parabola}
				\mathcal{F}(x) = \frac{\logq}{2} x^2 + (1-q)\quad\mbox{ for all } \ \ x\le x_q.
		\end{equation}
\end{proposition}

\begin{proof} The fact that $\mathcal{F}$ is decreasing follows from the definition. Note that
$$\mathcal{W}^{(q)}(\kappa,\phi;x)=(1+\kappa \log \kappa-\kappa)+\kappa(1+2I_{\operatorname{hook}}(\phi))+\kappa \mathcal{V}^{(q)}(x/\sqrt{\kappa};\phi).$$
The three terms above on the right-hand side are non-negative. Thus $\mathcal{F}$ is non-negative. The first two terms are zero if and only if $\kappa=1$ and $\phi=\phi_{\mathrm{VKLS}}$. From \eqref{eq:VKLS_shape}, we get that $\overline{\phi}_\mathrm{VKLS}(-\sqrt{2})=0$ which forces $\phi_\mathrm{VKLS}(0)=2$. Thus, we have {$\max_{y\geq 0}\{\phi_{\mathrm{VKLS}}(y)-y\}=2$}. Hence $\mathcal{V}^{(q)}(x;\phi_{\mathrm{VKLS}})=0$ if and only if $x\ge 2$.  Consequently, $\mathcal{F}(x)=0$ if and only if $x\ge 2$.

   \begin{figure}[t]
		\centering
        \begin{tikzpicture}
            \node[right] at (0.12,-.18) {\includegraphics[trim={0 0 1cm 3cm},clip,width=.4\linewidth]{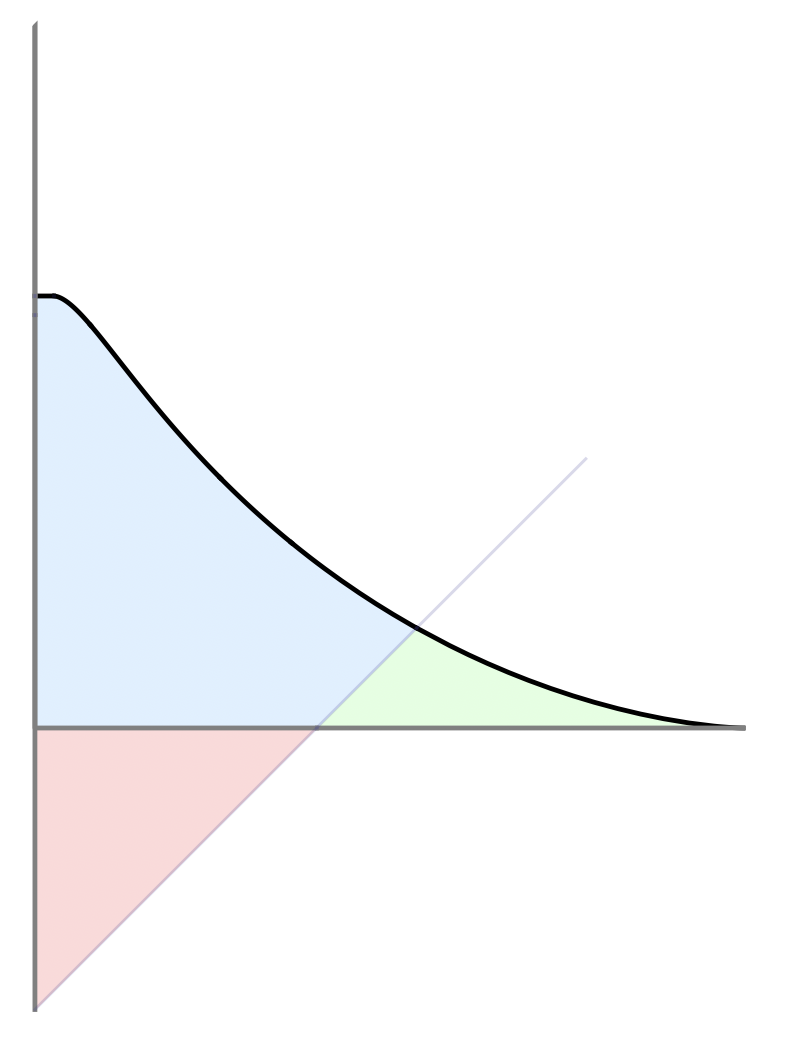}};
            \draw[ultra thick,->,gray] (.5,-3.7) -- (.5,4.5);
            \draw[ultra thick,->,gray] (-.5,-1.1) -- (7,-1.1);
            \node[right] at (1,-2) {$C$};
            \node[right] at (1.5,0) {$A$};
            \node[right] at (3.5,-0.8) {$B$};
            \draw[thick] (.35,-3.32) node[left] {$x$} -- (.75,-3.32);
            \node[above] at (1.8,1.8) {\large $\phi$};
        \end{tikzpicture}
		\caption{The functional $\mathcal{V}^{(q)}(\phi;x)$ is the sum of areas $A$ and $C$.}
		\label{fig:Vfunctional_area}
	\end{figure}

   We now turn toward the proof of parabolic behavior on the left. For a given $\phi$, let $A,B,C$ be the area of the blue, green, and red regions in \Cref{fig:Vfunctional_area}, respectively. Then 
   \begin{equation}\label{eq:relation_area}
       \begin{aligned}
     &  B=\int_{-\frac{x}{\sqrt{2}}}^{\infty}\overline{\phi}(y)\diff y, \quad C  = \frac{x^2}{2}, \\
           & \int_{0}^{\infty}[\phi(y)-y-x]_+\diff y = A+C,\quad 1=\int_{0}^{\infty}\phi(y)\diff y=A+B.
       \end{aligned}
   \end{equation}
   We see that for $x \ll 0$ we have
			\begin{equation}\label{eq:V_largex}
				\begin{split}
					\mathcal{V}^{(q)}(\phi;x) &=
					\logq (A+C)=\logq[C+(A+B)-B]= \logq\left[\tfrac{x^2}{2}  + 1 - \mathsf{r}(-x/\sqrt{2},\phi) \right],
				\end{split}
			\end{equation}
   where we define
			\begin{equation*}
				\mathsf{r}(M,\phi) := \int_{M}^\infty \overline{\phi}(y) \diff y .
			\end{equation*}
   From \eqref{eq:vershik_kerov_expression} we see that
			\begin{equation*}
				I_{\mathrm{hook}} (\phi) \ge - \tfrac{1}{2} + 2 \int_{y>M}  \overline{\phi}(y) \arccosh \left( {y}/{\sqrt{2}} \right) \diff y \ge -\tfrac{1}{2} + \mathsf{r}(M,\phi) \arccosh \left( {M}/{\sqrt{2}} \right).
			\end{equation*}
			Setting $M=-x/\sqrt{2\kappa}$, which we can assume to grow linearly with $x$ since a minimizing $\kappa$ remains bounded by \Cref{lem:range_kappa}, we have
			\begin{equation*}
				\begin{split}
					\mathcal{W}^{(q)}(\kappa,\phi;x) &= 1+\kappa \log \kappa + 2\kappa I_{\mathrm{hook}} (\phi) + \kappa \logq \left(\frac{x^2}{2 \kappa }  + 1 - \mathsf{r}(-x/\sqrt{2\kappa},\phi) \right).
					\\
					& \ge 1+\kappa \log \kappa + \kappa \big(-1 +2 \mathsf{r}(-x/\sqrt{2\kappa},\phi) \arccosh(-x/2\sqrt{\kappa}) \big) 
					\\
					& \qquad \qquad + \kappa \logq \left(\frac{x^2}{2 \kappa }  + 1 - \mathsf{r}(-x/\sqrt{2\kappa},\phi) \right)
					\\
					& = 1 + (\logq -1) \kappa + \kappa \log \kappa + \frac{\logq}{2} x^2
					\\&
					\qquad \qquad + \kappa \mathsf{r}(-x/\sqrt{2\kappa},\phi) \big( 2 \arccosh(-x/2\sqrt{\kappa}) - \logq
					\big).
				\end{split}
			\end{equation*}
			It is clear that the above expression grows when $\kappa$ becomes large and therefore to minimize it we have to keep $\kappa$ bounded. Then, when $-x$ becomes large enough so that
				$2 \arccosh(-x/2\sqrt{\kappa}) \ge \logq$,
			we can further write
			\begin{equation*}
				\mathcal{W}^{(q)}(\kappa,\phi;x) \ge 1 + (\logq -1) \kappa + \kappa \log \kappa + \frac{\logq}{2} x^2=\mathcal{W}^{(q)}(\kappa,\phi_{\mathrm{VKLS}};x),
			\end{equation*}
			for $x\ll0$ sufficiently large. Here the second equality is due to \eqref{eq:wbd}. This implies that $\phi_{\mathrm{VKLS}}$ minimizes $\mathcal{W}^{(q)}(\kappa , \cdot \, ; x)$ for $x$ negative large enough. Minimizing $\mathcal{W}^{(q)}(\kappa , \phi_{\mathrm{VKLS}} ; x)$ in the parameter $\kappa$ we obtain (the minimum is attained at $\kappa=q$)
			\begin{equation*}
				\min_{\kappa>0} \big\{ \mathcal{W}^{(q)}(\kappa , \phi_{\mathrm{VKLS}} ; x) \big\} =  \frac{\logq}{2} x^2 + (1-q),
			\end{equation*}
			which completes the proof. 
\end{proof}

\subsection{Continuity of different functionals} \label{sec:cont} In this subsection, we state various continuity properties of different functions introduced in the previous subsection. We first state a continuity property of the Sobolev norm $\| \, \cdot \, \|_1$ defined in \eqref{def:sob} when restricted to the subspace of functions in $\overline{\mathcal{Y}}_1$ defined in \eqref{def:y1bar}.

\begin{proposition} \label{prop:L1_continuity}
	Let $\{h_n\}_{n\in \mathbb{N}} \in \overline{\mathcal{Y}}_1$ such that $h_n \xrightarrow[n\to +\infty]{}h \in \overline{\mathcal{Y}}_1$ in $L^1$ norm. Let $g\in \overline{\mathcal{Y}}_1$ and assume also that $\| g \|_1^2, \| h \|_1^2, \| h_n \|_1^2 < +\infty$ for all $n$. Then we have
	\begin{equation*}
		\lim_{n \to +\infty}\| g- h_n \|_1^2 = \| g- h \|_1^2.
	\end{equation*}
\end{proposition}
\begin{proof} Fix any $u\in\overline{\mathcal{Y}}_1$. We first claim that 
\begin{equation}
    \label{eq:twobds}
    \norm{u}_{L^\infty} \le \sqrt{2}, \qquad \norm{u'}_{L^2}\le 4\sqrt2.
\end{equation}
We prove the above two inequalities in the two bullet points below.
\begin{itemize}[leftmargin=10pt]
    \item Note that if $f \in {\mathcal{Y}}_1$, then the area under $f$ contains the square with corner points $(0,0)$ and $(\widetilde{f}(0)/\sqrt{2},\widetilde{f}(0)/\sqrt{2})$. This forces $\widetilde{f}(0) \le \sqrt{2}$. Thus for any $u\in \overline{\mathcal{Y}}_1$, we have $|u(y)| \le \sqrt{2}$.
    \item For the derivative, we have $u'(y)^2 \le 2|u'(y)|$ and hence
\begin{align}\label{eq:onebd}
    \norm{u'}_{L^2} \le 2\norm{u'}_{L^1} =-2\int_{-\infty}^0 u'(x)\diff x+2\int_{0}^\infty u'(x)\diff x =4u(0) \le 4\sqrt{2}.
\end{align}
\end{itemize}

Let us now turn towards the proof of \Cref{prop:L1_continuity}. Set $\Delta h_n:= h - h_n$ and $v:=g-h$. Then we have
	\begin{equation*}
		\begin{split}
			\left| \| g-h \|_1^2 - \| g-h_n \|_1^2 \right| &= \left| \| v \|_1^2 - \| \Delta h_n + v \|_1^2 \right|
			\\
			& = \left| \int_\mathbb{R} |\omega| \left( |\widehat{v} (\omega)|^2 - ( \widehat{\Delta h_n}(\omega) + \widehat{v}(\omega) )\overline{(\widehat{\Delta h_n}(\omega) + \widehat{v}(\omega))} \right) \diff \omega \right|
			\\
			& \le \int_{\mathbb{R}} \left| \widehat{\Delta h_n}(\omega) \right| \left(  \left| \omega \widehat{\Delta h_n}(\omega) \right| + 2\left| \omega \widehat{v}(\omega) \right| \right) \diff \omega
			\\
			& \le \| \Delta h_n \|_{L^2} \left( \| \Delta h_n' \|_{L^2} + 2 \| v' \|_{L^2} \right)
			\\
			& \le \| \Delta h_n \|_{L^2} \left( \| h_n' \|_{L^2} + 3\| h' \|_{L^2}+ \| g' \|_{L^2} \right).
		\end{split}
	\end{equation*}
By the first bound in \eqref{eq:twobds}, we have $\| \Delta h_n \|_{L^2} \le \big(\| h \|_{L^\infty} +\| h_n \|_{L^\infty}\big) \cdot \| \Delta h_n \|_{L^1} \le 2\sqrt{2}\| \Delta h_n \|_{L^1}$. Invoking the second bound in \eqref{eq:twobds} we thus obtain
	\begin{equation*}
		\left| \| g-h \|_1^2 - \| g-h_n \|_1^2 \right| \le 80 \cdot \| \Delta h_n \|_{L^1} ,
	\end{equation*}
	which tends to zero by the hypothesis, completing the proof.
\end{proof}

We now provide a continuity-type result for the hook functional defined in \eqref{eq:vershik_kerov_expression}.

\begin{proposition}\label{p:W_cont}
	Fix $\phi \in \mathcal{Y}_1$ such that $I_{\operatorname{hook}}(\phi) < \infty$. Then, there exists a sequence of partitions $\lambda^{(n)}$ with $\lambda^{(n)} \vdash n$ such that
	\begin{equation}\label{eq:hookconv}
	I_{\operatorname{hook}}(\phi_{\lambda^{(n)}})	  \xrightarrow[n\to +\infty]{} I_{\operatorname{hook}}(\phi).
	\end{equation}
	Furthermore, $\phi_{\lambda^{(n)}}$ converges to $\phi$ in $L^1$.
		\end{proposition}

		\begin{proof}
			Fix $\varepsilon>0$. 
			Assume first that the function $\phi$ is such that $\phi(0),\phi^{-1}(0)<+\infty$ and set $a=\phi(0), b= \phi^{-1}(0)$. This implies that the transformed function $\overline{\phi}$ defined as in \eqref{eq:phi_bar} is compactly supported. Define the partition
			\begin{equation*}
				\mu_i = \left\lfloor \sqrt{n} \phi( i/\sqrt{n} ) \right\rfloor,
			\end{equation*}
			which, in words is the largest partition to fit below the graph of $x\to \sqrt{n} \phi(x/\sqrt{n})$. Let $m = |\mu|$ and since $\int_0^{+\infty} \sqrt{n}\phi(x/\sqrt{n}) \diff x =n$ we have $0 \le n-m \le \sqrt{2n} (a+b).$
			The second inequality follows from the fact that the length of the graph of the partition $\mu$ is contained within a strip of width $\sqrt{2}$ from the graph of $\sqrt{n}\phi(x/\sqrt{n})$ and that the length of the latter is at most $(a+b)\sqrt{n}$. We can at this point define the partition $\lambda \vdash n$ by adding $n-m$ boxes to the partition $\mu$. There are many ways to do so. One way is to define the sequence $\mu^{(i)}$ as $\mu^{(0)}=\mu$ and
			\begin{equation*}
				\begin{split}
					&\mu^{(2i+1)}=\mu^{(2i)} + \mathbb{e}_{r_i} \quad \text{for } i\ge 0,
					\qquad
	\left(\mu^{(2i)}\right)'=\left(\mu^{(2i-1)}\right)' + \mathbb{e}_{c_i} \quad \text{for } i\ge 1
				\end{split}
			\end{equation*}
			where $\mathbb{e}_k$ are standard basis vectors and indices $r_i,c_i$ are defined by
			\begin{equation*}
				r_i = \min\{ r \ge i : \mu^{(2i)} + \mathbb{e}_{r} \text{ is a partition} \},
				\quad
				c_i = \min\{ c \ge i : \left(\mu^{(2i-1)}\right)' + \mathbb{e}_{c} \text{ is a partition} \}.
			\end{equation*}
			This simple sequential construction is explained by the example below
			\begin{equation*}
				\begin{ytableau}
					*(white) & *(white) \\
					*(white) & *(white) & *(white) \\
					*(white) & *(white) & *(white) \\
					*(white) & *(white) & *(white) & *(white) \\
					*(white) & *(white) & *(white) & *(white) & *(white)
				\end{ytableau}
				\qquad
				\begin{ytableau}
					*(white) & *(white) \\
					*(white) & *(white) & *(white) \\
					*(white) & *(white) & *(white) \\
					*(white) & *(white) & *(white) & *(white) \\
					*(white) & *(white) & *(white) & *(white) & *(white) & *(red)
				\end{ytableau}
				\qquad
				\begin{ytableau}
					*(red) \\
					*(white) & *(white) \\
					*(white) & *(white) & *(white) \\
					*(white) & *(white) & *(white) \\
					*(white) & *(white) & *(white) & *(white) \\
					*(white) & *(white) & *(white) & *(white) & *(white) & *(red)
				\end{ytableau}
				\qquad
				\begin{ytableau}
					*(red) \\
					*(white) & *(white) \\
					*(white) & *(white) & *(white) \\
					*(white) & *(white) & *(white) \\
					*(white) & *(white) & *(white) & *(white) & *(red) \\
					*(white) & *(white) & *(white) & *(white) & *(white) & *(red)
				\end{ytableau}
				\qquad
				\begin{ytableau}
					*(red) & *(red) \\
					*(white) & *(white) \\
					*(white) & *(white) & *(white) \\
					*(white) & *(white) & *(white) \\
					*(white) & *(white) & *(white) & *(white) & *(red) \\
					*(white) & *(white) & *(white) & *(white) & *(white) & *(red)
				\end{ytableau}
				\qquad
				\cdots,
			\end{equation*}
			where the first partition should be $\mu$ and red cells represent cells that are added as we move on with the sequence $\mu^{(i)}$. Then we define $\lambda:= \mu^{(n-m)}$ and it is clear by construction that
			\begin{equation*}
				\| \phi_\lambda - \phi \|_{L^1} \le \frac{2\sqrt{2}(a+b)}{\sqrt{n}}.
			\end{equation*}
			This is because the graph of the partition $\lambda$ stays within a strip of length $2\sqrt{2}$ around the graph of $\sqrt{n}\phi(x/\sqrt{n})$. Then, we have
			\begin{equation*} 
				\lim_{n \to + \infty} I_\mathrm{hook}(\phi_\lambda) = I_\mathrm{hook}(\phi),
			\end{equation*}
			which follows since
			\begin{equation*}
				\lim_{n \to + \infty} \| \Omega - \overline{\phi}_\lambda \|_1^2 = \| \Omega - \overline{\phi} \|_1^2,
				\quad
				\lim_{n \to + \infty} \int\limits_{|y|>1} \overline{\phi}_\lambda(y) \arccosh|y| \diff y = \int\limits_{|y|>1} \overline{\phi}(y) \arccosh|y| \diff y,
			\end{equation*}
			where the first limit follows from \Cref{prop:L1_continuity}, while the second uses the fact that $\overline{\phi}, \overline{\phi}_\lambda$ are compactly supported.

			We now would like to remove the assumption that $\phi(0),\phi^{-1}(0)$ are bounded to allow shapes $\phi$ with possibly infinite tails. Let therefore $\phi$ be an arbitrary shape in $\mathcal{Y}_1$. For any $K>0$ we define the truncation $\phi_K(x) := (\phi(x)\wedge K)\mathbf{1}_{x<K}$. It is clear that $\phi_K \in \mathcal{Y}$ and $\phi_K \xrightarrow[K\to +\infty]{} \phi$ in $L^1$. Moreover, since the convergence of $\overline{\phi}_K$ to $\overline{\phi}$ is monotone we have that
			\begin{equation*}
				\lim_{K \to + \infty} I_\mathrm{hook}(\phi_K) = I_\mathrm{hook}(\phi).
			\end{equation*}
			Here the monotone convergence was needed to control the convergence of the integral of $\phi_K(y)$ against $\arccosh|y|$. Let $\theta_K=\|\phi_K\|$ and since $\phi_K$ converges in $L^1$ norm to $\phi$ we also have $\theta_K \to 1$ as $K \to +\infty$. Defining 
			\begin{equation*}
				\psi_K (x) = \frac{1}{\sqrt{\theta_K}} \phi_K( \sqrt{\theta_K} x),
			\end{equation*}
			we have $\| \psi_K \|=1$ and $I_\mathrm{hook}(\psi_K) = \theta_K I_\mathrm{hook}(\phi_k).$
			Since $\psi_K$ satisfies $\psi_K(0),\psi_K^{-1}(0)<+\infty$ we can now use the previous part of the proof to find an $n$ large enough and a partition $\lambda^{(K)}\vdash n$ such that $I_\mathrm{hook}(\phi_{\lambda^{(K)}})$ is arbitrarily close to $I_\mathrm{hook}(\psi_K)$. Then we have
			\begin{equation*}
				\begin{split}
					& \left| I_\mathrm{hook}(\phi_{\lambda^{(K)}}) -  I_\mathrm{hook}(\phi) \right|
					\\ & \le \left| I_\mathrm{hook}(\phi_{\lambda^{(K)}}) - I_\mathrm{hook}(\psi_K)  \right| + \left| I_\mathrm{hook}(\psi_K) - I_\mathrm{hook}(\phi_K) \right| + \left| I_\mathrm{hook}(\phi_K) - I_\mathrm{hook}(\phi) \right|
					\\
					&
					= \left| I_\mathrm{hook}(\phi_{\lambda^{(K)}}) - I_\mathrm{hook}(\psi_K)  \right|+ (1-\theta_K) \left| I_\mathrm{hook}(\phi_K) \right| + \left| I_\mathrm{hook}(\phi_K) - I_\mathrm{hook}(\phi) \right|.
				\end{split}
			\end{equation*}
			Finally, choosing $K$ large enough we can make the last two terms smaller than $\varepsilon/3$ and later letting $n$ grow large we can find $n_\varepsilon$ such that the first term is smaller than $\varepsilon/3$. This  completes the proof of the proposition.
		\end{proof}

  The following proposition discusses continuity for the functional $\mathcal{V}^{(q)}$ defined in  \eqref{def:Vq}.
		\begin{proposition}\label{p:vqconv}
			Fix {$y\in \R$}, and $\phi \in \mathcal{Y}_1$.  Take any sequence $\{y_n\}_{n\ge1}$ such that $y_n\to y$. Take any sequence $\{\phi_n\}_{n\ge1}$ in $\mathcal{Y}_1$  such that $\phi_n\to \phi$ in $L^1$. Then, 
			\begin{equation}\label{eq:vqconv}
				\mathcal{V}^{(q)}(y_n;\phi_{n}) \xrightarrow[n\to +\infty]{} \mathcal{V}^{(q)}(y;\phi)
			\end{equation}
		\end{proposition}
		
		\begin{proof} Using the fact that $|[a]_+-[b]_+|\le |a-b|$, we deduce that
			$$\int_0^\infty \big|[\phi_n(z)-z-y_n]_+-[\phi(z)-z-y_n]_+\big|dz\le \int_0^\infty |\phi_n(z)-\phi(z)|dz \to 0.$$
			Thus it suffices to show
			\begin{align}\label{e:vqpart}
				\int_0^\infty \big|[\phi(z)-z-y]_+-[\phi(z)-z-y_n]_+\big|\diff z \to 0.
			\end{align}
			Note that $|[\phi(z)-z-y]_+-[\phi(z)-z-y_n]_+| \to 0$
			pointwise. Set $w:=y \wedge \min \{y_n \mid n\ge 1\}$. Using the fact that $|[a]_+-[b]_+|\le [a]_+\vee [b]_+$ we deduce that  $$|[\phi(z)-z-y]_+-[\phi(z)-z-y_n]_+| \le \phi(z)+[-z-w]_+,$$ which is integrable. Thus by dominated convergence theorem, we arrive at \eqref{e:vqpart}.
		\end{proof}

\subsection{Existence of lower-tail rate function $\mathcal{F}$} \label{sec:exists} In this section we use probabilistic arguments to show the existence of a lower-tail rate function for the first row of shifted cylindric Plancherel measure (\Cref{thm:limit_q_Laplace}). The rate function is given by $\mathcal{F}$ defined in \eqref{eq:f_Plancherel}. For the unshifted measure, we shall provide sharp lower-tail estimates for the first row in \Cref{prop:expapprox}. The existence of the rate function for the unshifted ones requires further argument involving convex analysis which we postpone to later subsections.
  
		\begin{theorem} \label{thm:limit_q_Laplace} Let $\h$ be the height function of the $q$-PNG with intensity $\Lambda=2(1-q)$ and droplet initial condition. Let $\lambda/\rho\sim \Pr_{\mathsf{cPlan}(t(1-q))}$, $\chi\sim q$-$\mathrm{Geo}(q)$, and $S\sim \mathrm{Theta}(q,1)$ all independent from $\h$. For all $x \in \R$ we have
			\begin{equation} \label{eq:limit_q_Laplace}
				-\lim_{t\to \infty}   \frac{1}{t^2} \log \mathbb{P}(\h(0,t)+\chi+S\le xt)=-\lim_{t\to \infty}   \frac{1}{t^2} \log \mathbb{P}_{\mathsf{cPlan}(t(1-q))}(\lambda_1+S\le xt)=\mathcal{F}(x).
			\end{equation}
		\end{theorem}
		
		\begin{proof} The first equality in \eqref{eq:limit_q_Laplace} is a consequence of \Cref{thm:matching_qPNG_cylindric_plancherel}. We focus on proving that the limit exists and is given by $\mathcal{F}(x)$. From \Cref{cor:iden} (with $\theta=\sqrt{2}(1-q)$, $\zeta=1$, and $s=\lfloor xt \rfloor$), we have
  $$\mathbb{P}(\h(0,t)+\chi+S\le xt)=\mathbb{E}_{\mathsf{Plan}(t)} \left( \prod_{i=1}^{\infty} \frac{1}{1+q^{\lfloor xt \rfloor+i-\lambda_i}} \right).$$
  Hence it suffices to analyze the right-hand side of the above equation. {This analysis essentially follows the idea of Varadhan's lemma.} For clarity, we divide the proof into two steps. In \textbf{Step 1} we prove the theorem assuming a technical estimate \eqref{def:approxf}, which in turn is proven in \textbf{Step 2}.

			\noindent\textbf{Step 1.} Fix any $x\in \R$ and $\e>0$. Since $\mathcal{F}$ is obtained by minimizing $\mathcal{W}^{(q)}$ (see \eqref{eq:f_Plancherel}), get ${\kappa}_*, {\phi}_*$ (depending on $\e$) such that
   $$\mathcal{W}^{(q)}(\kappa_*,\phi_*;x) \le \mathcal{F}(x)+\e.$$
   Due to \Cref{lem:range_kappa} we may choose $\kappa_*$ so that $\kappa_*\le \kappa^*$ defined in \Cref{lem:range_kappa}.

             \noindent\textbf{Lower Bound.} Fix an $M>\max\{{\kappa}^*,\mathcal{F}(x),e^2\}$. Using tail estimates for Poisson random variable $X\sim \mathsf{Poi}(t^2)$ we have
			\begin{align}\label{eq:poi}
				\mathbb{P}_{\mathsf{Plan}(t)} \left(|\lambda| >Mt^2\right)=\mathbb{P}_{\mathsf{Poi}(t^2)} \left(X > Mt^2\right) \le e^{-Mt^2}.
			\end{align}
			We claim that 
			\begin{align}\label{eq:trunc}
				\mathbb{E}_{\mathsf{Plan}(t)} \left( \prod_{i=1}^{\infty} \frac{1}{1+q^{\lfloor xt \rfloor+i-\lambda_i}}\ind_{|\lambda|\le Mt^2} \right) \le e^{-t^2 \mathcal{F}(x) + O(t)}.
			\end{align}
			Combining the above claim with \eqref{eq:poi}, by the choice of our $M$ we see that
			\begin{align*}
				\liminf_{t\to\infty} -\frac1{t^2}\log \mathbb{E}_{\mathsf{Plan}(t)} \left( \prod_{i=1}^{\infty} \frac{1}{1+q^{\lfloor xt \rfloor+i-\lambda_i}} \right)  \ge \mathcal{F}(x),
			\end{align*}
			which verifies the lower bound. Let us now focus on proving \eqref{eq:trunc}.  We expand the expectation of the $q$-product over the Plancherel measure
			\begin{equation} \label{eq:expectation_qL_expanded}
			\begin{aligned}	\mathbb{E}_{\mathsf{Plan}(t)} \left( \prod_{i=1}^{\infty} \frac{1}{1+q^{\lfloor xt \rfloor+i-\lambda_i}}\ind_{|\lambda|\le Mt^2} \right) & \le \mathbb{E}_{\mathsf{Plan}(t)} \left( \prod_{i=1}^{\infty} \frac{1}{1+q^{xt +i-\lambda_i}}\ind_{|\lambda|\le Mt^2} \right)
				\\ & = \sum_{n\in \Z\cap[0,Mt^2]} e^{-t^2} \frac{t^{2n}}{n!} \sum_{\lambda \vdash n }  \frac{(f^\lambda)^2}{n!} \prod_{i=1}^{\infty} \frac{1}{1+q^{x t+i-\lambda_i}}.
    \end{aligned}
			\end{equation}
			We will produce estimates of the various factors appearing in the summation on the right-hand side of \eqref{eq:expectation_qL_expanded}. Fix a partition $\lambda$ with $|\lambda|\le Mt^2$.
			 We claim that 		\begin{equation}\label{def:approxf}
				\prod_{i=1}^{\infty} \frac{1}{1+q^{x t+i-\lambda_i}} = e^{-n \mathcal{V}^{(q)}(xt/\sqrt{n}; \phi_\lambda) + O(t)}, \quad \prod_{i=1}^{\infty} \frac{1}{1+q^{\lfloor x t\rfloor+i-\lambda_i}} = e^{-n \mathcal{V}^{(q)}(\lfloor xt \rfloor/\sqrt{n}; \phi_\lambda) + O(t)}
			\end{equation}
			where $n=|\lambda|$ the size of the partition $\lambda$ and $\mathcal{V}^{(q)}$ and $\phi_\lambda$ are defined in \eqref{def:Vq} and \eqref{eq:defphi} respectively. The error term $O(t)$ appearing above depends only on $M$ and $x$. We assume \eqref{def:approxf} for the moment and complete the proof of the theorem. Setting $n=\kappa t^2$, from \Cref{prop:plancherel_hook_integral} and the approximation of the density of the Poisson distribution, we have
			\begin{equation}\label{eq:hookappr}
				\frac{(f^\lambda)^2}{n!} =  e^{-t^2 \kappa( 1+ 2 I_{\mathrm{hook}} (\phi_\lambda)) + O(t\log t)}, \quad  e^{-t^2}\frac{t^{2n}}{n!} = e^{- t^2 (1-\kappa + \kappa \log \kappa) + O(\log t)}
			\end{equation}
			respectively.
			Combining the above two estimates with \eqref{def:approxf}, we are able to write 
			\begin{equation*} 
				\begin{split}
					\mbox{r.h.s.~of \eqref{eq:expectation_qL_expanded}} 
					& \le \sum_{\kappa \in \frac1{t^2}\Z \cap [0,Mt^2]} \sum_{\lambda \vdash \kappa t^2 } e^{-t^2\mathcal{W}^{(q)}(\kappa,\phi_{\lambda};x)+O(t)}  \\ & \le e^{-t^2 \mathcal{F}(x)+O(t)} \sum_{n\in \Z\cap[0,Mt^2]} \mathsf{p}_n \le e^{-t^2 \mathcal{F}(x) + O(t)}.
				\end{split}
			\end{equation*}
			Above, $\mathsf{p}_n$ denotes the number of partitions of $n$ defined in \eqref{eq:p_n}. The second inequality uses the fact that $\mathcal{F}$ is the minimizer of the $\mathcal{W}^{(q)}$ functional and the third inequality uses the bound $\mathsf{p}_n \le e^{\Con \sqrt{n}}$ from \eqref{eq:pnbd}. This proves \eqref{eq:trunc}.

			\noindent\textbf{Upper Bound.}  For the upper bound we set $n= \lfloor \kappa_* t^2 \rfloor$ and take $\lambda^{(t)}\vdash \lfloor \kappa_* t^2 \rfloor$ so that \Cref{p:W_cont} holds for $\phi=\phi_*$. We use the immediate lower bound:
			\begin{align*}
				\mathbb{E}_{\mathsf{Plan}(t)} \left( \prod_{i=1}^{\infty} \frac{1}{1+q^{\lfloor x t\rfloor+i-\lambda_i}} \right) & \ge e^{-t^2} \frac{t^{2n}}{n!} \frac{ (f^{\lambda^{(t)}})^2}{n!} \prod_{i=1}^{\infty} \frac{1}{1+q^{\lfloor x t\rfloor+i-\lambda_i^{(t)}}}.
			\end{align*}
			Taking logarithms on both sides, dividing by $-t^2$, and using the asymptotics from \eqref{def:approxf} and the second equality from \eqref{eq:hookappr} we get
			\begin{align*}
				-\frac1{t^2}\log \mathbb{E}_{\mathsf{Plan}(t)} \left( \prod_{i=1}^{\infty} \frac{1}{1+q^{\lfloor x t\rfloor+i-\lambda_i}} \right) & \le (1-\tfrac{n}{t^2}+\tfrac{n}{t^2}\log\tfrac{n}{t^2})-\tfrac1{t^2}\log\big(\tfrac1{n!} (f^{\lambda^{(t)}})^2\big)\\ & \hspace{1cm}+ \tfrac{n}{t^2}\mathcal{V}^{(q)}(\lfloor x t\rfloor/\sqrt{n};\phi_{\lambda^{(t)}}))+O(1/t).
			\end{align*}
			Since $n/t^2\to \kappa_*$ and $\lfloor x t\rfloor/t\to x$, combining \Cref{p:W_cont,p:vqconv} we have that 
			\begin{align*}
				& \limsup_{t\to\infty} -\frac1{t^2}\log \mathbb{E}_{\mathsf{Plan}(t)} \left( \prod_{i=1}^{\infty} \frac{1}{1+q^{\lfloor x t\rfloor+i-\lambda_i}} \right) \\ & \le (1-\kappa_*+\kappa_*\log\kappa_*)-\kappa_*(-1-2I_{\operatorname{hook}}({\phi}_*))+ \kappa_*\mathcal{V}^{(q)}(x/\sqrt{\kappa_*};{\phi}_*) \\ & = \mathcal{W}^{(q)}(\kappa_*,\phi_*;x) \le \mathcal{F}(x)+\e.
			\end{align*}
			Since $\e>0$ is arbitrary, taking $\e\downarrow 0$ produces a matching upper bound.

			\noindent\textbf{Step 2.} In this step we prove \eqref{def:approxf}. We shall only prove the first estimate in \eqref{def:approxf}, the argument for the second one is exactly the same. Let us write
			\begin{equation*}
				\prod_{i=1}^{\infty} \frac{1}{1+q^{x t+i-\lambda_i}} = \exp \left\{ - \sum_{i=1}^{\infty} f(i/\sqrt{n}) \right\}, \mbox{ where }f(y):=\log\big(1+q^{xt+y\sqrt{n}-\lambda_{\lceil\sqrt{n} y\rceil}}\big).
			\end{equation*}
   We rely on the following two estimates:
		\begin{align}\label{e:approxf1}
				& \sqrt{n}\int_0^{\infty} f(y)\d y -\sum_{i=1}^\infty f(i/\sqrt{n})=O(t), \\
				\label{e:approxf2}
				& \sqrt{n}\int_0^{\infty} f(y)\d y -n \mathcal{V}^{(q)}(xt/\sqrt{n}; \phi_\lambda)=O(1),
			\end{align}
		where the error term $O(t)$ depends on $M,x$ as well. Recall that $n=|\lambda|\le Mt^2$.     \eqref{def:approxf} follows from the above two estimates.

			\noindent\textbf{Proof of \eqref{e:approxf1}.} Observe that $f$ is decreasing. Hence the left-hand side of \eqref{e:approxf1} is positive. It thus suffices to look for an upper bound for the same. Observe that
			\begin{align} \nonumber
				\sqrt{n}\int_0^\infty f(y)\diff y -\sum_{i=1}^\infty f(i/\sqrt{n}) &  = \sum_{i=1}^n\sqrt{n}\int_{(i-1)/\sqrt{n}}^{i/\sqrt{n}} (f(y)-f(i/\sqrt{n}))\diff y \\ \nonumber &  = \sum_{i=1}^{\infty} \sqrt{n}\int_{(i-1)/\sqrt{n}}^{i/\sqrt{n}} \log \left[1+\frac{q^{xt+i-\lambda_i}(q^{y\sqrt{n}-i}-1)}{1+q^{xt+i-\lambda_i}}\right]\diff y  \\ \nonumber &  = \sum_{i=1}^{\infty}\int_{0}^1 \log \left[1+\frac{q^{xt+i-\lambda_i}(q^{-z}-1)}{1+q^{xt+i-\lambda_i}}\right]\diff z \\ & \le q^{-1}\sum_{i=1}^\infty \frac{q^{xt+i-\lambda_i}}{1+q^{xt+i-\lambda_i}}, \label{e:v22}
			\end{align}
			where in the last line we use the inequality $\log(1+z)\le z$ for each term in the summand. We now split the above sum into two parts:
			\begin{align*}
				\sum_{i=1}^\infty \frac{q^{xt+i-\lambda_i}}{1+q^{xt+i-\lambda_i}} =\left[\sum_{i=1}^{k}+\sum_{i=k}^\infty \right] \frac{q^{xt+i-\lambda_i}}{1+q^{xt+i-\lambda_i}}.
			\end{align*}
			For the first sum, we bound each term by $1$ and get that the sum is at most $k$. Note that for each $i$, we have $\lambda_i\le \frac1i(\lambda_1+\cdots+\lambda_i) \le n/i$. Thus, for each term in the second sum, we have $\frac{q^{xt+i-\lambda_i}}{1+q^{xt+i-\lambda_i}} \le q^{xt+i-n/i} \le q^{i-k}q^{xt+k-Mt^2/k}$. Now if we choose $k=\lceil t(\sqrt{x^2+4M}-x) \rceil$, this ensures $xt+k-Mt^2/k \ge 0$. Then the right-hand side is at most $O(t)$.

			\noindent\textbf{Proof of \eqref{e:approxf2}.} Using the elementary inequality: $[v]_+ \le \log(1+e^v)$ we obtain
			\begin{align*}
				\sqrt{n}\int_0^{\infty}f(y)\d y \ge n\logq\int_0^{\infty} [\phi_\lambda(y)-y-xt/\sqrt{n}]_{+}\d y.
			\end{align*}
We thus focus on proving
   \begin{align}\label{e:apr4}
				\sqrt{n}\int_0^{\infty}f(y)\d y \le n\logq\int_0^{\infty} [\phi_\lambda(y)-y-xt/\sqrt{n}]_{+}\d y+O(1).
			\end{align}
For each $k\ge 1$, define $a_k:=\log(1+q^{xt+k-\lambda_k})$ and $b_k:=\log(1+q^{xt+k-1-\lambda_k})$. Note that 
$$b_1\ge a_1\ge b_2 \ge a_2 \ge b_3 \ge a_3 \ge \cdots.$$
Making the change of variable $u=\log(1+q^{xt+y-\lambda_k})$ we get
\begin{equation}
    \label{e:apr5}
     \begin{aligned}
				\sqrt{n}\int_{(k-1)/\sqrt{n}}^{k/\sqrt{n}}f(y)\d y & = \int_{k-1}^k \log(1+q^{xt+y-\lambda_k})\d y \\ & = \frac1{\logq} \int_{a_k}^{b_k} \frac{ue^u}{e^u-1}\d u   = \frac1{\logq}\frac{b_k^2-a_k^2}{2}+\frac1{\logq}\int_{a_k}^{b_k} \frac{u}{e^u-1}\d u.
\end{aligned}
\end{equation}
 Note that
   $0\le b_k-a_k \le \logq$, and $\sum_{k\ge 1}\int_{a_k}^{b_k} \frac{u}{e^u-1}\d u \le \int_0^{\infty} \frac{u}{e^u-1}\d u=\frac{\pi^2}{6}$. Summing over $k\ge 1$ on both sides of \eqref{e:apr5} and applying these inequalities we get 
   \begin{align}
       \label{e:apr6}
				\sqrt{n}\int_{0}^{\infty} f(y)\d y \le \frac{\pi^2}{6\logq} +\frac1{2}\sum_{k\ge 1} (b_k+a_k).
   \end{align}
 We now partition the index set $\mathbb{N}$ into three sets:
   \begin{align*}
       I_1:=\{k\in \mathbb{N} : xt+k-\lambda_k\le 0\}, \ \ I_2:=\{k\in \mathbb{N} : xt+k-1-\lambda_k\ge 0\}, \ \  I_3:=\mathbb{N} \setminus (I_1\cup I_2).
   \end{align*}
 Note that the cardinality of $I_3$ is at most $1$, and if $k\in I_3$, we have $a_k,b_k \le \log(1+q^{-1})$. For $I_2$, using the $\log(1+x)\le x$ inequality, we observe that
 \begin{align*}
       \sum_{k\in I_2} a_k & = \sum_{k\in I_2} \log(1+q^{xt+k-\lambda_k})  \le \sum_{k\in I_2} q^{xt+k-\lambda_k} \le \frac{1}{1-q},
   \end{align*}
   and $\sum_{k\in I_2} b_k \le \frac1{1-q}$ similarly. For $I_1$, using the $\log(1+x)\le x$ inequality, we note that
   \begin{align*}
       \sum_{k\in I_1} [a_k-\logq(\lambda_k-k-xt)] & = \sum_{k\in I_1} \log(1+q^{\lambda_k-k-xt})  \le \sum_{k\in I_1} q^{\lambda_k-k-xt} \le \frac{1}{1-q},
   \end{align*}
   and $\sum_{k\in I_2} b_k-\logq(\lambda_k-k+1-xt) \le \frac1{1-q}$ similarly. Inserting these bounds back in the right-hand side of \eqref{e:apr6} we get
   \begin{align*}
       \sqrt{n}\int_0^{\infty} f(y)\d y & \le O(1)+\eta\sum_{k\in I_1}\frac12\big((\lambda_k-k+1-xt)+(\lambda_k-k-xt)\big) \\ & =O(1)+\eta\sum_{k\in I_1}\frac12\big((\lambda_k-k+1-xt)^2-(\lambda_k-k-xt)^2\big) \\ & =O(1)+\eta\sum_{k\in I_1}\sqrt{n}\int_{(k-1)/\sqrt{n}}^{k/\sqrt{n}} (\lambda_{\lfloor y\sqrt{n}\rfloor +1}-y\sqrt{n}-xt)\d z \\ & \le O(1)+n\eta\int_0^\infty [\phi_{\lambda}(y)-y-xt/\sqrt{n}]_+\d y. 
   \end{align*}
   This proves \eqref{e:apr4}, completing the proof.
		\end{proof}
		
	Next, we provide an approximation for the lower-tail rate function of the cylindric Plancherel measure that will be useful throughout. For $t>0$ and $a\ge 0$ we define
			\begin{equation}\label{def:tta}
				\mathcal{T}_t(a) := -\frac{1}{t^2} \log\left\{ \sup_{ \substack{\rho \subset \lambda \\ \lambda_1 \le a t} } \mathbb{P}_{\mathsf{cPlan}(t(1-q))}(\lambda/\rho) \right\}.
			\end{equation}
{As a convention we set $\mathcal{T}_t(a):=\infty$ for $a<0$.}
\begin{proposition}\label{proptt0} For each $t>0$, $\mathcal{T}_t(a)$ is nonnegative and decreasing in $a$. Moreover, for all $t>0$ we have $\mathcal{T}_t(0)=1-q$.
\end{proposition}
\begin{proof} The fact that $\mathcal{T}_t$ is nonnegative and decreasing follows immediately from the definition of $\mathcal{T}_t(a)$. Notice that 
				\begin{equation*}
					 \sup_{ \substack{\rho \subset \lambda \\ \lambda_1=0} } \mathbb{P}_{\mathsf{cPlan}(t(1-q))}(\lambda/\rho) = \mathbb{P}_{\mathsf{cPlan}(t(1-q))}(\varnothing/\varnothing)=e^{-t^2(1-q)}.
				\end{equation*}
				$-\frac{1}{t^2}\log$ of the left- and right-hand sides of the above equality yields $\mathcal{T}_t(0)=1-q$.
\end{proof}

			\begin{proposition}\label{prop:expapprox}
				There exists a constant $\Con=\Con(q)>0$ such that for all $a\ge 0$ we have
				\begin{equation} \label{eq:exp_approximation_lower_tail_lambda1}
					e^{-t^2\mathcal{T}_t(a)} \le \mathbb{P}_{\mathsf{cPlan}(t(1-q))} (\lambda_1 \le at) \le e^{-t^2 \mathcal{T}_t(a) + \Con \cdot t}.
				\end{equation}
			\end{proposition}
			
			\begin{proof} The lower bound in \eqref{eq:exp_approximation_lower_tail_lambda1} is obvious from the definition of $\mathcal{T}_t(a)$.  We focus on the upper bound. Let us set $\theta:=\frac12\logq$ so that $e^{\theta}=q^{-1/2}$. Set $M:=\max\{3,3\theta^{-1}\}$.   By a union bound we have
   \begin{equation}
       \label{eq:3terms}
        \begin{aligned}
       \mathbb{P}_{\mathsf{cPlan}(t(1-q))} (\lambda_1 \le at) & \le \mathbb{P}_{\mathsf{cPlan}(t(1-q))} (|\lambda / \rho| > Mt^2)+ \mathbb{P}_{\mathsf{cPlan}(t(1-q))} (|\rho| > Mt^2) \\ & \hspace{1cm}+\mathbb{P}_{\mathsf{cPlan}(t(1-q))} (\lambda_1 \le at, |\lambda/\rho|<Mt^2,|\rho|<Mt^2).
   \end{aligned}
   \end{equation}
We shall now bound each of the three terms above separately. Recall that by \eqref{eq:law_size_skew_shape} 
, $|\lambda/\rho|\sim \mathsf{Poi}(t^2(1-q))$, hence Markov inequality yields
\begin{equation}\label{eq:1bdlr}
\begin{aligned}
     \mathbb{P}_{\mathsf{cPlan}(t(1-q))} (|\lambda / \rho| > Mt^2) & \le e^{-Mt^2}\Ex_{\mathsf{Poi}(t^2(1-q))}[e^{|\lambda/\rho|}]\\ & =e^{-Mt^2+t^2(1-q)(e-1)} \le e^{-t^2}.
\end{aligned}
\end{equation}
Furthermore by \eqref{eq:expmombd} we have
\begin{align}\label{eq:2bdlr}
    \mathbb{P}_{\mathsf{cPlan}(t(1-q))} (|\rho| > Mt^2) \le e^{-M\theta t^2}\Ex[e^{\theta |\rho|}] \le \Con \exp\left(-M\theta t^2+\tfrac{t^2(1-q)^2\sqrt{q}}{1-\sqrt{q}}\right) \le \Con e^{-t^2}
\end{align}
For the last term in \eqref{eq:3terms} observe that
				\begin{align*}
				& \mathbb{P}_{\mathsf{cPlan}(t(1-q))} (\lambda_1 \le at, |\lambda/\rho|<Mt^2,|\rho|<Mt^2) \\ & \le \mathbb{P}_{\mathsf{cPlan}(t(1-q))} (\lambda_1 \le at, |\lambda|,|\rho|<2Mt^2)  \\ & = \sum_{\substack{\rho \subset \lambda \\ \lambda_1 \le at, |\lambda|,|\rho|\le 2Mt^2}} \mathbb{P}_{\mathsf{cPlan}(t(1-q))} (\lambda / \rho) \\ & \le e^{-t^2\mathcal{T}_t(a)}\bigg(\sum_{n\in \Z\cap[0,2Mt^2]}\mathsf{p}_n\bigg)^2 \le e^{-t^2\mathcal{T}_t(a)+\Con t}.
				\end{align*}
 where in the last line we used the well known $\mathsf{p}_n\le e^{\Con\sqrt{n}}$ bound. Inserting the above bound along with the bounds in \eqref{eq:1bdlr} and \eqref{eq:2bdlr} back in \eqref{eq:3terms} yields
 \begin{align*}
    \mathbb{P}_{\mathsf{cPlan}(t(1-q))} (\lambda_1 \le at) \le e^{-t^2\mathcal{T}_t(a)+\Con t}+(\Con +1)e^{-t^2}. 
 \end{align*}
 Since $\mathcal{T}_t(a) \le 1-q$ by \Cref{proptt0}, $e^{-t^2\mathcal{T}_t(a)}$ dominates $e^{-t^2}$. Thus adjusting the constant $\Con$ we derive the upper bound in \eqref{eq:exp_approximation_lower_tail_lambda1}. This completes the proof.
 \end{proof}
			
\begin{proposition}\label{prop:tt2} {We have $\lim_{t\to\infty} \mathcal{T}_t(x)=0$ for all $x\geq 2$.}
\end{proposition}
\begin{proof} 
Utilizing the upper bound in \eqref{eq:exp_approximation_lower_tail_lambda1} and the relation \eqref{eq:matching_height_qPNG_cylindric_Plancherel} ($\theta=\sqrt{2}(1-q)$ and $x=0$) we deduce
    $$e^{-t^2\mathcal{T}_t(2)+\Con t} \ge \Pr(\lambda_1\le 2t) \ge \Pr(\chi=0)\Pr(\h(0,t)\le 2t)=(q;q)_{\infty} \cdot \Pr(\h(0,t)\le 2t).$$
    Taking $-\frac1{t^2}\log$ both sides and then taking $\limsup_{t\to\infty}$ we obtain
    $$\limsup_{t\to \infty} \mathcal{T}_t(2) \le  \limsup_{t\to\infty} -\frac1{t^2}\log \Pr(\h(0,t)\le 2t).$$
From the fluctuation result of $q$-PNG height function \cite{aggarwal_borodin_wheeler_tPNG}, we have that $$\lim_{t\to \infty} \Pr(\h(0,t)\le 2t) =\Pr(\operatorname{TW}_{\operatorname{GUE}}\le 0)>0.$$ 
   As $\mathcal{T}_t$ is nonnegative, this implies $\lim_{t\to\infty} \mathcal{T}_t(2)=0$. {The conclusion for general $x\geq 2$ follows from monotonicity.}
\end{proof}

\subsection{Convexity of rate function $\mathcal{F}$} \label{sec:ltconvex} The scope of this subsection is to prove convexity of the lower-tail rate function $\mathcal{F}$ for the edge of the shift-mixed cylindric Plancherel measure. The starting point of our argument is the celebrated Okounkov inequality. Indeed, Schur functions enjoy remarkable log-concavity properties, a version of which was first conjectured by Okounkov in \cite{Okounkov1997}. This conjecture was later proved and refined in \cite{Lam_Postnikov_Pylyavskyy_concavity} by Lam, Postnikov, and Pylyavskyy. We report this result next, and then use it to establish a certain asymptotic midpoint convexity of the family of functions $\mathcal{T}_t$ in \Cref{prop:midpoint_convexity_T}.

\begin{proposition}{\cite[Theorem 12]{Lam_Postnikov_Pylyavskyy_concavity}} \label{prop:Okounkov_ineq} 
	For any pair of skew-partitions $\lambda/\mu,\nu/\rho$ we have 
	\begin{equation*}
		s_{\lceil \frac{\lambda+\nu}{2} \rceil / \lceil \frac{\mu+\rho}{2} \rceil} s_{\lfloor \frac{\lambda+\nu}{2} \rfloor / \lfloor \frac{\mu+\rho}{2} \rfloor} \ge_{\mathrm{s}} s_{\lambda/\mu} s_{\nu/\rho}.
	\end{equation*}
	where for any two symmetric functions $A, B$, $A \ge_{\mathrm{s}} B$ means that the difference $A-B$ possesses an expansion in the basis of Schur functions with positive coefficients. Here the operations $+,/2,\lfloor\cdot\rfloor,\lceil\cdot\rceil$ on partitions are performed coordinate-wise. 
\end{proposition}

			The following proposition establishes a certain asymptotic midpoint convexity of the family of functions $\mathcal{T}_t$.
            
			\begin{proposition} \label{prop:midpoint_convexity_T}
				Let $a, a' \in \R$. Then, for all $\varepsilon>0$, there exists $t_\varepsilon$ such that for all $t>t_\varepsilon$ we have
				\begin{equation}\label{eq:conT}
					\mathcal{T}_t\left( \frac{a+a'}{2} +\frac1t\right) \le \frac{1}{2} \left( \mathcal{T}_t(a) + \mathcal{T}_t(a') \right) + \varepsilon.
				\end{equation}
			\end{proposition}
			\begin{proof} As $\mathcal{T}_t(x)=+\infty$ for $x<0$, we may assume $a,a'\ge 0$. Without loss of generality assume $0\le a\le a'$. By the Okounkov's inequality (\Cref{prop:Okounkov_ineq}) we have, for any pair of skew partitions $\lambda/\mu, \nu/\rho$ we have 
				\begin{equation} \label{eq:generalized_Okounkov_inequality}
					\max\left\{ \left(s_{\lceil \frac{\lambda+\nu}{2} \rceil / \lceil \frac{\mu+\rho}{2} \rceil} \right)^2, \left(s_{\lfloor \frac{\lambda+\nu}{2} \rfloor / \lfloor \frac{\mu+\rho}{2} \rfloor} \right)^2 \right\} \ge s_{\lambda/\mu} s_{\nu/\rho},
				\end{equation}
                whenever the Schur functions are evaluated at some positive specialization.
				Let  $\overline{\frac{\lambda+\nu}{2}} /  \overline{\frac{\mu+\rho}{2}}$ be the skew partition that maximizes the left-hand side of \eqref{eq:generalized_Okounkov_inequality} and set $\Delta := |\overline{\frac{\mu+\rho}{2}}| - \frac{|\mu|+|\rho|}{2}$. Taking the exponential specialization with parameter $\gamma:=t(1-q)$ in all the Schur functions (see \eqref{eq:periodic_schur} and \eqref{eq:Schur_exponential}) we obtain that 
				\begin{equation*}
					\begin{split}
						\mathbb{P}_{\mathsf{cPlan}(t(1-q))} \left( \overline{\frac{\lambda+\nu}{2}} \bigg/ \overline{ \frac{\mu+\rho}{2} } \right) 
						&= e^{-t^2(1-q)} q^{\frac{|\mu|}{2}+\frac{|\rho|}{2}}  \left(   s_{\overline{\frac{\lambda+\nu}{2}} /  \overline{\frac{\mu+\rho}{2}}} \right)^2  q^{\Delta}
						\\
						&
						\ge 
						e^{-t^2(1-q)} q^{\frac{|\mu|}{2} + \frac{|\rho|}{2}}   s_{\lambda/\mu} s_{\nu/\rho} \,\, q^{\Delta}
						\\
						& 
						= \sqrt{\mathbb{P}_{\mathsf{cPlan}(t(1-q))} \left( \lambda / \mu \right) } 
						\,
						\sqrt{ \mathbb{P}_{\mathsf{cPlan}(t(1-q))} \left( \nu / \rho \right) } \,\, q^{\Delta}.
					\end{split}
				\end{equation*}
			Note that $\lambda_1\leq at$ and $\nu_1\leq a't$ implies that the first row of $\overline{\frac{\lambda+\nu}{2}}$ has length $\leq \frac{(a+a')t}{2}+1$. Taking $-\frac{1}{t^2} \log$ of both sides of the previous inequality and optimizing over the choice of $\lambda/\mu$ with $\lambda_1 \le a t$ and $\nu/\rho$ with $\nu_1 \le a't$ we find 
				\begin{equation*}
					\mathcal{T}_t\left( \frac{a+a'}{2} +\frac{1}{t} \right)  \le -\frac{1}{t^2} \log \mathbb{P}_{\mathsf{cPlan}(t(1-q))} \left( \overline{\frac{\lambda+\nu}{2}} \bigg/ \overline{ \frac{\mu+\rho}{2} } \right)   \le \tfrac{1}{2} \left[\mathcal{T}_t(a) + \mathcal{T}_t(a') \right] + \frac{ \log q^\Delta }{t^2},
				\end{equation*}
				Since $\Delta\le 1$, taking $t$ large enough, we get the desired result.
			\end{proof}

			For the next proposition, we introduce the operation of \emph{infimal convolution} between two real-valued functions $g,h$ as
			\begin{equation*}
				(g \oplus h)(x) := \inf_{y \in \mathbb{R}} \left\{ g(y) + h(x-y) \right\}.
			\end{equation*}
			The infimal convolution is the analog of the integral convolution $g*h$ in the $(\inf,+)$ algebra and it is a common object in convex analysis. We use the notion of infimal convolution in the following result.

			\begin{proposition} \label{prop:f_limit_g+T}
				Let $x \in \mathbb{R}$ and define the function $g(x)=\frac{x^2}{2}\logq$. Then we have
				\begin{equation}
                \label{eq:lim_inf_conv}
					\mathcal{F}(x) = \lim_{t \to \infty}  \left( g \oplus \mathcal{T}_t \right)(x).
				\end{equation}
			\end{proposition}
			\begin{proof}  By \Cref{lem:range_kappa}, $\mathcal{F}(x)=0$ for $x\ge2$. On the other hand, we have for $x\ge2$,
\begin{equation*}
0\leq \liminf_{t\to \infty}\left(g \oplus \mathcal{T}_t \right)(x)\leq \limsup_{t\to \infty}\left(g \oplus \mathcal{T}_t \right)(x)\le g(0)+\lim_{t\to \infty}\mathcal{T}_{t}(x)=0
\end{equation*}
by \Cref{prop:tt2}. Hence \eqref{eq:lim_inf_conv} holds for $x\ge2$. Let us fix any $x<2$. Suppose $\lambda/\rho \sim \Pr_{\mathsf{cPlan}(t(1-q))}$ and $S\sim \mathrm{Theta}(q;1)$. For each $i\ge 0$, let us set
				\begin{equation*}
					\overline{y}_i := \mathrm{argmax} \left\{ \mathbb{P}(S = yt) \mathbb{P}(\lambda_1 \le (x-y)t) : y\in \tfrac{1}{t} \mathbb{Z}\cap [x-i,x] \right\},
				\end{equation*}
  and  define
    \begin{align*}
        (g{\oplus}_i\mathcal{T}_t)(x):=\inf_{y\in \frac{1}{t} \mathbb{Z} \cap[x-i,x]} \{g(y)+\mathcal{T}_t(x-y)\}.
    \end{align*}
  We shall only consider $\overline{y}_2$ and $\overline{y}_3$ in our proof.  Observe that
   \begin{equation}\label{eq:upl+S}
   \begin{aligned}
       \mathbb{P}(\lambda_1 + S \le xt) & \ge \Pr(S\in [xt-3t,xt],\lambda_1 + S \le xt) \\ & \ge \mathbb{P}(S = \overline{y}_3t) \mathbb{P}(\lambda_1 \le (x-\overline{y}_3)t),
   \end{aligned}   
   \end{equation}
   and 
   \begin{equation}
       \label{eq:dwl+S}
       \begin{aligned}
       \mathbb{P}(\lambda_1 + S \le xt) & \le \Pr(S\in [xt-2t,xt],\lambda_1 + S \le xt) +\mathbb P(S\le (x-2)t) \\ & \le 2t\cdot \mathbb{P}(S = \overline{y}_2t) \mathbb{P}(\lambda_1 \le (x-\overline{y}_2)t)+\mathbb P(S\le (x-2)t).
   \end{aligned}
   \end{equation}
  Using \eqref{eq:exp_approximation_lower_tail_lambda1} and the explicit law \eqref{eq:S_distribution} (with $\zeta=1$) of the random variable $S$ we obtain the estimates
				\begin{equation*}
					\mathbb{P}(S = \overline{y}_i t) \mathbb{P}(\lambda_1 \le (x-\overline{y}_i)t) = e^{-t^2 (g \oplus_i \mathcal{T}_t )(x) + O(t)}, \qquad \mathbb{P}(S \le (x-2)t) = e^{-t^2 g(x-2) + O(t)}.
				\end{equation*}
Inserting these estimates back in \eqref{eq:upl+S} and \eqref{eq:dwl+S}, we obtain
\begin{equation*}
					e^{-t^2 (g \oplus_3 \mathcal{T}_t )(x) + O(t)} \le \mathbb{P}(\lambda_1+S\le xt) \le e^{-t^2 (g \oplus_2 \mathcal{T}_t )(x) + O(t)} +  e^{-t^2 g(x-2) + O(t)}.
				\end{equation*}    
   From \Cref{thm:limit_q_Laplace}  we know that $\lambda_1+S$ satisfies a lower-tail LDP with speed $t^2$ and rate function $\mathcal{F}$. Taking the $-\frac{1}{t^2} \log$ of all terms in the previous chain of inequality and letting $t$ tend to $+\infty$, we thus obtain
				\begin{equation}\label{eq:limconv2}
					\limsup_{t\to \infty} \min\left\{ \left( g \oplus_2 \mathcal{T}_t \right)(x), g(x-2) \right\} \le \mathcal{F}(x) \le \liminf_{t\to\infty} \left( g \oplus_3 \mathcal{T}_t \right)(x). 
				\end{equation}
Note that $(g\oplus \mathcal{T}_t)(x)\le (g\oplus_2 \mathcal{T}_t)(x)$. Since $(g\oplus \mathcal{T}_t)(x) \le g(x-2)+\mathcal{T}_t(2)$ and $\mathcal{T}_t(2) \to 0$ as $t$ tends to $\infty$ via \Cref{prop:tt2}, from the first inequality in \eqref{eq:limconv2} we deduce that
\begin{align}\label{eq:limconv3}
    \limsup_{t\to \infty} \left( g \oplus \mathcal{T}_t \right)(x)  \le \mathcal{F}(x).
\end{align}
For $x\le 2$, we claim that
    \begin{align}\label{eq:tiloplus}
        \lim_{t\to \infty} |(g\oplus_3 \mathcal{T}_t)(x)-(g\oplus \mathcal{T}_t)(x)|=0.
    \end{align}
Owing to the second inequality in \eqref{eq:limconv2}, the above claim forces $\mathcal{F}(x) \le \liminf_{t\to\infty} (g\oplus \mathcal{T}_t)(x)$ for $x\le 2$. Combining this with \eqref{eq:limconv3} verifies \eqref{eq:lim_inf_conv} for $x\le 2$. Thus we are left to show \eqref{eq:tiloplus} for $x\le 2$.

Fix any $\varepsilon>0$ and $x\le 2$. Since $g(y)\to \infty$ as $|y|\to \infty$ and $\sup_{t>0, y\ge 0} |\mathcal{T}_t(y)| \le 1-q$, we may find a sequence $\{z_t\}_t$ such that $\sup_t |z_t|<\infty$ and 
$$g(z_t)+\mathcal{T}_t(x-z_t)-(g\oplus \mathcal{T}_t)(x) \le \e.$$ 
Clearly, $z_t\le x$ for all $t$. Let $z$ be any limit point of the sequence $\{z_t\}_t$. We have 
$$g(z_t)+\mathcal{T}_t(x-z_t)-\varepsilon \le (g\oplus \mathcal{T}_t)(x) \le g(x-2)+\mathcal{T}_t(2).$$ 
Taking subsequential limit and using \Cref{prop:tt2} we obtain $g(z) \le g(x-2)$. Since $x<2$, we have $z\ge x-2$. Note that $\mathcal{T}_t(t^{-1}\lfloor tz_t\rfloor)=\mathcal{T}_t(z_t)$. Since all limit points of $\{z_t\}$ are in $[x-2,x]$ for all enough $t$ we can ensure $t^{-1}\lfloor tz_t\rfloor \in \frac1t\Z\cap[x-3,x]$. Thus,
$$(g\oplus_3\mathcal{T}_t)(x)-g(\oplus \mathcal{T}_t)(x) \le |g(z_t)-g(t^{-1}\lfloor tz_t\rfloor)|+\varepsilon.$$
Taking $\limsup_{t\to\infty}$ on both sides and noticing that $\e$ is arbitrary, we arrive at \eqref{eq:tiloplus}. \end{proof}

Armed with the asymptotic midpoint convexity of $\mathcal{T}_t$ from \Cref{prop:midpoint_convexity_T} and the pointwise convergence result from \Cref{prop:f_limit_g+T}, we can now prove certain properties of $\mathcal{F}$. 

			
\begin{theorem}\label{prop:fcon}
	The lower-tail rate function $\mathcal{F}$ is  convex,  continuous on the entire real line and takes finite real values.
\end{theorem}
			
\begin{proof}
	The rate function $\mathcal{F}$ is decreasing and to prove its convexity it is sufficient to show that $\mathcal{F}$ is midpoint convex. For this we take $x,x' \in \mathbb{R}$ and we will show that
	\begin{equation*}
		\mathcal{F}\left( \tfrac{x+x'}{2} \right) \le \tfrac{1}{2}(\mathcal{F}(x)+\mathcal{F}(x')).
	\end{equation*}
	We are going to use the characterization \eqref{eq:lim_inf_conv} of $\mathcal{F}$ as a limit of $g \oplus \mathcal{T}_t$. By the definition of infimal convolution, fixed $x,x'$ and an arbitrary small number $\varepsilon>0$ we can find $u,v,u',v'$ such that $u+v=x$, $u'+v'=x'$ and
	\begin{equation*}
		g(u)+\mathcal{T}_t(v) \le \left(g \oplus \mathcal{T}_t \right) (x) + \varepsilon,
		\qquad
		g(u')+\mathcal{T}_t(v') \le \left(g \oplus \mathcal{T}_t \right) (x') + \varepsilon.
	\end{equation*}
	Let $u'' = (u+u')/2$ and $v''=(v+v')/2$. By \Cref{prop:midpoint_convexity_T} we get, for $t$ large enough 
	\begin{equation*}
		\mathcal{T}_t(v''+\tfrac1t) \le \tfrac{1}{2} \left( \mathcal{T}_t(v) + \mathcal{T}_t(v') \right) + \varepsilon.
	\end{equation*}
	Clearly we also have $g(u'') \le \frac{1}{2}(g(u) + g(u'))$ and noticing that $u''+v'' = \frac{1}{2}(x+x')$ we have
	\begin{equation*}
		\begin{split}
			\left( g \oplus \mathcal{T}_t \right)  \left(\frac{x+x'}{2} \right) &\le g(u''-\tfrac1t) + \mathcal{T}_t(v''+\tfrac1t) 
			\\ & = g(u''-\tfrac1t)-g(u'')+g(u'') + \mathcal{T}_t(v''+\tfrac1t)
			\\
			&
			\le g(u''-\tfrac1t)-g(u'')+\frac{1}{2}(g(u) + g(u')) + \frac{1}{2} \left( \mathcal{T}_t(v) + \mathcal{T}_t(v') \right) + \varepsilon
			\\
			&
			\le g(u''-\tfrac1t)-g(u'')+\frac{1}{2} \bigg( \left(g \oplus \mathcal{T}_t \right) (x) + \varepsilon + \left(g \oplus \mathcal{T}_t \right) (x') + \varepsilon \bigg) + 2\varepsilon.
		\end{split}
	\end{equation*}
	Now, using \eqref{eq:lim_inf_conv} and the fact that $g$ is continuous we obtain that
    $$\mathcal{F}\left(\tfrac{x+x'}{2}\right) \le \tfrac12(\mathcal{F}(x)+\mathcal{F}(x'))+3\e.$$
    Since $\varepsilon$ is arbitrary, this proves the midpoint convexity of $\mathcal{F}$. In addition to convexity, since  $\mathcal{F}$ is non-negative, decreasing, and has parabolic behavior for $x\le x_q$ via \Cref{prop:fprop}, it follows that $\mathcal{F}(x)\in \R$ for all $x\in \R$. Finally, it is well known that a real-valued convex function is continuous.
			\end{proof}

    \subsection{Proof of \Cref{thm:lower_tail_intro} and \Cref{prop:fall}}\label{sec:4.5} In this section we prove \Cref{thm:lower_tail_intro}. In \Cref{thm:limit_q_Laplace} we derived the lower-tail rate function of $\h(0,t)+\chi+S$ and in \Cref{prop:expapprox} we obtain sharp estimates for the lower-tail probability of $\h(0,t)+\chi\stackrel{d}{=}\lambda_1$ in terms of $\mathcal{T}_t$. We would first like to show that $\mathcal{T}_t$ converges pointwise. This will enable us to show the existence of the lower-tail rate function of $\h(0,t)+\chi$. Using a deconvolution lemma (\Cref{l:9.1}), we can then obtain the lower-tail rate function of $\h(0,t)$.

    To show the convergence of $\mathcal{T}_t$, the following equicontinuity-type result is crucial.

    \begin{proposition}\label{prop:equicont} For any $\varepsilon>0$, there exists $\delta>0$ and $t_\varepsilon>0$ such that 
    $$
    |\mathcal{T}_t(x)-\mathcal{T}_t(y)|\le \varepsilon,
    $$ 
    for all $t\ge t_\varepsilon$ and for all $x,y\in[0,2]$ with $|x-y|\le \delta$.
    \end{proposition}

\begin{proof} 
    Let $x<y$ and assume $y-x=\delta$. By the midpoint convexity stated in \Cref{prop:midpoint_convexity_T}, for any fixed $\varepsilon'$, there exists $t_{\varepsilon'}$ such that
   \begin{equation*}
        2 \mathcal{T}_t(x) \le \mathcal{T}_t(x-\delta-\tfrac1t) + \mathcal{T}_t(x+\delta) + \varepsilon',
    \end{equation*}
    for any $t>t_{\varepsilon'}$, which implies
    \begin{equation} \label{eq:T_t_continuity_1}
        \mathcal{T}_t(x) - \mathcal{T}_t(x+\delta) \le \mathcal{T}_t\big(x-\delta-\tfrac1t\big) - \mathcal{T}_t(x) + \varepsilon'. 
    \end{equation}
    Consider a non-negative integer $k$ such that
    \begin{equation*}
        k\delta+\frac{k(k+1)}{2t} \le x < (k+1)\delta+\frac{(k+1)(k+2)}{2t}.
    \end{equation*}
    Then, iterating \eqref{eq:T_t_continuity_1} we obtain
    \begin{equation} \label{eq:T_t_continuity_2}
        \begin{split}
            \mathcal{T}_t(x)-\mathcal{T}_t(y) &\le \mathcal{T}_t\left(x-k\delta-\frac{k(k+1)}{2t}\right)-\mathcal{T}_t\left(x-(k-1)\delta-\frac{k(k-1)}{2t}\right) + k\varepsilon'
            \\
            &\le \mathcal{T}_t(0) - \mathcal{T}_t(2\delta+\tfrac{2k+1}{t}) + k\varepsilon'.
        \end{split}
    \end{equation}
    Next, we estimate the term $\mathcal{T}_t(0) - \mathcal{T}_t\big(2\delta+\frac{2k+1}{t}\big)$. Combining \eqref{eq:f_parabola}, \Cref{proptt0} and \Cref{prop:f_limit_g+T}, we know that there exists $x_q$ such that 
    \begin{equation*}
    \begin{split}
        \mathcal{F}(x_q) = \mathcal{T}_t(0) + g(x_q) = \lim_{t \to +\infty} \inf_{y\in [0,2]} \left\{ \mathcal{T}_t(y) + g(x_q-y) \right\},
    \end{split}
    \end{equation*}
    where $g(y)= \logq y^2/2 $.
    This implies that for any fixed $\varepsilon''$ we can pick $t_{\varepsilon''}$ such that
    \begin{equation*}
        -\varepsilon'' \le \inf_{y\in [0,2]} \left\{ \mathcal{T}_t(y) + g(x_q-y) \right\} - \mathcal{T}_t(0) - g(x_q) \le \varepsilon''
    \end{equation*}
     for all $t>t_{\varepsilon''}$. Then, for any $y\in[0,2]$ we have
    \begin{equation}\label{eq:T_t_continuity_3}
        0\le \mathcal{T}_t(0)- \mathcal{T}_t(y) \le \varepsilon'' + g(x_q-y)-g(x_q) = \varepsilon'' + y( y-2x_q ) \le \varepsilon'' + M y,
    \end{equation}
    where $M=-2x_q+2$. Combining the estimates \eqref{eq:T_t_continuity_2}, \eqref{eq:T_t_continuity_3} we arrive at the bound
    \begin{equation} \label{eq:T_t_continuity_4}
       0\le \mathcal{T}_t(x)-\mathcal{T}_t(y) \le \varepsilon'' + 2M\delta+\frac{M(2k+1)}{t} + k \varepsilon',
    \end{equation}
    which holds for any $t>\max\{ t_{\varepsilon'}, t_{\varepsilon''} \}$. It is now clear that the right-hand side of \eqref{eq:T_t_continuity_4} can be made arbitrarily small, since $k<2/\delta$ and $\varepsilon', \varepsilon''$ are independent of $\delta$. Moreover, we can also allow $|x-y|<\delta$ using the fact that $\mathcal{T}_t$ is decreasing. This completes the proof. 
\end{proof}

 We now state two real analysis lemmas that will allow us to deduce the lower-tail rate functions first for $\h(0,t)+\chi$ and then for $\h(0,t)$. Their proofs are deferred to the supplement material \cite{supp}.

    \begin{lemma} \label{lem:deconv1}
        Let $h_n:\mathbb{R} \to [0,\infty]$ be a family of decreasing functions such that
        \begin{itemize}[leftmargin=20pt]
            \item $h_n(x)=+\infty$ for $x<0$, and there exists $M>$ such that $h_n(x)\in[0,M]$ for $x\ge 0$ and $\sup_{x\ge 2} h_n(x) \to 0$ as $n\to \infty$.
            \item For all $\varepsilon>0$, there exists $\delta>0$ and $n_\e>0$ such that for all $n\ge n_\e$ and for all $x,y \in [0,2]$ with $|x-y|\le \delta$  we have $$|h_n(x)-h_n(y)|\le \varepsilon.$$
            \item Every subsequential limit of $\{h_n\}$ is convex.
        \end{itemize}
          Let $g(x)=\frac{x^2}{2}\logq$. Assume that $(h_n \oplus g)(x)$ converges pointwise to a proper, lower-semicontinuous convex function $f(x)$.  Then $h_n(x)$ converges pointwise to  $$h(x)=(f\ominus g)(x) := \sup_{y\in\R} \{ f(x-y) - g(y) \}.$$ Moreover we have $f=g\oplus h$,  the function $h$ is continuous on $[0,\infty)$, and the function $f$ is differentiable with derivative $f'$ being $\logq$-Lipschitz.
    \end{lemma}

    \begin{lemma}\label{l:9.1} Let $G:\R\to [0,1]$ be a decreasing function with the property that
					$$\lim_{x\to \infty}\frac1{x^2}\log G(-x)=0,\qquad \lim_{x\to \infty}\frac1{x^2}\log G(x)= -\infty.$$
					Fix an open set $O\in \R$. Suppose  $\mathpzc{g}: O\to \R$ is a continuous function. Let $\{X_t\}_{t\ge 1}$ be a sequence of random variables satisfying
					\begin{align}\label{e:9.1c}
						\lim_{t\to \infty} \frac1{t^2}\log\Ex\big[G(X_t-xt)\big]=\mathpzc{g}(x)
					\end{align}
					for all $x\in O$. Then for all $x \in O$ we have
					\begin{align*}
						\lim_{t\to \infty} \frac1{t^2}\log\Pr\big(X_t \le xt\big)=\mathpzc{g}(x).
					\end{align*}
				\end{lemma}
    \begin{proof}[Proof of \Cref{thm:lower_tail_intro}] Fix $\mu\in (0,2)$. Recall $\mathcal{T}_t$ from \eqref{def:tta}. We would like to apply \Cref{lem:deconv1} with $h_t = \mathcal{T}_t$ to deduce that $\lim_{t\to\infty} \mathcal{T}_t(\mu)$ exists. Note that $\{\mathcal{T}_t\}$ satisfies the three conditions of \Cref{lem:deconv1}. Indeed, the first condition follows from \Cref{proptt0} and \Cref{prop:tt2}, whereas the second one follows from \Cref{prop:equicont}. The third one is a consequence of \Cref{prop:midpoint_convexity_T}. Since by \Cref{prop:f_limit_g+T} we have $g\oplus\mathcal{T}_t\to \mathcal{F}$ pointwise and $\mathcal{F}$ is proper, lower semicontinuous and convex by \Cref{prop:fcon}, we thus have that
    \begin{align}
        \label{e:des}
        \mathcal{T}_t(\mu)\xrightarrow[t\to\infty]{} \Phi_-(\mu):=\sup_{y\in \R} \{\mathcal{F}(y)-g(\mu-y)\}
    \end{align}
and $\Phi_-$ is continuous on $[0,\infty)$. Due to the properties of $\mathcal{T}_t$ established in  \Cref{proptt0,prop:tt2,prop:midpoint_convexity_T}, we readily have that $\Phi_-$ is decreasing, non-negative and convex with $\Phi_-(\mu)=+\infty$ for $\mu<0$, $\Phi_-(0)=1-q$, and $\Phi_-(\mu)=0$ for $\mu\ge 2$. But by the lower-tail estimate of $\lambda_1$ in \eqref{eq:exp_approximation_lower_tail_lambda1} we have thus shown that
    \begin{align*}
         -\lim_{t\to+\infty}\frac1{t^2}\log \Pr_{\mathsf{cPlan}(t(1-q))}(\lambda_1\le \mu t)=\Phi_-(\mu).
    \end{align*}
    Recall that $\h(0,t)+\chi\stackrel{d}{=}\lambda_1$ from \Cref{thm:matching_qPNG_cylindric_plancherel}. To remove $\chi$, we appeal to \Cref{l:9.1}.  For every fixed $q\in (0,1)$ consider the function $G_q:\R\to [0,1]$ defined as
					\begin{align*}
     G_q(y):=\Pr(\chi \le -y),
					\end{align*}
					where $\chi\sim q$-$\mathrm{Geo}(q)$. We claim  that $G_q$ satisfies the conditions in  \Cref{l:9.1}. Clearly, $G_q$ is decreasing. Since $G_q(y)=0$ for all $y>0$, the right tail condition in \Cref{l:9.1} is satisfied trivially. Since $\lim_{x\to \infty} G_q(-x)=1$, the left tail condition also holds. Thus $G_q$ satisfies the conditions in \Cref{l:9.1}. The utility of $G_q$ function is that it allows us to write
					\begin{align*}
						\Pr(\h(0,t)+\chi \le \mu t)=\Ex[G_q(\h(0,t)-\mu t)].
					\end{align*}
				Since $\Phi_-$ is continuous on $[0,\infty)$,	invoking \Cref{l:9.1} with $G=G_q$ and $\mathpzc{g}=\Phi_-$, we see that $\h(0,t)$ satisfies a lower-tail LDP with the same rate function $\Phi_-$ and speed $t^2$. This completes the proof. 
    \end{proof}

\begin{proof}[Proof of \Cref{prop:fall}] A few of the properties of $\mathcal{F}$ are already proven in \Cref{prop:fprop} and \Cref{prop:fcon}. The remaining properties of $\mathcal{F}$ claimed in \Cref{prop:fall} follow from \Cref{lem:deconv1}. 
\end{proof}
   
\section{Further approaches to the lower-tail} \label{sec:further_approaches}

In this section, we outline two possible approaches to the explicit characterization of the lower-tail rate function.

\subsection{Non rigorous approach to characterization of $\mathcal{F}$: nonlinear differential equations} \label{subs:eq_diff}

In \cite{cafasso_ruzza_2022}, starting from a discrete Riemann-Hilbert characterization of the Fredholm determinant \eqref{eq:fred}, the authors derived a  differential-difference equation for the quantity
\begin{equation*}
    Q(t,s) := \mathbb{P}_{\mathsf{cPlan}(t(1-q))} (\lambda_1+S \le s),
\end{equation*}
which reads
\begin{equation} \label{eq:discrete_Toda}
    \partial_t^2 \log Q(t,s) + \frac{1}{t} \partial_t \log Q(t,s) + 4 = 4 \frac{Q(t,s+1) Q(t,s-1)}{Q(t,s)^2}.
\end{equation}
When $s$ is of order $t$, from the Large Deviation Principle \eqref{eq:limit_q_Laplace} we have approximation
\begin{equation} \label{eq:rate_function_f}
    \log Q_\sigma(t,s) = -t^2 \mathcal{F}(s/t) + o(t^2),
\end{equation}
where we recall that the function $\mathcal{F}$ is given by the variational problem \eqref{eq:f_Plancherel}. Plugging this approximation in \eqref{eq:discrete_Toda}, assuming that the function $\mathcal{F}$ is twice differentiable (which we did not prove), we find a closed equation for $\mathcal{F}$ as 
\begin{equation} \label{eq:differental_equation}
    -4 \mathcal{F} (x) + 3 x \mathcal{F}'(x) - x^2 \mathcal{F}''(x) + 4 = 4 e^{-\mathcal{F}''(x)},
\end{equation}
where $x=s/t$. 

Motivated by the fluctuation result \eqref{eq:Tracy_Widom} we expect that at $x=2$, the behavior of $\mathcal{F}$ is given by $\mathcal{F}(2)=\mathcal{F}'(2)=\mathcal{F}''(2)=0$, since $F_{\operatorname{GUE}}(s) \sim e^{-\frac{|s|^3}{12}}$ for $s \to -\infty$. A quick check shows that these conditions also determine $\mathcal{F}'''(2)$. In fact deriving twice both sides of \eqref{eq:differental_equation} we get
\begin{equation*}
    -x  \mathcal{F}''' (x )+\mathcal{F}''''(x ) \left(4 e^{-\mathcal{F}''(x )}-x ^2\right)-4 \mathcal{F}'''(x )^2 e^{-\mathcal{F}''(x )} = 0,
\end{equation*}
which evaluated at $x=2$ and assuming $\mathcal{F}(2)=\mathcal{F}'(2)=\mathcal{F}''(2)=0$ leaves us with
$\mathcal{F}''' (2 )(1+2 \mathcal{F}'''(2 )) = 0.$
This forces, once we exclude constant solutions, $\mathcal{F}'''(2)=-\frac{1}{2}$. These considerations motivate imposing the boundary conditions 
\begin{equation*}
    \mathcal{F}(2)=\mathcal{F}'(2)=\mathcal{F}''(2)=0 \qquad \text{and} \qquad \mathcal{F}'''(2)=- \tfrac{1}{2}.
\end{equation*}
In general, these boundary conditions do not guarantee uniqueness. We look for a one-parameter family of solutions to this problem of the following particular form. Let us write
\begin{equation*}
    \mathcal{F}(x) = \int_2^x \diff y \int_2^y \diff z \log \mathcal{G}(z),
\end{equation*}
in which case $\mathcal{G}(x)$, which needs to be positive for $x>0$, solves the differential equation
\begin{equation*}
    \mathcal{G} \left(4 \mathcal{G}''+x ^2 (\mathcal{G}')^2\right)=8 (\mathcal{G}')^2+x  \mathcal{G}^2 \left(x  \mathcal{G}''+\mathcal{G}'\right)
    \end{equation*}
with boundary conditions $\mathcal{G}(2)=1,  \mathcal{G}'(2)=-\tfrac{1}{2}.$
In the following proposition, we introduce a one-parameter family of solutions $\mathcal{G}$.
\begin{proposition}
For any $c\in (0,+\infty]$ and $x \in (0,2)$ let $\mathcal{G}=\mathcal{G}_c(x)>1$ be solution of the equation
\begin{equation} \label{eq:g_implicit}
        \log (\mathcal{G})= 
        \begin{cases}
        \frac{2c}{\sqrt{c^2-4}} \mathrm{arctanh} \left(\frac{c x \sqrt{\left(c^2-4\right) \mathcal{G} }-\sqrt{\left(c^2-4\right)^2 \mathcal{G} x^2+16 \left(c^2-4\right)}}{\left(c^2-4\right) x \sqrt{\mathcal{G}}-\sqrt{\left(c^2-4\right) c^2 \mathcal{G} x ^2+16 c^2}}\right)
        \qquad &\text{if } c \neq 2,+\infty
        \\
        2-x \sqrt{\mathcal{G}} \qquad &\text{if } c = 2
        \\
        \log \left(\frac{4}{\mathcal{G} x ^2}\right) \qquad &\text{if } c = +\infty,
        \end{cases}
    \end{equation}
    where $\sqrt{\,\,\,\,}$ is the principal branch of the square root.
    Then, $\mathcal{G}_c(x)$ solves the system
    \begin{equation} \label{eq:diff_system_G}
        \begin{cases}
            \mathcal{G} \left(4 \mathcal{G}''+x ^2 (\mathcal{G}')^2\right)=8 (\mathcal{G}')^2+x  \mathcal{G}^2 \left(x  \mathcal{G}''+\mathcal{G}'\right) \qquad x \in (0,2)
            \\
            \mathcal{G}(2)=1, \qquad \mathcal{G}'(2)=-\frac{1}{2}, \qquad \mathcal{G}''(2)=\frac{1}{2}-\frac{1}{4 c^2}.
        \end{cases}
    \end{equation}
\end{proposition}
\begin{proof}
    Using the implicit function theorem we can verify that the implicitly defined function $\mathcal{G}_c$ is in fact the solution of the desired differential system.
\end{proof}
 
\begin{remark}
    When $0<c<+\infty$, the function $\mathcal{G}$ of \eqref{eq:g_implicit} is defined on a larger interval $(\widetilde{x}_c,0)$ for $\widetilde{x}_c<0$ and hence it solves the differential system \eqref{eq:diff_system_G} on a larger interval. 
\end{remark}

Leveraging these explicit solutions for the differential system for the function $\mathcal{G}$ we are able to obtain solutions of the system \eqref{eq:differental_equation} with the condition on the fourth derivative of $\mathcal{F}$.

\begin{corollary}
    Fix $c > 0$ or $c=+\infty$ and consider the function $\mathcal{G}_c$ defined in \Cref{eq:g_implicit}. Then, the function 
    \begin{equation} \label{eq:f_c}
        \mathcal{F}_c(x) := \int_2^x \diff y \int_2^y \diff z \log \mathcal{G}_c(z),
    \end{equation}
solves the differential system
\begin{equation}
    \begin{cases} \label{eq:differential_system}
    4 \mathcal{F} (x) - 3 x \mathcal{F}'(x) + x^2 \mathcal{F}''(x) - 4 + 4 e^{-\mathcal{F}''(x)} =0
    \qquad x \in (0,2),
    \\
    \mathcal{F}(2)=\mathcal{F}'(2)=\mathcal{F}''(2) = 0, \qquad \mathcal{F}'''(2) = -\frac{1}{2},\qquad \mathcal{F}''''(2)= \frac{1}{4}-\frac{1}{4 c^2} .
    \end{cases}
\end{equation}
\end{corollary}

\begin{remark}
    We can check that the function $f(x)=\frac{\logq}{2} x^2 + (1-q)$, which by virtue of \eqref{eq:f_parabola} coincides with the lower-tail rate function $\mathcal{F}(x)$ for $x\ll 0$ is in fact a solution of the differential equation \eqref{eq:differental_equation}.
\end{remark}

\begin{remark}
    In the particular case where $c=+\infty$, the function $\mathcal{G}_{\infty}$ is explicit and given by $\mathcal{G}_{\infty}(x) = 2/x$. Then, integrating twice as in \eqref{eq:f_c} we obtain
    \begin{equation*}
        \mathcal{F}_{\infty}(x) = 1-2 x + \frac{3 x^2}{4}+\frac{1}{2} x^2 \log \left(\frac{2}{x}\right),
    \end{equation*}
    which coincides with the lower-tail rate function of the PNG \cite{logan_shepp1977variational,seppalainen_98_increasing,deuschel_zeitouni_1999}.
\end{remark}

\begin{remark}
    In the study of the lower-tail rate function of the KPZ equation, an analogous idea to that discussed in this subsection was developed by Le Doussal in \cite{Le_Doussal_KP_large_deviations}. The probability distribution of the KPZ equation with droplet initial condition obeys a variant of the Kadomtsev–Petviashvili equation, as found in \cite{quastel_remenik_2022_KP}, which when scaled as in \eqref{eq:rate_function_f} gives rise to a first order non-linear differential equation whose solution matches the lower-tail rate function mathematically derived in \cite{tsai_lower_tail,cafasso_claeys_KPZ}.
\end{remark}

We end this subsection considering the case where $c=2$ in the relation \eqref{eq:g_implicit}. Here the function $\mathcal{G}$ can be written in terms of Lambert's $W$ function as
    \begin{equation*}
        \mathcal{G}(x) = \frac{4 }{x^2} W\left(\frac{e x }{2}\right)^2,
    \end{equation*}
    so that, integrating twice the logarithm of $\mathcal{G}$ we obtain the explicit expression
    \begin{equation*}
        1- x^2 W\left(\frac{e x}{2}\right)-\frac{6 x^2}{4 W\left(\frac{e x}{2}\right)}-\frac{x^2}{4 W\left(\frac{e x}{2}\right)^2}+\frac{5}{2} x^2.
    \end{equation*}
    As a function of $x$ the above expression can be connected in $C^1$ manner to a parabola of the form $1-q + \frac{\logq}{2} x^2$ as
    \begin{equation}\label{eq:F_2}
        \mathcal{F}_2(x)=\begin{cases}
            1-q_2 +\frac{\log q_2^{-1}}{2} x^2 \qquad &x < x_{q_2}
            \\
            1- x^2 W\left(\frac{e x}{2}\right)-\frac{6 x^2}{4 W\left(\frac{e x}{2}\right)}-\frac{x^2}{4 W\left(\frac{e x}{2}\right)^2}+\frac{5}{2} x^2
            \qquad & x_{q_2} \le x \le 2
            \\
            0 \qquad &x >  2,
        \end{cases}
    \end{equation}    
    where values $q_2,x_{q_2}$ can be found numerically and they are $q_2 \approx 0.00003724$, $x_{q_2} \approx -0.1867$. A plot of the function $\mathcal{F}_2$ is given in \Cref{fig:F_2}. We conjecture this behavior to be general, i.e. that the lower-tail rate function $\mathcal{F}$, for any $q\in(0,1)$, possesses three behaviors: one where it is identically 0 for $x>2$, one where it assumes the form prescribed by \eqref{eq:f_c} for $x \in [x_q,2]$ and finally, for $x<x_q$ it is equal to $1-q + \frac{\logq}{2} x^2$. {The values of $c,x_q$ are determined by $q$ and they are unique if we require that $\mathcal{F}$ is $C^1$. At the moment we cannot prove this statement as we are not able to prove that the rate function $\Phi_-$ is $C^2$ in its domain of definition.}
    \begin{figure}
        \centering
        \includegraphics[scale=.7]{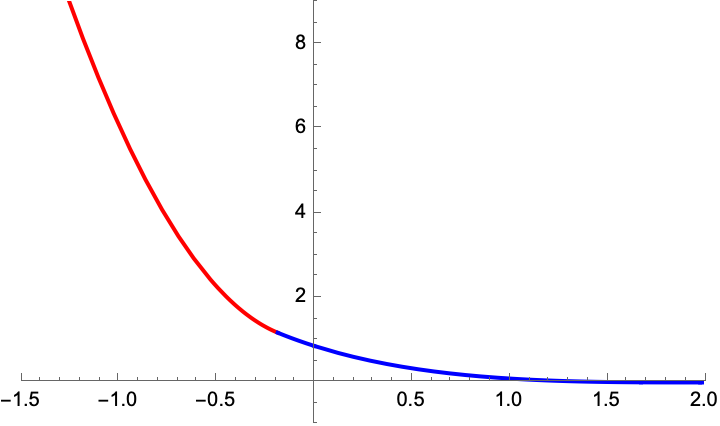}
        \caption{A plot of the piecewise defined function $\mathcal{F}_2(x)$ of \eqref{eq:F_2}. The red curve corresponds to the parabolic region of the curve, while the blue curve represents the region where $x_{q_2} \le x \le 2$.}
        \label{fig:F_2}
    \end{figure}

\subsection{Potential theoretic approach to the minimizer of $\mathcal{W}^{(q)}$} \label{subsub:logan_shepp_generalization}

We propose another possible way of finding the exact form of the lower-tail rate function $\mathcal{F}(x)$, by solving an energy minimization problem, closely related to the variational characterization \eqref{eq:f_Plancherel}. 
Recall the functional $\mathcal{W}^{(q)}(\kappa,\phi;x)$ defined in \eqref{eq:W_functional}. Consider first a ``de-Poissonized" variant of the functional $\mathcal{W}^{(q)}$ and the associate variational problem as
\begin{equation} \label{eq:W_depoissonized}
    \mathcal{W}^{(q)}(\phi;x) := 1+ 2 I_{\mathrm{hook}} (\phi) + \mathcal{V}^{(q)}(x;\phi),\quad \mathcal{U}(x):= \inf_{\phi} \,\{\mathcal{W}^{(q)}(\phi;x)\}.
\end{equation}
The function $\mathcal{F}(x)$ can be deduced from $\mathcal{U}(x)$ through a further minimization problem:
\begin{equation*}
    \mathcal{F}(x)=\inf_{\kappa>0} \{\kappa\,\mathcal{U}(x/\sqrt{\kappa})+\kappa\log \kappa +1-\kappa\}.
\end{equation*}
A promising observation is that the functional $\mathcal{W}^{(q)}(\phi;x)$ takes a special form, known as the logarithmic energy (associated with certain external fields). To elaborate on this, 
first, we use the following alternative representation due to Logan and Shepp \cite[eqs. (2.9)-(2.11)]{logan_shepp1977variational} of the hook functional $I_\mathrm{hook}$:
\begin{equation} \label{eq:logan_shepp_expression}
    2 I_\mathrm{hook}(\phi) = \log 2 - \frac{1}{2} \int_\mathbb{R} \int_\mathbb{R} \log| s-t | \overline{\phi}'(s) \overline{\phi}'(t) \diff s \diff t - 2 \int_\mathbb{R} \overline{\phi}'(t) (t \log|t|-t) \diff t.
\end{equation}
Recall that the notation $\overline{\phi}$ was defined in \eqref{eq:phi_bar}. On the other hand, a simple integration by parts implies that the functional $\mathcal{V}^{(q)}$ can be expressed as
\begin{equation} \label{eq:V_derivative}
    \mathcal{V}^{(q)}(\phi;x) = - \logq \, \int_\mathbb{R} [t-x]_+ \overline{\phi}'(t) \diff t.
\end{equation}
Combining \eqref{eq:logan_shepp_expression} \eqref{eq:V_derivative} with \eqref{eq:W_depoissonized} we recast the problem of minimizing the $q$-deformed hook functional $W$ as that of minimizing the functional $\mathsf{J}_{\eta,x} : \overline{\mathcal{Y}}_1 \longrightarrow \mathbb{R}$ defined by
\begin{equation} \label{eq:J_potential}
    \mathsf{J}_{\eta,x}(h) = - \tfrac{1}{2} \int_\mathbb{R} \int_\mathbb{R} \hspace{-0.1cm}\log| s-t | h'(s) h'(t) \diff s \diff t - 2 \hspace{-0.1cm}\int_\mathbb{R} \hspace{-0.2cm} h'(t) \left(t \log|t|-t + \eta [t-x]_+ \right) \diff t,
\end{equation}
with $\eta>0$ and $x\in \mathbb{R}$. The functional $\mathsf{J}_{\eta,x}(h)$ is, up to a scaling factor $\frac{1}{2}$, the logarithmic energy \cite{saff1997logarithmic} associated to the measure $\mathrm{d}h:=h'(t)\mathrm{d}t$, with the external field given by 
\begin{equation*}
    V_{\mathrm{ext}}(t;\eta,x):= -4(t\log|t|-t+\eta[t-x]_{+}).
\end{equation*}
Finding the minimizer (known as the equilibrium measure) of certain logarithmic energy functionals is a well-studied problem in potential theory literature and has
 applications to random matrix theory and many other mathematical physics problems. One has the following standard necessary condition for a function $h_0\in \overline{\mathcal{Y}}_1$ to be the minimizer (see e.g. \cite[Theorem 1.16]{romik_2015} and \cite{logan_shepp1977variational}): Defining
\begin{equation*}
    p(t) := -\int_\mathbb{R} h_0'(s) \log|s-t| \diff s - 2(t \log|t| - t + \eta[t-x]_+) +\lambda t,
\end{equation*}
where the parameter $\lambda\in \mathbb{R}$ is the Lagrange multiplier. Then for a minimizer $h_0$, the function $p$ must satisfy the following Euler-Lagrange type equations:
\begin{equation*}
    p(u) \, \text{is} \, 
    \begin{cases}
         =0, \qquad &\text{if }  \mathrm{sign}(u) \cdot h_0'(u) \in (-2,0) 
        \\
        \le 0, \qquad & \text{if } h_0'(u) + \mathrm{sign}(u) = 1
        \\
        \ge 0, \qquad & \text{if } h_0'(u) + \mathrm{sign}(u) = -1.
    \end{cases}
\end{equation*}
Taking a (weak) derivative of $p$ we see that on $\{u: \mathrm{sign}(u)h_0'(u)\in (-2,0)\}$, one must have 
\begin{equation}\label{eq:Cauchy_integral_eq_1}
    \frac{1}{\pi}\,\mathrm{P.V.}\int \frac{h_0'(y)}{y-u} \diff y = -\frac{2}{\pi} \log|u| - \frac{2 \eta}{\pi} \mathbf{1}_{[x,+\infty)}(u) -  \frac{\lambda}{\pi},
\end{equation}
where P.V. stands for the Cauchy principal value. We are not able to solve the Cauchy integral equation \eqref{eq:Cauchy_integral_eq_1} explicitly. The main difficulty is due to the fact that we do not know a priori the form of the support of $h_0'$. An ansatz of the form $\mathrm{supp}(h_0')=\cup_{i=1}^k (a_i,b_i)$ leads to highly transcendental relations on the endpoints which are not straightforward to solve (see \cite{zinn2000sixvertex,bleher2011uniform} for solutions to simpler versions of the above problem). We hope to explore this direction in a future work.

We end this subsection by remarking that solving certain energy-minimization problems similar to that of \eqref{eq:J_potential} is also the starting point for a nonlinear steepest descent analysis of the Riemann-Hilbert problem (as mentioned earlier, the relevant RHP for $q$-PNG was introduced in \cite{cafasso_ruzza_2022}). Indeed, such analysis usually starts with finding a so-called $g$-function as a conjugating factor such that the transformed RHP behaves properly at $\infty$. The construction of such $g$-functions is a potential-theoretic problem that we expect to be rather involved for the $q$-PNG case. 

Riemann-Hilbert methods have long been used as a powerful tool for studying the tail behaviors of Fredholm determinants, especially in the oscillating regimes, see e.g.~\cite{baik2001optimal,deift2011asymptotics}. More recently, a similar but more involved analysis has been implemented for finite-temperature models, see \cite{cafasso_claeys_KPZ,cafasso2021airy,charlier2022uniform}. The discrete nature of the RHP associated with $q$-PNG seems to lead to additional technical difficulty for a suitable nonlinear steepest descent analysis.

\begin{acks}[Acknowledgments] We thank Ivan Corwin. Promit Ghosal, and Guilherme Silva for their feedback on an earlier draft of the paper. MM thanks Mattia Cafasso and Giulio Ruzza for several comments about the results of this paper and Vadim Gorin for suggesting references about the problem of equilibrium measure with singular potentials. We thank the anonymous referees for
their careful reading and useful comments on improving our manuscript.
\end{acks}

\begin{funding}
  SD's research was partially supported by Ivan Corwin's NSF grant DMS-1811143, the Fernholz Foundation's ``Summer Minerva Fellows'' program, and also the W.M. Keck Foundation Science and Engineering Grant on ``Extreme diffusion''.  The work of MM has been supported by the European Union’s Horizon 2020 research and innovation programme under the Marie Skłodowska-Curie grant agreement No. 101030938. The work of YL was supported by the EPSRC grant EP/R024456/1.\end{funding}

\begin{supplement}

\stitle{Supplement to ``Large deviations for the $\lowercase{q}$-deformed polynuclear growth''.}
\sdescription{The supplementary material contains the proofs of the two lemmas, \Cref{lem:deconv1,l:9.1}, that are skipped in \Cref{sec:4.5}.}

\section{Technical lemmas}

We first remind the readers of the notion of proper and closed functions.

\begin{definition}[Proper and closed functions]\label{def:pc} 
	We call a  function $f:\R\to [-\infty,+\infty]$ proper if $f(x)>-\infty$ for all $x\in \R$ and $f(x)<\infty$ for some $x\in \R$. We call  $f:\R\to [-\infty,+\infty]$ to be closed  if $\{x :f(x)\le \alpha\}$ is a closed set for all $\alpha\in \R$.
\end{definition}

For convenience, we recall the statement of the lemmas here.

\begin{lemma}[Lemma 4.17 in the main text] \label{lem:deconv1}
	Let $h_n:\mathbb{R} \to [0,\infty]$ be a family of weakly decreasing functions such that
	\begin{itemize}[leftmargin=20pt]
		\item $h_n(x)=+\infty$ for $x<0$, and there exists $M>$ such that $h_n(x)\in[0,M]$ for $x\ge 0$ and $\sup_{x\ge 2} h_n(x) \to 0$ as $n\to \infty$.
		\item For all $\varepsilon>0$, there exists $\delta>0$ and $n_\e>0$ such that for all $n\ge n_\e$ and for all $x,y \in [0,2]$ with $|x-y|\le \delta$  we have $$|h_n(x)-h_n(y)|\le \varepsilon.$$
		\item Every subsequential limit of $\{h_n\}$ is convex.
	\end{itemize}
	Let $g(x)=\frac{x^2}{2}\logq$. Assume that $(h_n \oplus g)(x)$ converges pointwise to a proper, lower-semicontinuous convex function $f(x)$.  Then $h_n(x)$ converges pointwise to  $$h(x)=(f\ominus g)(x) := \sup_{y\in\R} \{ f(x-y) - g(y) \}.$$ Moreover we have $f=g\oplus h$,  the function $h$ is continuous on $[0,\infty)$, and the function $f$ is differentiable with derivative $f'$ being $\logq$-Lipschitz.
\end{lemma}
\begin{proof} Fix any subsequence $(n_k)_{k\in \N}$. Although $h_n$'s are not given to be continuous, one can use the first two conditions on $\{h_n\}$ and follow the proof of Arzela-Ascoli theorem verbatim to  
	extract a (uniformly) converging subsequence 
	\begin{equation*}
		h_{n_{k_{\ell}}} \xrightarrow[\ell\to +\infty]{} h \ \mbox{ on }[0,2].
	\end{equation*}
	We define $h(x)=+\infty$ for $x<0$ and $h(x)=0$ for $x>2$. Clearly, $h$ is continuous by the equicontinuity type hypothesis on $\{h_n\}_{n\ge 1}$. We claim that $(g\oplus  h_{n_{k_{\ell}}}) \to (g\oplus h)$ pointwise.  
	Fix any $x, v\in \R$. Observe that
	\begin{align*}
		g(v)+h(x-v)=\lim_{\ell \to \infty} g(v)+h_{n_{k_{\ell}}}(x-v) \ge \limsup_{\ell \to \infty} (g\oplus  h_{n_{k_{\ell}}})(x).
	\end{align*} Taking infimum over $v$ above on both sides, we get $(g\oplus h)(x) \ge \limsup_{\ell\to\infty}(g\oplus  h_{n_{k_{\ell}}})(x)$. For the other way, suppose $v_\ell:=\arg\inf g\oplus h_{n_{k_{\ell}}}(x)$. By the conditions on $h_n$, it is clear that $\sup_{\ell} |v_{\ell}|< \infty$. Passing through a subsequence we may assume $v_\ell \to v$ for some $v \in \R$. Then utilizing the uniform convergence for $h_{n_{k_{\ell}}}$ we obtain
	\begin{align*}
		\liminf_{\ell\to\infty} (g\oplus  h_{n_{k_{\ell}}})(x) & =  \liminf_{\ell\to\infty} [g(v_\ell)+h_{n_{k_{\ell}}}(x-v_{\ell})]= g(v)+h(x-v) \ge (g\oplus h)(x).
	\end{align*}
	This verifies our claim. Consequently, we obtain that $g\oplus h=f$. {By \cite[p.57]{borwein2006convex}}, taking Legendre transform of both sides we find that $g^*+h^*=f^*$. Since $g$ is quadratic, we know $g^*(x)\in \R$ for all $x\in \R$. Thus we have $h^*=f^*-g^*$. Since $h$ is continuous on $[0,\infty)$ and $h(x)=+\infty$ for $x<0$, we see that $h$ is a closed function (see \cref{def:pc}). Since $h$ is closed and convex, using the Fenchel biconjugation theorem we get that
	$$h=h^{**}=(f^*-g^*)^*.$$
	Thus $h$ is uniquely determined by $f,g$. Since every subsequence has a further subsequence that converges to the \textit{same} quantity, we have thus shown $h_n$ converges uniformly to $(f^*-g^*)^*$. Since both $f,g$ are proper, lower-semicontinuous, convex functions, by \cite{hiriart1986general} we have
	$(f^*-g^*)^*=f\ominus g$. Finally, since $f=g\oplus h$, the properties of $f$ claimed in the lemma follow from \cite[Theorem 3.24, Proposition 3.56]{attouch}.
\end{proof}

\begin{lemma}[Lemma 4.18 in the main text] \label{l:9.1} Let $G:\R\to [0,1]$ be a decreasing function with the property that
	$$\lim_{x\to \infty}\frac1{x^2}\log G(-x)=0,\qquad \lim_{x\to \infty}\frac1{x^2}\log G(x)= -\infty.$$
	Fix an open set $O\in \R$. Suppose  $\mathpzc{g}: O\to \R$ is a continuous function. Let $\{X_t\}_{t\ge 1}$ be a sequence of random variables satisfying
	\begin{align}\label{e:9.1c}
		\lim_{t\to \infty} \frac1{t^2}\log\Ex\big[G(X_t-xt)\big]=\mathpzc{g}(x)
	\end{align}
	for all $x\in O$. Then for all $x \in O$ we have
	\begin{align*}
		\lim_{t\to \infty} \frac1{t^2}\log\Pr\big(X_t \le xt\big)=\mathpzc{g}(x).
	\end{align*}
\end{lemma}

\begin{proof}[Proof of \cref{l:9.1}] Fix $x\in O$. Choose $\delta$ small enough so that $(x-\delta,x+\delta)\subset O$. Note that for all $y\in \R$ we have
	\begin{align*}
		& G(y+\delta t) \le \ind_{y<0}+G(\delta t), \qquad
		G(y-\delta t) \ge G(-\delta t)\cdot\ind_{y<0}. 
	\end{align*}
	Setting $y=X_t-xt$ in the above relations and taking expectation we get
	\begin{align}\label{e:9.1}
		& \Ex[G(X_t-xt+\delta t)] \le \Pr(X_t<xt)+G(\delta t), \quad
		\Ex[G(X_t-xt-\delta t)] \ge G(-\delta t)\cdot\Pr(X_t<xt). 
	\end{align}
	In the second inequality above, we use the left tail decay of $G$, so that taking $\log$ and dividing by $t^2$, and then taking $t\uparrow \infty$ we get
	\begin{align*}
		\limsup_{t\to \infty} \frac1{t^2}\log\Pr\big(X_t \le xt\big)\le \mathpzc{g}(x+\delta).
	\end{align*}
	Taking $\delta\downarrow 0$ and using the fact $\mathpzc{g}$ is continuous we get
	\begin{align}\label{e:9.1up}
		\limsup_{t\to \infty} \frac1{t^2}\log\Pr\big(X_t \le xt\big)\le \mathpzc{g}(x).
	\end{align}
	Fix $\rho>0$. Using \eqref{e:9.1c} we can ensure from the first inequality in \eqref{e:9.1} that
	\begin{align*}
		\Pr(X_t<xt) \ge \exp(t^2[\mathpzc{g}(x-\delta)-\rho])-G(\delta t).
	\end{align*}
	Due to the right tail decay conditions on $G$, the first term dominates the second term as $t\to \infty$. Thus,
	\begin{align*}
		\liminf_{t\to \infty} \frac1{t^2}\log\Pr\big(X_t \le xt\big)\le \mathpzc{g}(x-\delta)-\rho.
	\end{align*}
	Taking $\delta\downarrow 0$ and $\rho \downarrow 0$, and again using the fact $\mathpzc{g}$ is continuous we get
	\begin{align*}
		\liminf_{t\to \infty} \frac1{t^2}\log\Pr\big(X_t \le xt\big)\le \mathpzc{g}(x).
	\end{align*}
	Combining this with \eqref{e:9.1up} we get the desired limit. 
\end{proof}
\end{supplement}


\bibliographystyle{imsart-number} 
\bibliography{name.bib}       

\end{document}